\documentclass[10pt,a4paper]{article}
\usepackage{amsmath,amssymb,amsthm,graphicx}
\theoremstyle{definition}

\newcommand{\Real}{\mathbb{R}}							
\newcommand{\Wholes}{\mathbb{Z}}							
\newcommand{\abs}[1]{\left\vert#1\right\vert}			
\newcommand{\sref}[1]{(\ref{#1})}                       

\newtheorem{thm}{Theorem}[section]
\newtheorem{cor}[thm]{Corollary}
\newtheorem{lem}[thm]{Lemma}
\newtheorem{prop}[thm]{Proposition}

\numberwithin{equation}{section}

\usepackage{authblk}
\title{\sc Bichromatic travelling waves for lattice Nagumo equations}
\author[1]{Hermen Jan Hupkes \thanks{\tt hhupkes@math.leidenuniv.nl}}
\author[1]{Leonardo Morelli \thanks{corresponding author, \tt l.morelli@math.leidenuniv.nl}}
\author[2]{Petr Stehl\'{\i}k\thanks{\tt pstehlik@kma.zcu.cz}}
\affil[1]{\small Mathematisch Instituut, Universiteit Leiden, P.O. Box 9512, 2300 RA Leiden, The Netherlands}
\affil[2]{\small Department of Mathematics and NTIS, Faculty of Applied Sciences, University of West Bohemia,\authorcr Univerzitn\'\i~8, 306 14 Plze\v{n}\\ Czech Republic}

\begin{document}
\maketitle

\begin{abstract}
We discuss bichromatic (two-color) front solutions
to the bistable Nagumo lattice differential equation.
Such fronts connect the stable spatially homogeneous equilibria with spatially heterogeneous 2-periodic equilibria
and hence are not monotonic like the standard monochromatic fronts.
We provide explicit criteria that can determine whether or not these fronts are stationary
and show that the bichromatic fronts can travel in parameter regimes
where the monochromatic fronts are pinned.
The presence of these bichromatic waves allows the two stable homogeneous equlibria to both spread out through the
spatial domain towards each other, buffered by a shrinking intermediate zone in which the periodic pattern is visible.




\end{abstract}


\section{Introduction}

In this paper we consider the Nagumo lattice differential equation (LDE)
\begin{equation}\label{eq:int:nagumo:lde}
\dot{u}_j(t)=d \big[ u_{j-1}(t) -2u_j(t) + u_{j+1}(t)  \big] + g\big(u_j(t); a\big),
\end{equation}
posed on the spatial lattice $j \in \mathbb{Z}$, with $t \in \Real$. 
We assume $d > 0$ and use
the standard cubic bistable nonlinearity
$g(u;a) = u(1 -u)(u -a)$ with $a \in (0, 1)$.
This LDE is well-known as a prototypical model that describes the competition
between two stable states $u = 0$ and $u = 1$ in a discrete spatial
environment.
A crucial role is reserved
for so-called travelling front solutions, which have the form
\begin{equation}
\label{eq:int:discrete:wave}
u_j(t) = \Phi(j- ct), \qquad \qquad \Phi(-\infty) = 0, \qquad \Phi(+\infty) = 1 .
\end{equation}
Such solutions are often referred to as invasion waves, as they provide a mechanism by which the energetically preferred state can
invade the spatial domain.

Our work focuses on the case where $c = 0$ holds for these primary invasion waves,
indicating a delicate balance between the two competing states.
In this case \sref{eq:int:nagumo:lde} can admit stable spatially periodic rest-states.
Numerical results indicate that these states can act as a buffer between regions
of space where $u = 0$ and $u = 1$ dominate the dynamics. This buffer shrinks
as these two stable states appear to move towards each other. This latter
process is governed by secondary two-component invasion waves
that we analyze in detail in this paper.

\paragraph{Nagumo PDE}
The LDE \sref{eq:int:nagumo:lde} can be seen as the
nearest-neighbour discretization of the
Nagumo reaction-diffusion PDE  \cite{Nagumo1962Active}
\begin{equation}
\label{eq:int:nagumo:pde}
u_t = u_{xx} + g (u; a), \qquad \qquad x \in \mathbb{R}
\end{equation}
on a spatial grid with size $h = d^{-1/2}$.
This PDE 
has been used as a highly simplified model for the spread of genetic traits
\cite{Aronson1975nonlinear} and the propagation of electrical signals through nerve
fibres \cite{Bell1984}.
In higher space dimensions it also serves as a desingularization of the
standard mean-curvature flow that is often used to describe the evolution of interfaces \cite{Evans1992}.

Fife and McLeod \cite{Fife1977} used phase plane analysis to show that
\sref{eq:int:nagumo:pde} admits
a front solution for each $a \in [0,1]$. Such solutions have the form
\begin{equation}
\label{eq:int:nagumo:front:ansatz}
u(x,t) = \Phi(x- ct), \qquad \Phi(-\infty) = 0, \qquad \Phi(+\infty) = 1,
\end{equation}
for some smooth waveprofile $\Phi$ and wavespeed $c$ that has the same sign as $a - \frac{1}{2}$.
These fronts hence connect the two stable spatially homogeneous equilibria
$u(x,t) \equiv 0$ and $u(x,t) \equiv 1$.

Exploiting the comparison principle, Fife and McLeod were able to show that these
front solutions have a surprisingly large basin of attraction. Indeed, any
solution to \sref{eq:int:nagumo:pde} with an initial condition
$u(x,0) = u_0(x)$
that has $u_0(x)\approx 0$ for $x \ll -1$ and  $u_0(x) \approx 1$ for $x \gg + 1$
will converge to a shifted version 
of this
front as $t \to \infty$.

These front solutions can be used as building blocks to capture the behaviour
of a more general class of solutions to \sref{eq:int:nagumo:pde}.
Consider for example the two-parameter family of functions
\begin{equation}
\label{eq:int:def:u:glue}
u_{\mathrm{plt};\alpha_0, \alpha_1}(x,t) = \Phi(x - c t + \alpha_0) + \Phi( - x - ct + \alpha_1) - 1,
\end{equation}
with $\alpha_1 \ge \alpha_0$. Each of these functions can be interpreted as a shifted version  of the front solution \sref{eq:int:nagumo:front:ansatz}
that is reflected in a vertical line to form a plateau.

If $c < 0$, then any initial configuration that has $u_0(x) \approx 0$
for $\abs{x} \gg L$ and $u_0(x) \approx 1$ for $\abs{x} \le L$
will converge to a member of the family \sref{eq:int:def:u:glue} as $t \to \infty$.
This provides a mechanism by which compact regions where $u \sim 1$ can spread out to fill the entire domain.

On the other hand, when $c > 0$ one can construct entire solutions that converge to
an element of \sref{eq:int:def:u:glue} as $t \to - \infty$ and tend to zero as $t \to + \infty$. These solutions
are stable under small perturbations \cite{Yagisita2003backward}. In particular, they can be viewed
as a robust elimination process whereby compact regions that have $u \sim 1$ are annihilated
by two incoming travelling fronts that collide as $t \to \infty$.

\paragraph{Nagumo LDE}

For many physical phenomena
such as
crystal growth in materials \cite{CAHN},
the formation of fractures
in elastic bodies \cite{Slepyan2012models}
and the motion of dislocations \cite{Celli1970motion}
and domain walls  \cite{Dmitriev2000domain} through crystals,
the discreteness and topology of the underlying spatial domain
have a major impact on the dynamical behaviour.
It is hence important to develop
mathematical modelling tools that
can incorporate such structures effectively.
Indeed, by now it is well known that discrete models can capture dynamical behaviour
that their continuous counterparts can not.

The LDE \sref{eq:int:nagumo:lde} has served as a prototype system
in which such effects can be explored.
It arises as a highly simplified model for the propagation of action potentials through nerve fibers
that have regularly spaced gaps in their myeline coating \cite{Bell1984}.
Two-dimensional versions have been used to describe
phase transitions in Ising models \cite{BatesDiscConv},
to analyze predator-prey interactions 
\cite{Slavik2017}
and to develop pattern recognition algorithms in image processing \cite{CHUA1988b,CHUA1988}. Recently,
an interest has also arisen in Nagumo equations posed on graphs \cite{Stehlik2017},
motivated by the network structure present in many biological systems \cite{Selley2015dynamic}.

Many authors have studied the LDE \sref{eq:int:nagumo:lde},
focusing primarily on the richness of the set of equilibria \cite{VL30}
and the existence of travelling and standing front solutions \cite{MPB, VL50}.
Such solutions have the form \sref{eq:int:discrete:wave},
which leads naturally to the waveprofile equation
\begin{equation}
\label{eq:int:mwv:mfde}
-c \Phi'(\xi) = d \big[ \Phi(\xi - 1) - 2 \Phi(\xi) +  \Phi(\xi + 1) \big] + g\big(\Phi(\xi) ; a \big) .
\end{equation}
Since the behaviour of every lattice point is governed by the same profile $\Phi$, we refer to these
front solutions as monochromatic waves in this paper (in order to distinguish them from the
bichromatic waves we discuss in the sequel). The seminal results by Mallet-Paret \cite{MPB} show that for each $a \in [0,1]$ and $d > 0$ there exists a unique $c = c_{\mathrm{mc}}(a,d)$ for which such monochromatic (mc)
solutions exist.

\paragraph{Pinning}
Upon fixing $a \in (0,1) \setminus \{\frac{1}{2} \}$,
Zinner \cite{VL50} established that $c_{\mathrm{mc}}(a , d) \neq 0$ for $d \gg 1$,
while Keener \cite{VL28} showed that $c_{\mathrm{mc}}(a, d)= 0$ for $0 < d \ll 1$.
Upon fixing $d > 0$, Mallet-Paret established \cite{MPB} that
$c_{\mathrm{mc}}(a, d) \neq 0$ for $a \approx 1 $ and $a \approx 0$.
In addition, again for fixed $d > 0$,
the results in \cite{HOFFMPcrys,MPCP}
strongly suggest 
that 
there exists $\delta > 0$ so that
$c_{\mathrm{mc}}(a ,d) = 0$ whenever $\abs{a - \frac{1}{2}} \le \delta$; see Figure \ref{f:regions}.

This last phenomenon is called pinning and distinguishes the LDE
\sref{eq:int:nagumo:lde} from the PDE \sref{eq:int:nagumo:pde}.
It is a direct consequence of the fact that we have broken the translational invariance
of space. Indeed, \sref{eq:int:mwv:mfde} becomes
singular in the limit $c \to 0$ and the corresponding waveprofiles indeed typically
lose their smoothness.
Many results suggest that this phenomenon is generic for discrete systems
\cite{BatesDiscConv,CMPVV, EVV,EVV2005AppMath,HJHVL2005}.
However, by carefully tuning the nonlinearity $g$ it is possible
to design systems for which this pinning is absent \cite{ELM2006,HJH2011}.
Understanding the pinning phenomenon is an important and challenging mathematical problem
that also has practical ramifications.

\begin{figure}
\begin{center}
\includegraphics[width=.75\textwidth]{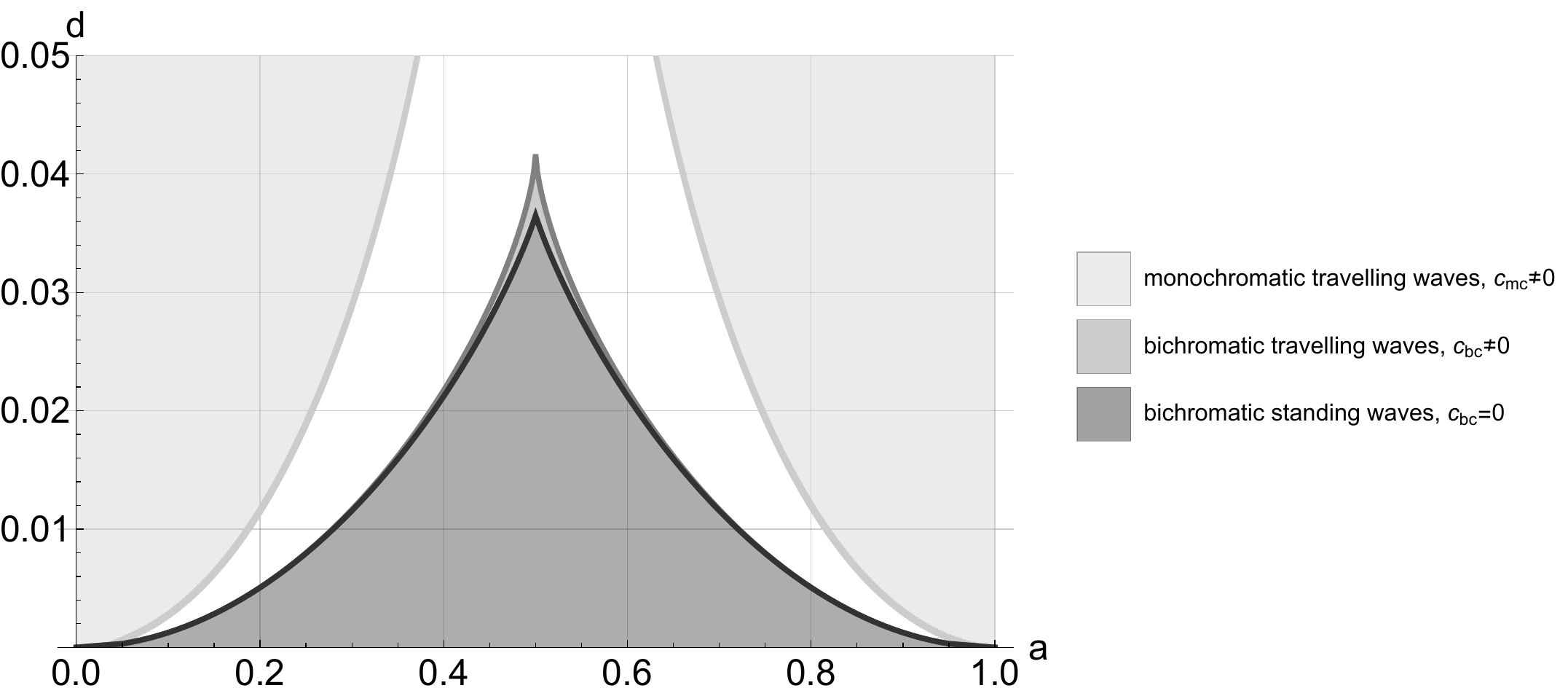}
\caption{Existence regions for monochromatic and bichromatic wave solutions to \eqref{eq:int:nagumo:lde}.} \label{f:regions}
\end{center}
\end{figure}

\paragraph{Periodicity}

In this paper we study waves that connect spatially homogeneous stationary solutions of \sref{eq:int:nagumo:lde}
with spatially heterogeneous 2-periodic stationary solutions. It is well known that many physical systems
exhibit spatially periodic
features \cite{FREID1985,GART1979,SHIG1986}.
Examples that also feature spatial discreteness
include
the presence of twinning microstructures
in shape memory alloys \cite{Bhattacharya2003microstructure}
and the formation of
domain-wall microstructures
in dielectric crystals \cite{Tagantsev2010domains}.

In many cases the underlying periodicity
comes from the spatial
system itself. For example,
in \cite{Faver2017nanopteron,Faver2018exact, Hoffman2017nanopteron} the authors
consider chains of alternating masses connected by identical springs
(and vice versa). The dynamical behaviour of such systems
can be easily modelled by LDEs
with periodic coefficients. In certain limiting cases the authors were able to construct
so-called nanopterons, which are multi-component wave solutions
that have low-amplitude oscillations in their tails.

However, periodic patterns
also arise naturally as solutions to spatially homogeneous discrete systems.
Indeed, we shall see in {\S}\ref{sec:eqlb}
that the LDE \sref{eq:int:nagumo:lde} with $d > 0$
admits many periodic equilibria.
In addition, 
the results in \cite{VL30} explore 
the periodicity and chaos present in the set of equilibria to homogeneous LDEs
with simplified nonlinearities.

It is also possible
to introduce a natural periodicity into the structure
of \sref{eq:int:nagumo:lde} by taking $d < 0$.
This can be seen by introducing new variables $v_j = (-1)^j u_j$,
which restores the applicability of the comparison principle.
This choice essentially
decomposes the lattice sites $\mathbb{Z}$ into two groups
$\mathbb{Z}_{\mathrm{odd}}$ and $\mathbb{Z}_{\mathrm{even}}$ that
each have their own characteristic behaviour.

Such anti-diffusion models have been
used to describe phase transitions
for grids of particles that have visco-elastic interactions
\cite{CAHNNOV1994,CAHNVLECK1999,VAIN2009}.
In \cite{BRUC2011} this problem has been analyzed in considerable detail. The authors
show that the resulting two component system admits co-existing patterns
that can be both monostable and bistable in nature. Similar results
with piecewise linear nonlinearities but more general couplings
between neighbours can be found in \cite{Vainchtein2015propagation}.

\begin{figure}
\begin{center}
\begin{minipage}{0.45\textwidth}
\includegraphics[width=\textwidth]{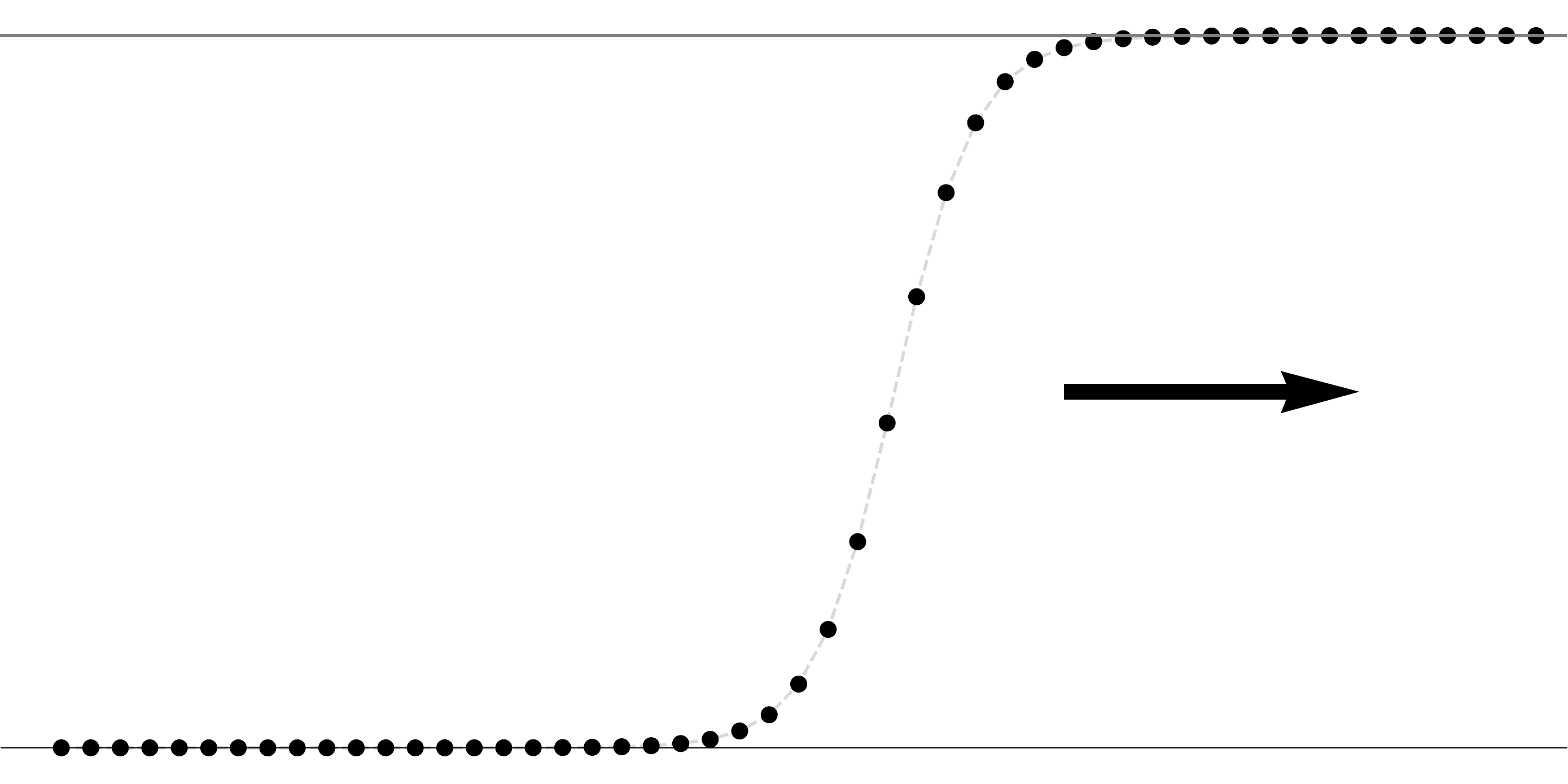}
\end{minipage}\quad
\begin{minipage}{0.45\textwidth}
\includegraphics[width=\textwidth]{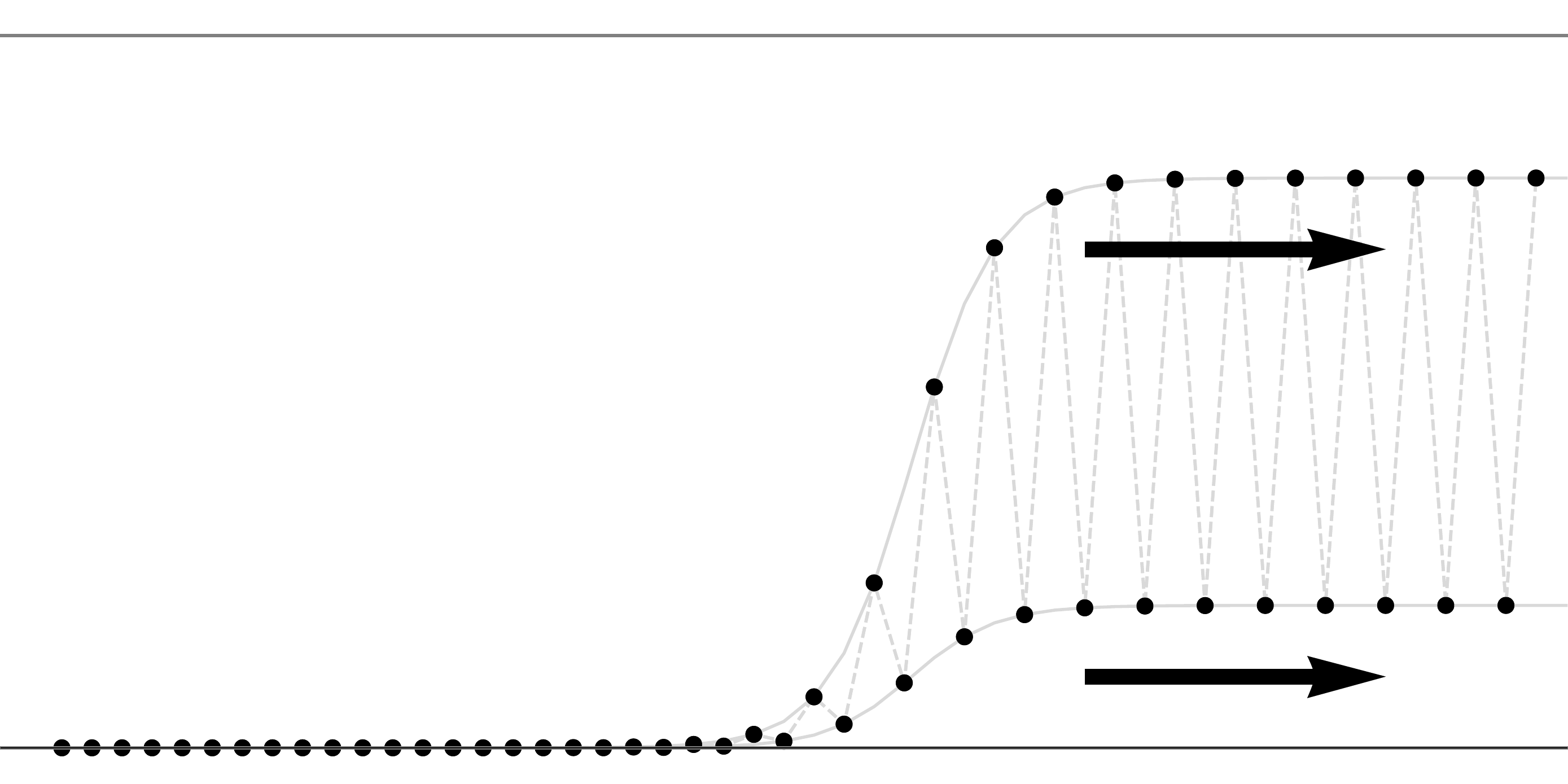}
\end{minipage}
\caption{A monochromatic travelling wave of \eqref{eq:int:nagumo:lde} (left panel) connects two spatially homogeneous stationary solutions.
A bichromatic travelling  wave of \eqref{eq:int:nagumo:lde} (right panel) connects a spatially homogeneous stationary solution with a spatially heterogeneous one.
}\label{f:waves}
\end{center}
\end{figure}

\paragraph{Bichromatic waves}

In this paper we are interested in the parameter region where $c_{\mathrm{mc}}(a,d) = 0$.
In a subset of this region it is possible to show that \sref{eq:int:nagumo:lde} has spatially heterogenous stable equilibria.
We focus on the simplest case and consider so-called bichromatic (two-color) equilibria,
which are spatially periodic with period two. As such, they are closely connected to solutions of the Nagumo equation
posed on a graph with two vertices. 
We set out to construct bichromatic front-solutions to \sref{eq:int:nagumo:lde},
which can be seen as waves that connect the spatially homogeneous equilibrium
$u \equiv 0$ with such a 2-periodic state. We emphasize that these differ
from the traditional front-solutions \sref{eq:int:discrete:wave} in the sense that
the odd and even lattice sites each have their own waveprofile, as illustrated in Figure \ref{f:waves}. Consequently, the bichromatic front-solutions are not monotone.

Our first main contribution is contained in \S\ref{sec:eqlb}, where we
give a detailed description of the set of
parameters $(a,d)$ where such 2-periodic equilibria exist and where they are stable. In contrast to
the setting encountered in \cite{BRUC2011}, the relevant bifurcation curves cannot all be described
explicitly. Besides a global result stating that the number of such equilibria decreases as $d$ is increased,
we also obtain precise asymptotics that describe the boundaries near
the three corners $(a,d) \in \{(0,0), (1/2,1/24), (1,0) \}$ in Figure \ref{f:regions}.

As in \cite{BRUC2011}, these preparations allow the existence of bichromatic fronts
to be established in a straightforward fashion. Indeed, one can apply the general theory
developed by Chen, Guo and Wu in \cite{CHENGUOWU2008}
for discrete periodic systems that admit a comparison-principle.
These results imply that there exists a unique wavespeed $c_{\mathrm{bc}}(a,d)$
for which such bichromatic (bc) fronts exist.
If $c_{\mathrm{bc}}(a,d) \neq 0$, the Fredholm theory developed in \cite{HJHNEGDIF} together with
the techniques from \cite[{\S}3]{HJHSTBFHN} can be used to show that these travelling fronts
depend smoothly on $(a,d)$ and are nonlinearly stable.

However, these general results cannot distinguish between
the cases $c_{\mathrm{bc}}(a,d) = 0$ and $c_{\mathrm{bc}}(a,d) \neq 0$
where we have standing respectively travelling fronts.
This should be contrasted to the situation for the PDE \sref{eq:int:nagumo:pde},
where the sign of the wavespeed is given by the sign of a
simple integral \cite{Fife1977}.
Indeed, there is a large set of
parameters $(a,d)$ for which the discrete bichromatic fronts fail to travel,
even though the analogous integral does not vanish.

Our second main contribution is that we provide explicit
criteria in \S\ref{sec:twv} that can guarantee $c_{\mathrm{bc}} = 0$
or $c_{\mathrm{bc}} > 0$. Together these results cover most
of the parameter region where bichromatic fronts exist.
In any case, they provide a two-component generalization
of the coercivity conditions introduced in \cite{MPB},
which ensure $c_{\mathrm{mc}}(a,d) \neq 0$ for
the boundary regions $a \approx 1$ and $a \approx 0$.

Our arguments to guarantee $c_{\mathrm{bc}} = 0$ are closely related
to the setup used by Keener \cite{VL28} to establish that
monochromatic waves are pinned for $0 < d \ll 1$.
In particular, for small values of $d$ one can neglect
the diffusion term in \S\ref{eq:int:nagumo:lde}
and use properties of the cubic to show that the derivative
of the waveprofile must change signs. This contradicts the fact that waveprofiles
must be strictly monotonic if they travel.

On the other hand, in \S4 we develop an intuitive geometric construction involving reflections to describe a planar
recurrence relation that standing bichromatic fronts must satisfy.
This allows us to rule out the presence of such fronts when a
scalar inequality is violated. This consequently implies the presence of travelling bichromatic fronts.

The parameter regimes where these two arguments apply
both converge towards the corner points $(0,0)$ and $(1, 0)$.
Near these corners we need the delicate asymptotics described above to distinguish
between the two cases.

\paragraph{Colliding fronts}
One of the main reasons for our interest in these bichromatic fronts is 
that they present mechanisms via which the stable homogeneous states
$u = 0$ and $u = 1$  can spread throughout the domain, even though
the primary invasion waves are blocked from propagation.
By using similar techniques as in \cite{Yagisita2003backward},
we believe it should be possible to construct entire solutions
consisting of a right-travelling bichromatic front
connection between the homogeneous equilibrium $u \equiv 0$ and a 2-periodic intermediate state
that collides with a left-travelling bichromatic connection
between 
 2-periodic intermediate state and the homogeneous equilibrium $u \equiv 1$; see Figure \ref{f:colfronts}.
The resulting state after the collision is then a pinned monochromatic front
that connects $0$ with $1$. We have been able to numerically verify the existence of these solutions
in the parameter regions predicted by the theory developed in this paper.

\begin{figure}
\begin{center}
\includegraphics[width=.75\textwidth]{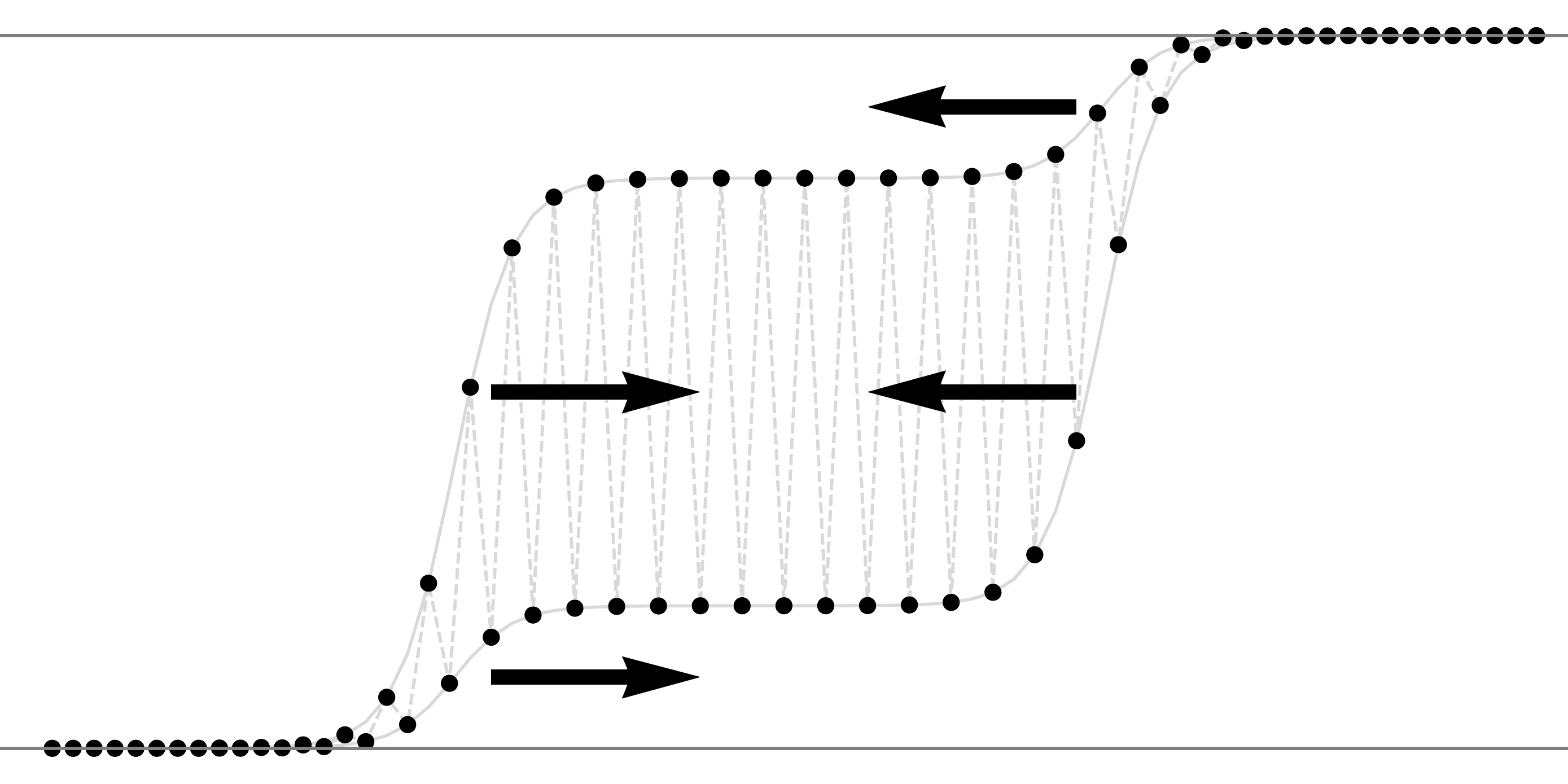}
\caption{Colliding front of \eqref{eq:int:nagumo:lde} consisting of a right-travelling bichromatic front
connection between the homogeneous equilibrium $u \equiv 0$ and a 2-periodic intermediate state
that collides with a left-travelling bichromatic connection
between the 2-periodic state and the homogeneous equilibrium $u \equiv 1$.} \label{f:colfronts}
\end{center}
\end{figure}

\paragraph{Acknowledgements}
HJH acknowledges support from the Netherlands Organization for Scientific Research (NWO)
(grant 639.032.612). LM acknowledges support from the Netherlands Organization for
Scientific Research (NWO) (grant 613.001.304).


\section{Main Results}
\label{sec:mr}
Our interest here is in the lattice differential equation
\begin{equation}
\label{eq:mr:lde:main}
\dot{x}_j(t) = d \big[ x_{j-1}(t) - 2 x_j(t) + x_{j+1}(t) \big] + g\big(x_j(t) ; a\big)
\end{equation}
posed on the one-dimensional lattice, i.e.,  $j \in \Wholes$. The bistable
nonlinearity is explicitly given by
\begin{equation}
g(u;a) = u ( 1 - u) ( u - a ),
\end{equation}
with $a\in (0,1)$. Our results concern so-called bichromatic (two-colour) travelling wave solutions to
the LDE \sref{eq:mr:lde:main}.
Such solutions can be written in the form
\begin{equation} \label{eq:tw:ansatz}
x_j(t) = \left\{ \begin{array}{lcl}
   \Phi_u( j - ct) & & \hbox{if } j \hbox{ is even}, \\[0.2cm]
   \Phi_v( j - ct) & & \hbox{if } j \hbox{ is odd}, \\[0.2cm]
   \end{array}
\right.
\end{equation}
for some wavespeed $c \in \Real$
and $\Real^2$-valued waveprofile
\begin{equation}
\Phi = (\Phi_u, \Phi_v): \Real \to \Real^2.
\end{equation}
Substituting this Ansatz into \sref{eq:mr:lde:main}
we obtain the travelling wave system
\begin{equation}
\label{eq:mr:wave:mfde}
\begin{array}{lcl}
- c  \Phi_u'(\xi)
 &=& d \big[ \Phi_v(\xi - 1) - 2 \Phi_u(\xi) + \Phi_v(\xi + 1) \big]
    + g\big(\Phi_u(\xi); a \big) ,
\\[0.2cm]
- c  \Phi_v'(\xi)
 &=& d \big[  \Phi_u(\xi - 1) - 2 \Phi_v(\xi) + \Phi_u(\xi + 1) \big]
    + g\big( \Phi_v(\xi); a \big) .
\end{array}
\end{equation}
Upon introducing the functions
\begin{equation}
\label{eq:mr:def:G:G12}
G(u,v;a, d)
= \left(
  \begin{array}{ccc}
     G_1(u,v;a,d)  \\[0.2cm]
     G_2(u, v; a, d)  \\[0.2cm]
  \end{array}
  \right)
= \left(
  \begin{array}{ccc}
     2d (v - u ) + g(u;a) \\[0.2cm]
     2d (u - v) + g(v;a) \\[0.2cm]
  \end{array}
  \right),
\end{equation}
we see that any stationary solution
\begin{equation}
(\Phi_u, \Phi_v)(\xi) = \big( \overline{u} , \overline{v} \big)
\end{equation}
to \sref{eq:mr:wave:mfde}
must satisfy the nonlinear algebraic equation
\begin{equation}
\label{eq:mr:eq:id}
G( \overline{u}, \overline{v} ; a , d) = 0.
\end{equation}
The full 
bifurcation
diagram for this equation is described in \S\ref{sec:eqlb}.
For our purposes here however it suffices
to summarize a subset of the conclusions from this analysis,
which we do in our first result below. In particular,
there exists a region $\Omega_{\mathrm{bc}}$ in the $(a,d)$-plane
for which the spatially homogeneous system
$(\dot{u}, \dot{v} ) = G(u, v; a, d)$
has a stable equilibrium
$\big(\overline{u}_{\mathrm{bc}}(a,d), \overline{v}_{\mathrm{bc}}(a,d) \big)$
that can be interpreted as a bichromatic
equilibrium state for the LDE \sref{eq:mr:lde:main}.

\begin{prop}[{see \S\ref{sec:eqlb}}]
There exists a continuous curve $d_{\mathrm{bc}}: [0,1] \to [0, \frac{1}{24}]$
with $d_{\mathrm{bc}}(\frac{1}{2}) = \frac{1}{24}$ and $d_{\mathrm{bc}}(1 - a) = d_{\mathrm{bc}}(a)$
so that for every $0 \le d < d_{\mathrm{bc}}$ and $0 < a < 1$
the system \sref{eq:mr:eq:id}
has nine distinct equilibria $(\overline{u},\overline{v}) \in [0, 1]^2$.
Upon writing
\begin{equation}
\Omega_{\mathrm{bc}} = \{ 0 < d < d_{\mathrm{bc}}(a) \hbox{ and } 0 < a < 1 \},
\end{equation}
there exist $C^\infty$-smooth maps
\begin{equation}
(\overline{u}_{\mathrm{bc}}, \overline{v}_{\mathrm{bc}}): \Omega_{\mathrm{bc}} \to (0, 1)^2
\end{equation}
with $\overline{u}_{\mathrm{bc}} < a < \overline{v}_{\mathrm{bc}}$ so that
for every $(a, d) \in \Omega_{\mathrm{bc}}$ we have
\begin{equation}
\label{eq:mr:ex:uv:ovl}
G\big( \overline{u}_{\mathrm{bc}}(a , d), \overline{v}_{\mathrm{bc}}(a, d) ;a  ,d \big)  = 0
\end{equation}
together with
\begin{equation}
\label{eq:mr:stab:uv:ovl}
\mathrm{det}  \, D_{1,2} G\big( \overline{u}_{\mathrm{bc}}(a , d),
  \overline{v}_{\mathrm{bc}}(a, d) ;a  ,d \big)  >0,
\qquad
\mathrm{Tr} \, D_{1,2} G\big( \overline{u}_{\mathrm{bc}}(a , d),
  \overline{v}_{\mathrm{bc}}(a, d) ;a  ,d \big)   <0 .
\end{equation}
\end{prop}
We note that the statements
\sref{eq:mr:ex:uv:ovl}-\sref{eq:mr:stab:uv:ovl}
are also valid upon replacing
the bichromatic rest-state $(\overline{u}_{\mathrm{bc}} , \overline{v}_{\mathrm{bc}})$
by the monochromatic equilibria $(0, 0)$ and $(1, 1)$.
We will be interested in
waves that connect these mono- and bichromatic equilibria together.
More precisely,
we set out to find solutions to
\sref{eq:mr:wave:mfde}
that satisfy either the 'lower' boundary conditions
\begin{equation}
  \label{eq:mr:bnd:low}
    \lim_{\xi \to - \infty}
      \Phi(\xi) = (0, 0),
      \qquad
    \lim_{\xi \to + \infty}
      \Phi(\xi) = (\overline{u}_{\mathrm{bc}} , \overline{v}_{\mathrm{bc}} ),
\end{equation}
or the 'upper' boundary conditions
\begin{equation}
  \label{eq:mr:bnd:up}
    \lim_{\xi \to - \infty}
      \Phi(\xi) = (\overline{u}_{\mathrm{bc}} , \overline{v}_{\mathrm{bc}} )
      \qquad
    \lim_{\xi \to + \infty}
      \Phi(\xi) = (1,1).
\end{equation}
The result below summarizes
several key facts concerning the existence and uniqueness
of such waves. It introduces subregions of $\Omega_{\mathrm{bc}}$ denoted by $\mathcal{T}_{\mathrm{low}}$ and $\mathcal{T}_{\mathrm{up}}$ where the bichromatic travelling waves \eqref{eq:tw:ansatz} exist with
nonzero speeds $c_{\mathrm{low}} > 0$ and $c_{\mathrm{up}} < 0$; see Figure \ref{f:Tlow:Tup}.
With the exception of the inequalities
$c_{\mathrm{low}} \ge 0$ and $c_{\mathrm{up}} \le 0$,
these properties follow directly from the theory
developed in \cite{CHENGUOWU2008,HJHNEGDIF}.

\begin{thm}[{see \S\ref{sec:twv:ex}}]
\label{thm:mr:twv:ex}
There exist continuous maps
\begin{equation}
c_{\mathrm{low}}: \Omega_{\mathrm{bc}} \to [0, \infty),
\qquad
c_{\mathrm{up}}: \Omega_{\mathrm{bc}} \to (-\infty, 0]
\end{equation}
that satisfy the following properties.
\begin{itemize}
\item[(i)]{
  Upon introducing the open sets
  \begin{equation}
    \begin{array}{lcl}
       \mathcal{T}_{\mathrm{low}} & = & \{ (a ,d) \in \Omega_{\mathrm{bc}} : c_{\mathrm{low}} > 0 \},
       \\[0.2cm]
       \mathcal{T}_{\mathrm{up}} & = & \{ (a ,d) \in \Omega_{\mathrm{bc}} : c_{\mathrm{up}} < 0 \},
    \end{array}
  \end{equation}
  the functions $c_{\mathrm{low}}$ and $c_{\mathrm{up}}$ are $C^\infty$-smooth
  on $\mathcal{T}_{\mathrm{low}}$
  respectively $\mathcal{T}_{\mathrm{up}}$.
}
\item[(ii)]{
  There exist $C^\infty$-smooth functions
  \begin{equation}
    \Phi_{\mathrm{low}} : \mathcal{T}_{\mathrm{low}}
      \to W^{1;\infty}(\Real; \Real^2),
    \qquad
    \Phi_{\mathrm{up}} : \mathcal{T}_{\mathrm{up}}
      \to W^{1;\infty}(\Real; \Real^2),
  \end{equation}
  such that for any $\# \in \{\mathrm{low}, \mathrm{up}\}$
  and any $(a,d) \in \mathcal{T}_{\#}$,
  the pair
  \begin{equation}
    (c, \Phi) = \big(c_{\#}(a,d) , \Phi_{\#}(a,d) \big)
  \end{equation}
  satisfies \sref{eq:mr:wave:mfde}
  together with the boundary condition
  \sref{eq:mr:bnd:low} if $\# = \mathrm{low}$
  or \sref{eq:mr:bnd:up} if $\# = \mathrm{up}$.
  In addition, we have the componentwise inequality $\Phi_{\#}' > (0,0)$.
}
\item[(iii)]{
  For any $\# \in \{\mathrm{low}, \mathrm{up}\}$
  and any $(a,d) \in \Omega_{\mathrm{bc}} \setminus \mathcal{T}_{\#}$,
  there exists a non-decreasing function
  $\Phi: \Real \to \Real^2$ that satisfies
  \sref{eq:mr:wave:mfde} with $c = 0$
  together with the boundary condition
  \sref{eq:mr:bnd:low} if $\# = \mathrm{low}$
  or \sref{eq:mr:bnd:up} if $\# = \mathrm{up}$.
}
\item[(iv)]{
  Pick $\# \in \{\mathrm{low}, \mathrm{up}\}$
  and $(a, d) \in \Omega_{\mathrm{bc}}$
  and consider any $c \neq 0$ together with
  a function $\Phi \in W^{1,\infty}(\Real;\Real^2)$
  that satisfies \sref{eq:mr:wave:mfde}
  together with the boundary condition
  \sref{eq:mr:bnd:low} if $\# = \mathrm{low}$
  or \sref{eq:mr:bnd:up} if $\# = \mathrm{up}$.
  Then we must have
  $c = c_{\#}(a , d)$
  and $\Phi = \Phi_{\#}(a,d)(\cdot - \vartheta)$
  for some $\vartheta > 0$. In particular,
  we have $(a, d) \in \mathcal{T}_{\#}$.
}
\item[(v)]{
  Pick $\# \in \{\mathrm{low}, \mathrm{up}\}$
  and $(a, d) \in \Omega_{\mathrm{bc}}$
  and consider any non-decreasing function
  $\Phi: \Real \to \Real^2$ that satisfies
  \sref{eq:mr:wave:mfde} with $c = 0$
  together with the boundary condition
  \sref{eq:mr:bnd:low} if $\# = \mathrm{low}$
  or \sref{eq:mr:bnd:up} if $\# = \mathrm{up}$.
  Then we must have
  $(a, d) \in \Omega_{\mathrm{bc}}
     \setminus \mathcal{T}_{\#}$.
}
\end{itemize}
\end{thm}

\begin{figure}
\centering
\begin{minipage}{0.55\textwidth}
\centering
\includegraphics[width=\textwidth]{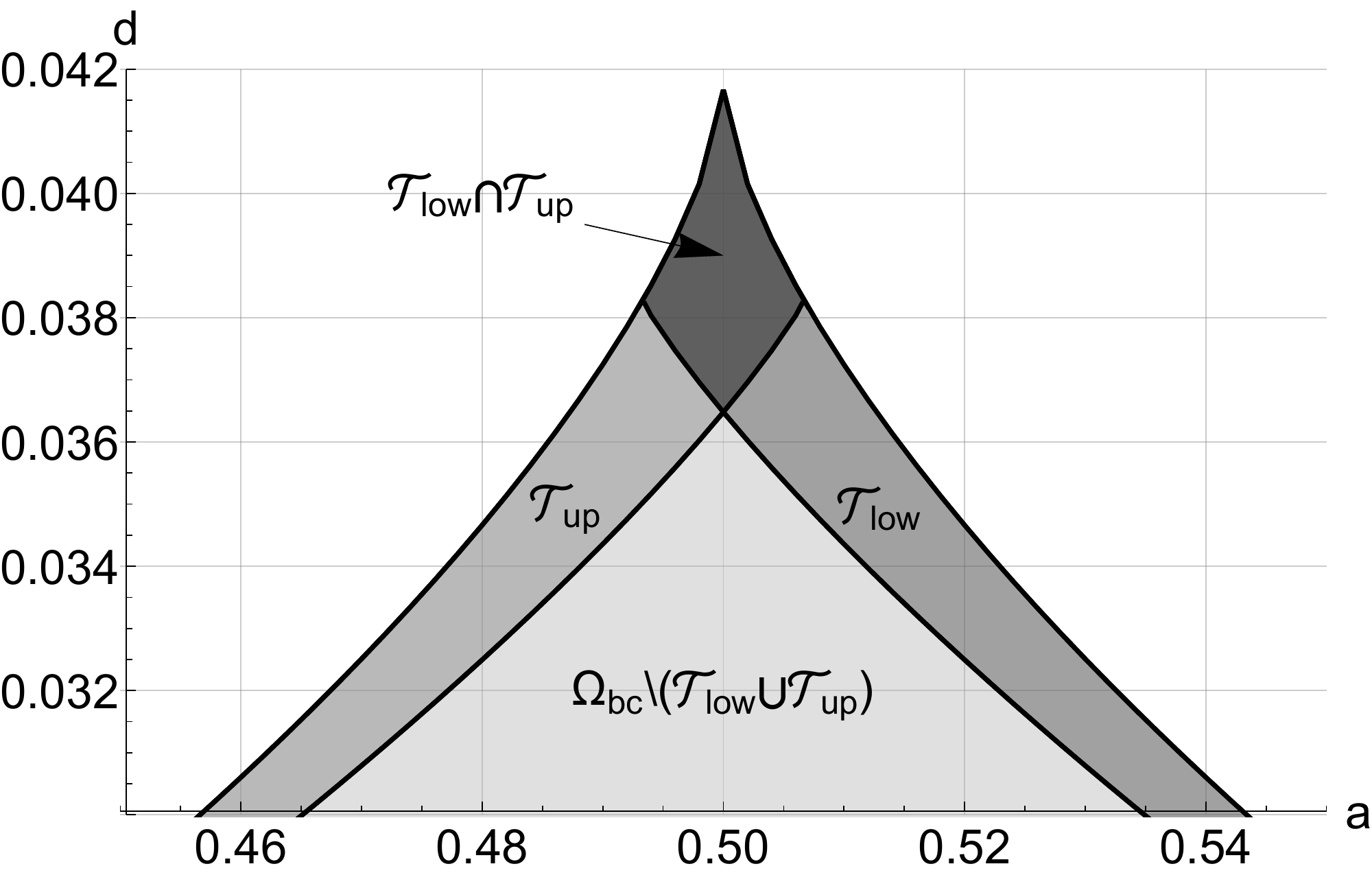}
\end{minipage}
\caption{Numerical bounds for the parameter sets $\Omega_{bc}$,  $\mathcal{T}_{\mathrm{low}}$ and
$\mathcal{T}_{\mathrm{up}}$ introduced in Theorem \ref{thm:mr:twv:ex},
in the neighbourhood of the cusp $(\frac{1}{2}, \frac{1}{24})$. 
}\label{f:Tlow:Tup}
\end{figure}

We numerically determined the locations
of the sets $\mathcal{T}_{\mathrm{low}}$ and $\mathcal{T}_{\mathrm{up}}$
in Figure \ref{f:Tlow:Tup}. In particular,
we simulated \sref{eq:mr:lde:main} with an initial
condition that consists of the stable periodic pattern
multiplied by a hyperbolic tangent.  By checking if
this solution converges to a travelling or stationary pattern
one can decide whether $(a,d) \in \mathcal{T}_{\mathrm{low}}$.

We now introduce the notation
\begin{equation}
\gamma_\pm(a)
= \frac{1}{3}
\Big[ a + 1 \pm \sqrt{ a^2 - a + 1 - 6 d_{\mathrm{bc}}(a) } \Big] .
\end{equation}
Writing $\alpha_a$
for the inverse of the strictly
increasing function
\begin{equation}
[0, \gamma_-(a)] \ni v \mapsto
   v - \frac{g(v;a)}{2 d_{\mathrm{bc}}(a)},
\end{equation}
we formally introduce the quantity
\begin{equation}
\label{eq:mr:expr:for:Gamma}
\begin{array}{lcl}
\Gamma(a)
 & = & 2 \gamma_+(a) - \frac{g'( \gamma_+(a);a)}{d_{\mathrm{bc}}(a)}
   - \overline{u}_{\mathrm{bc}}\big(a, d_{\mathrm{bc}}^-(a) \big)
\\[0.2cm]
& & \qquad
  - \max\big\{ u \in \big[0, \overline{u}_{\mathrm{bc}}\big(a, d_{\mathrm{bc}}^-(a)\big) \big] :
    2 u - \frac{g(u;a)}{d_{\mathrm{bc}}(a)} - \overline{v}_{\mathrm{bc}}\big(a, d_{\mathrm{bc}}^-(a) \big)
    = \alpha_a(u) \big\}
\end{array}
\end{equation}
for any $0 < a < 1$. Here the notation $d_{\mathrm{bc}}^-(a)$ refers
to the limit $d \uparrow d_{\mathrm{bc}}(a)$.
The geometric interpretation of this definition will be clarified in \S\ref{sec:twv:char}.
However,
we wish to emphasize here that one only needs information concerning the
quantities $(d_{\mathrm{bc}} , \overline{u}_{\mathrm{bc}}, \overline{v}_{\mathrm{bc}})$
associated to the two-dimensional algebraic problem $G(u,v;a,d) = 0$
in order to compute $\Gamma(a)$.  In particular, there is an essential
difference\footnote{The first problem is three dimensional, while the second problem is infinite dimensional
and hence involves truncations.} between computing
$\Gamma(a)$ and using the numerical procedure above to check whether $c \neq 0$.

The main contribution of the present paper is contained
in our final result, which provides analytical bounds for
the parameter regions
$\mathcal{T}_{\mathrm{low}}$ and $\mathcal{T}_{\mathrm{up}}$
where the bichromatic waves actually travel (i.e., where $c_{\mathrm{low}} > 0$
respectively $c_{\mathrm{up}} < 0$).
Both regions contain a neighbourhood of the cusp
$(a,d) = (\frac{1}{2} , \frac{1}{24})$. In addition,
the corners $(0,0)$ and $(1,0)$ are accumulation points
for the sets $\mathcal{T}_{\mathrm{up}}$ respectively $\mathcal{T}_{\mathrm{low}}$.


\begin{thm}[{see \S\ref{sec:twv:expl}}]
\label{thm:mr:twv:sets:t}
The sets $\mathcal{T}_{\mathrm{low}}$ and $\mathcal{T}_{\mathrm{up}}$
satisfy the following properties.
\begin{itemize}
\item[(i)]{
  For each $(a,d) \in \mathcal{T}_{\mathrm{up}}$
  we have $d > \frac{1}{8}a^2$,
  while for each $(a,d) \in \mathcal{T}_{\mathrm{low}}$
  we have $d > \frac{1}{8} (1 -a)^2$.
}
\item[(ii)]{
  If $(a,d) \in \mathcal{T}_{\mathrm{low}}$ then
  also $(a', d) \in \mathcal{T}_{\mathrm{low}}$ for all $(a', d) \in \Omega_{\mathrm{bc}}$
  that have $a' \ge a$. On the other hand, if $(a,d) \in \mathcal{T}_{\mathrm{up}}$
  then also $(a',d) \in \mathcal{T}_{\mathrm{up}}$ for all $(a', d) \in \Omega_{\mathrm{bc}}$
  that have $a' \le a$.
}
\item[(iii)]{
  There exists $\epsilon > 0$ so that we have the inclusions
  \begin{equation}
   (a, d) \in \mathcal{T}_{\mathrm{low}} \cap \mathcal{T}_{\mathrm{up}}
  \end{equation}
  for all $(a, d) \in \Omega_{\mathrm{bc}}$ that have
  \begin{equation}
    0 < \abs{a - \frac{1}{2}} + \abs{d - \frac{1}{24} } < \epsilon.
  \end{equation}
}
\item[(iv)]{
  The expression \sref{eq:mr:expr:for:Gamma} is well-defined
  for all $0 < a < 1$.
  If $\Gamma(a_*) > 0$ for some
  $0 < a_* < 1$, then
  there exists $\epsilon > 0$ so that
  $(a,d) \in \mathcal{T}_{\mathrm{low}}$
  for all $(a,d) \in \Omega_{\mathrm{bc}}$ that have
    \begin{equation}
      0 < \abs{a - a_*} + \abs{d - d_{\mathrm{bc}}(a_*)} < \epsilon.
    \end{equation}
}
\item[(v)]{
  The inequality $\Gamma(a) > 0$ holds whenever $1 - a > 0$
  is sufficiently small. In particular,
  we have $(0, 0) \in \overline{\mathcal{T}}_{\mathrm{up}}$
  and $(1, 0) \in \overline{\mathcal{T}}_{\mathrm{low}}$.
}
\end{itemize}
\end{thm}

Using numerics we have verified that 
$\Gamma(a) > 0$
holds for $a \in \big[.498, .999]$; see \S\ref{sec:twv:expl}.
Together with (ii) and (v) above, this strongly suggests that $\mathcal{T}_{\mathrm{low}}$ is a connected set
that extends towards the right boundary of $\Omega_{\mathrm{bc}}$.
By symmetry, this is equivalent to the statement that $\mathcal{T}_{\mathrm{up}}$ is a connected set
that extends towards the left boundary of $\Omega_{\mathrm{bc}}$.



\section{Bichromatic stationary solutions}
\label{sec:eqlb}

In this section we uncover the structure of the
solution set to $G(u, v;a, d) = 0$ as a function of
the parameters $(a, d)$. Our first
result shows that for $d \gg 1$ this equation
only has the monochromatic roots $(0, 0)$,
$(a,a)$ and $(1,1)$.
The threshold $d_+(a)$ between this region
and the region with five distinct roots can be explicitly
computed. However, we only have qualitative
and asymptotic results for the boundary $d_-(a)$ where the
root-count increases to the maximal value of nine.
In Figure \ref{f:bnd:curves} we compare these asymptotics
to numerically computed values for $d_-(a)$.
We remark here that the monotonicity of the root count with respect to $d$
does not hold for general bistable nonlinearities $g$.

\begin{prop}[{see \S\ref{sec:bif:glb}}]
\label{prp:st:dpm:ex}
There exist two continuous functions
\begin{equation}
d_\pm: [0, 1] \to [0,\infty)
\end{equation}
that satisfy the following properties.
\begin{itemize}
\item[(i)]{
  For any $0 < a < 1$ we have the explicit expression
  \begin{equation}
    d_+(a) = \frac{g'(a;a)}{4},
  \end{equation}
  together with the identities
  \begin{equation}
    d_-(a) = d_-(1 - a),
    \qquad
    d_+(a) = d_+(1 - a)
  \end{equation}
  and the inequality $d_-(a) < d_+(a)$.
  In addition, we have
  \begin{equation}
    d_-(0) = d_+(0) =  d_-(1) = d_+(1) =  0
  \end{equation}
  together with $d_-(\frac{1}{2}) = \frac{1}{24}$.
}
\item[(ii)]{
  We have $d_- \in C^\infty\big([0, \frac{1}{2}) \big) \cap C^\infty\big( (\frac{1}{2}, 1] \big)$.
  In addition, $d_-$ is strictly increasing on $[0, \frac{1}{2}]$
  and strictly decreasing on $[\frac{1}{2}, 1]$.
}
\item[(iii)]{
  Pick any $a \in (0, 1)$. The equation $G(u,v;a,d) = 0$
  has nine distinct roots   for $0 \le d < d_-(a)$, five distinct roots
  for $d_- < d < d_+(a)$ and three distinct real roots
  for $d \ge d_+(a)$.
}
\item[(iv)]{
  Pick any $a \in (0, 1)$. The equation $G(u,v;a,d) = 0$
  has seven distinct roots for $d = d_-(a)$
  if $a \neq \frac{1}{2}$ and five if $a = \frac{1}{2}$.
}
\item[(v)]{
 We have the expansion
 $d_-(a) = \frac{1}{8} a^2 + \frac{1}{32} a^4 +  O(a^5)$
 for $a \downarrow 0$. In addition, writing
 $a_{-}: [0, \frac{1}{24}] \to [0, \frac{1}{2}]$
 for the inverse  function of $d_-$
 on $[0, \frac{1}{2}]$,
 we have the expansion
 \begin{equation}
   a_-(d) =
   \frac{1}{2} - \sqrt{-1152(d - \frac{1}{24})^3}
     + O \big( ( d - \frac{1}{24})^2 \big)
 \end{equation}
 as $d \uparrow \frac{1}{24}$.
}
\end{itemize}
\end{prop}

\begin{figure}
\centering
\begin{minipage}{0.55\textwidth}
\includegraphics[width=\textwidth]{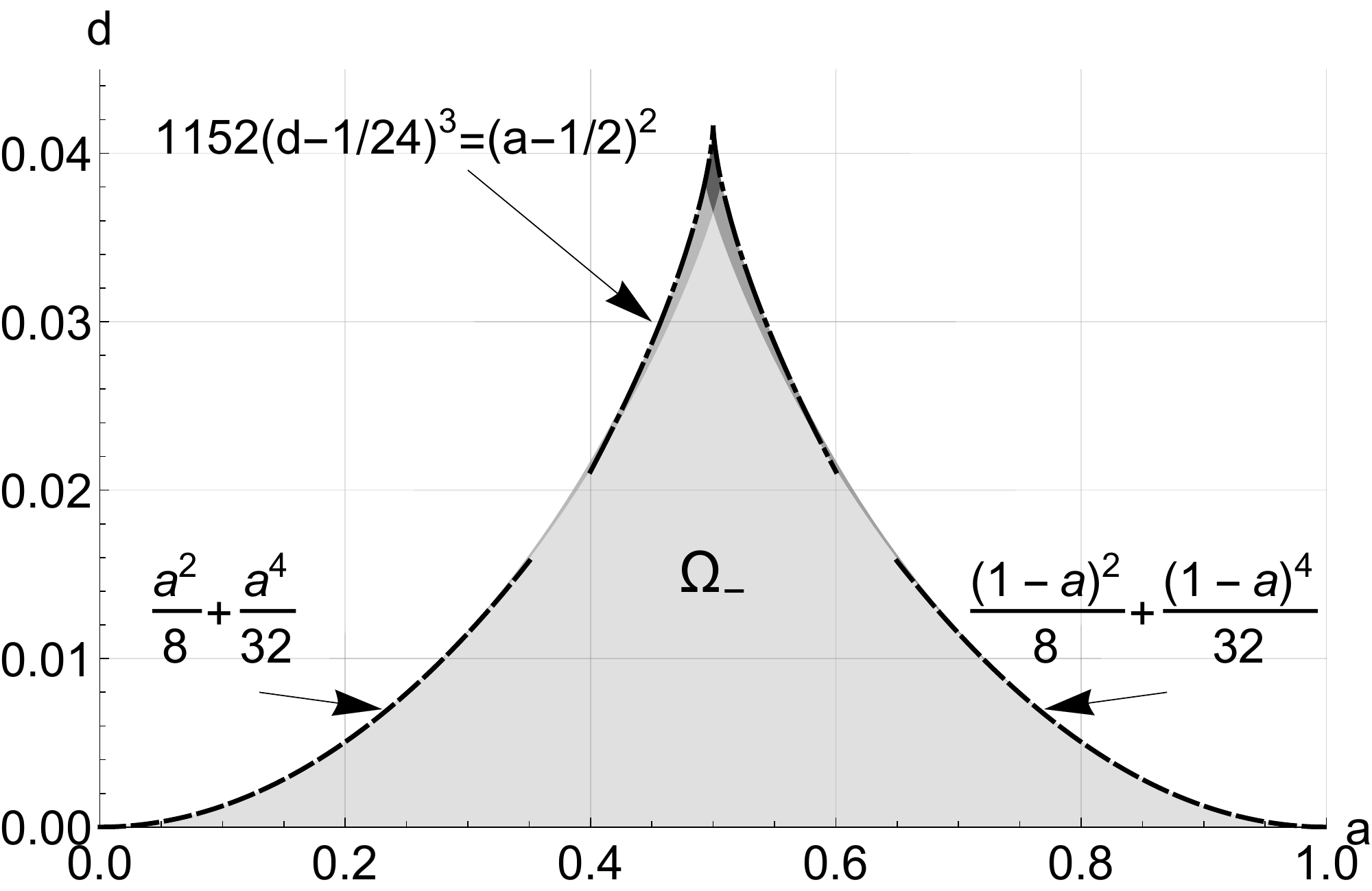}
\end{minipage}
\caption{
Comparison of the asymptotics for $d_{-}$ described in Proposition \ref{prp:st:dpm:ex} (v)
with the numerically computed border of the set $\Omega_-$.
}\label{f:bnd:curves}
\end{figure}

In order to break the symmetry caused by the swap
$u \leftrightarrow v$,
we set out to describe the roots of
$G(u,v;a,d) = 0$ that have $v > u$.
To this end,
we introduce two regions
\begin{equation}
\begin{array}{lcl}
\Omega_- & = & \{ (a , d) : 0 < a < 1 \hbox{ and } 0 < d < d_-(a) \} ,
\\[0.2cm]
\Omega_+ & = & \{ ( a, d) : 0 < a < 1 \hbox{ and } d_-(a) < d < d_+(a) \}
\end{array}
\end{equation}
that are studied separately in the two results below.
In $\Omega_-$ there are three such bichromatic equilibria
with $v > u$.
These equilibria can be ordered and the middle one is
the only stable one.
Two (or three) of these equilibria collide at $d = d_-(a)$
in a saddle node (or pitchfork) bifurcation,
leaving a single unstable bichromatic equilibrium in $\Omega_+$.
This equilibrium in turns collides with
its swapped counterpart
and the monochromatic equilibrium $(a,a)$
on the boundary $d_+(a)$. These processes are illustrated
in Figure \ref{f:ABCD}.
In particular, we see that $\Omega_-$ coincides with
the set $\Omega_{\mathrm{bc}}$ introduced in \S\ref{sec:mr}; cf. Figures \ref{f:Tlow:Tup} and \ref{f:bnd:curves}.

\begin{figure}
\begin{center}
\begin{minipage}{0.4\textwidth}
\includegraphics[width=\textwidth]{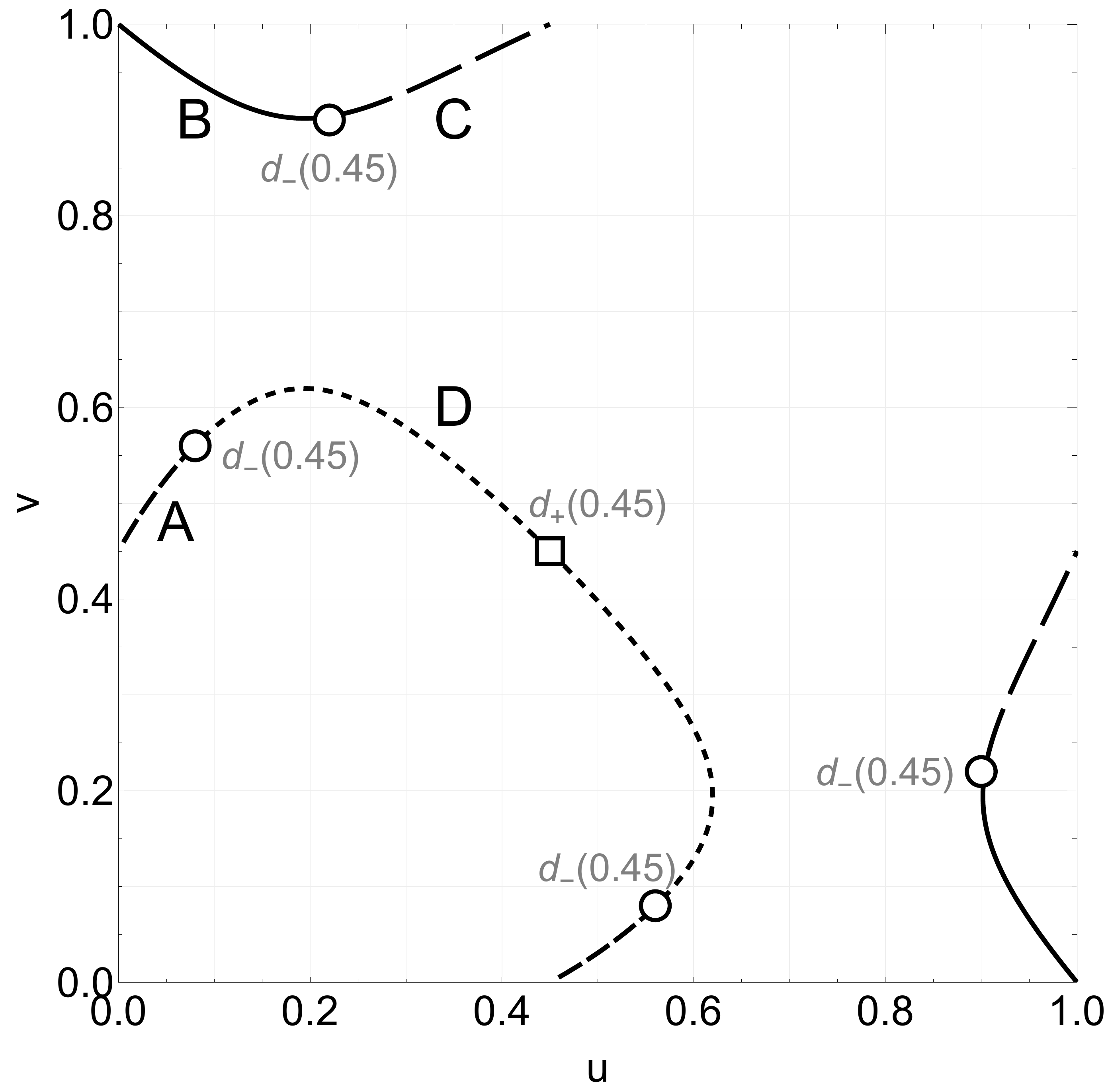}
\end{minipage}\quad
\begin{minipage}{0.4\textwidth}
\includegraphics[width=\textwidth]{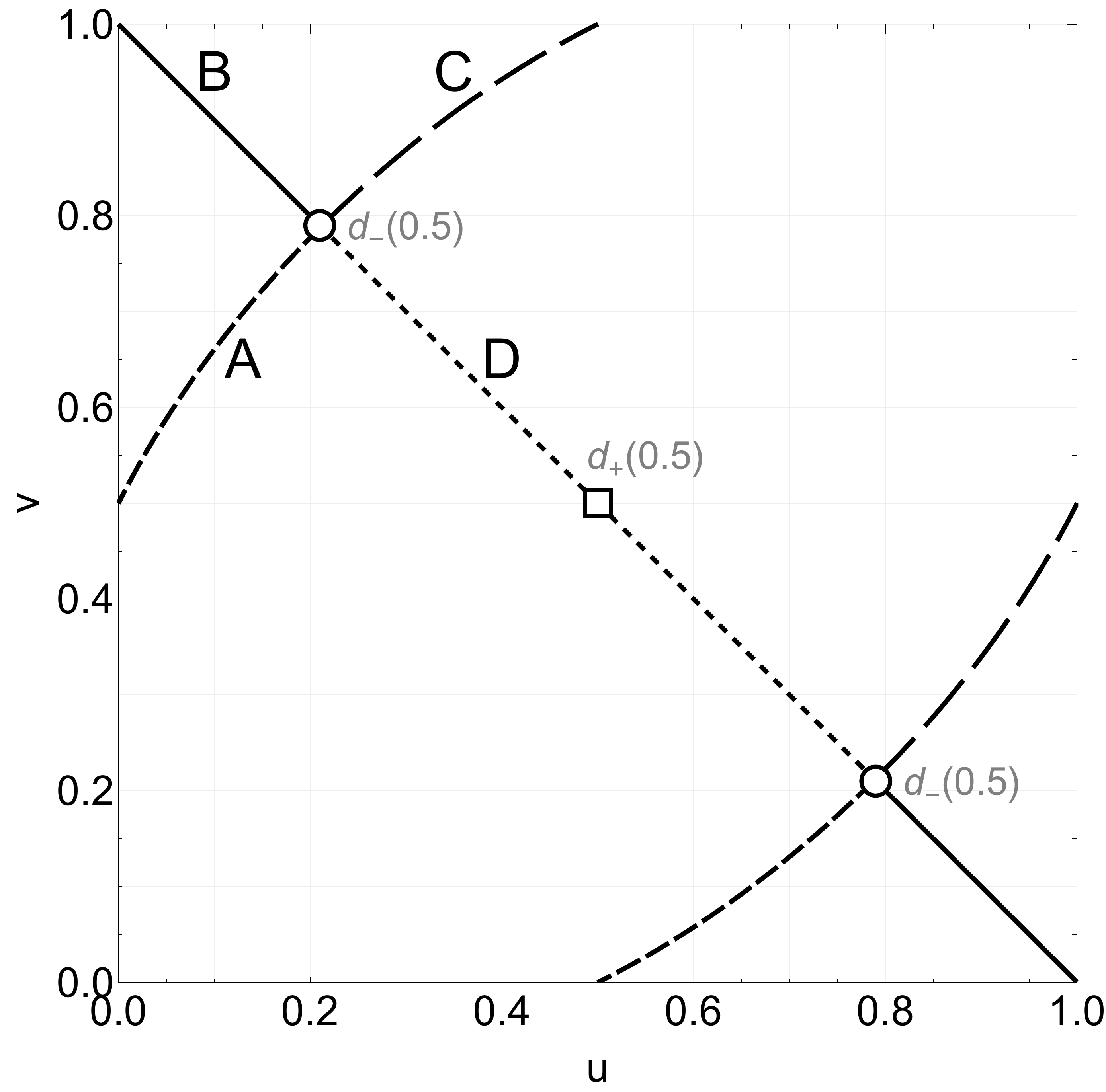}
\end{minipage}
\caption{Illustration of the functions $( \overline{u}_A,\overline{v}_A)$,
$( \overline{u}_B,\overline{v}_B)$, $( \overline{u}_C,\overline{v}_C)$  from Proposition \ref{prp:st:uv:stb:ex}
and the function $( \overline{u}_D,\overline{v}_D)$  from Proposition \ref{prp:st:uv:ustb:ex}
for $a=.45$ (left panel) and $a=.5$ (right panel). The bifurcations occuring at $d_+(a)$ and $d_-(a)$ are
indicated by squares and circles; see Proposition \ref{prp:st:dpm:ex}.}\label{f:ABCD}
\end{center}
\end{figure}

\begin{prop}[{see \S\ref{sec:bif:glb}}]
\label{prp:st:uv:stb:ex}
There exist continuous functions
\begin{equation}
 ( \overline{u}_A,\overline{v}_A): \overline{\Omega}_- \to [0, 1]^2
 \qquad
 ( \overline{u}_B,\overline{v}_B): \overline{\Omega}_- \to [0, 1]^2,
 \qquad
 ( \overline{u}_C,\overline{v}_C ): \overline{\Omega}_- \to [0, 1]^2
\end{equation}
that satisfy the following properties.
\begin{itemize}
\item[(i)]{
  Pick any $(a,d) \in \overline{\Omega}_-$. Then we have
  \begin{equation}
     G(\overline{u}_{\#}(a,d), \overline{v}_{\#}(a,d); a, d)  = 0
  \end{equation}
  for all $\# \in \{A , B, C \}$.
  If also $(a, d) \in \Omega_-$, then
  the matrix
  \begin{equation}
    D_{1,2} G(\overline{u}_{\#}(a,d),\overline{v}_{\#}(a,d) ;a , d)
  \end{equation}
  has two strictly negative eigenvalues if $\# = B$
  or one strictly positive
  and one strictly negative eigenvalue if $\# \in \{A, C \}$.
}
\item[(ii)]{
  For any $0 \le a \le 1$ we have the identities
  \begin{equation}
    (\overline{u}_A,\overline{v}_A)(a, 0) = (0, a),
    \qquad
    (\overline{u}_B,\overline{v}_B)(a, 0) = (0, 1),
    \qquad
    (\overline{u}_C,\overline{v}_C)(a, 0) = (a, 1) .
  \end{equation}
}
\item[(iii)]{
 For any $(a, d) \in \Omega_-$  we have the ordering
 \begin{equation}
 \label{eq:bif:abc:ordering}
 0 < \overline{u}_A(a,d) < \overline{u}_B(a,d) < \overline{u}_C(a,d)
   < a < \overline{v}_A(a,d) < \overline{v}_B(a,d) < \overline{v}_C(a,d) .
 \end{equation}
}

\item[(iv)]{
  For any $a \in [0, \frac{1}{2}]$ we have
  \begin{equation}
  (\overline{u}_B,\overline{v}_B)\big(a, d_-(a) \big) =
    (\overline{u}_C,\overline{v}_C)\big(a, d_-(a) \big) ,
  \end{equation}
  while for any $a \in [\frac{1}{2} , 1]$
  we have
  \begin{equation}
    (\overline{u}_A,\overline{v}_A)\big(a, d_-(a)\big)
     = (\overline{u}_B, \overline{v}_B )\big(a, d_-(a)\big).
  \end{equation}
}
\end{itemize}
\end{prop}

\begin{prop}[{see \S\ref{sec:bif:glb}}]
\label{prp:st:uv:ustb:ex}
There exist continuous functions
\begin{equation}
 (\overline{u}_D,\overline{v}_D):
   \overline{\Omega}_+ \to [0, 1]^2
\end{equation}
that satisfy the following properties.
\begin{itemize}
\item[(i)]{
Pick any $(a,d) \in \overline{\Omega}_+$.
Then we have
\begin{equation}
  G(\overline{u}_D(a,d) , \overline{v}_D(a,d) ; a, d) = 0.
\end{equation}
If also $(a, d) \in \Omega_+$, then the matrix
\begin{equation}
D_{1,2} G( \overline{u}_D(a,d) , \overline{v}_D(a,d) ; a, d )
\end{equation}
has one strictly positive
and one strictly negative eigenvalue.
}
\item[(ii)]{
  For any $0 \le a \le 1$ we have the identity
  \begin{equation}
    (\overline{u}_D,\overline{v}_D)\big(a, d_+(a)\big) = (a, a).
  \end{equation}
}
\item[(iii)]{
  For any $(a , d) \in \Omega_+$ we have the ordering
  \begin{equation}
  0 < \overline{u}_D(a,d) < a < \overline{v}_D(a,d) < 1 .
  \end{equation}
}

\item[(iv)]{
  For any  $a \in [0, \frac{1}{2}]$  we have the identity
  \begin{equation}
    (\overline{u}_D,\overline{v}_D)\big(a, d_-(a)\big) =
     (\overline{u}_A,\overline{v}_A)\big(a, d_-(a) \big)
     = (\overline{u}_B,\overline{v}_B)\big(a, d_-(a) \big)
     ,
  \end{equation}
  while for
  any $a \in [\frac{1}{2},  1]$ we have
  \begin{equation}
    (\overline{u}_D,\overline{v}_D)\big(a, d_-(a) \big) =
        (\overline{u}_B,\overline{v}_B)\big(a, d_-(a) \big)
      = (\overline{u}_C,\overline{v}_C)\big(a, d_-(a) \big)
     .
  \end{equation}
}
\end{itemize}
\end{prop}

\begin{cor}
\label{cor:bif:1:min:a}
For any $(a,d) \in \overline{\Omega}_-$,
we have the identities
\begin{equation}
\begin{array}{lcl}
\big(\overline{u}_{A}, \overline{v}_A \big) (  a, d)
 & = &
 \big( 1 - \overline{v}_C , 1 - \overline{u}_C \big)
     ( 1 - a , d),
\\[0.2cm]
\big(\overline{u}_{B}, \overline{v}_B \big) (  a, d)
 & = &
 \big( 1 - \overline{v}_B , 1 - \overline{u}_B \big)
     ( 1 -a , d),
\\[0.2cm]
\big(\overline{u}_{C}, \overline{v}_C \big) (  a, d)
 & = &
 \big( 1 - \overline{v}_A , 1 - \overline{u}_A \big)
     ( 1 -a , d).
\end{array}
\end{equation}
In addition, for any $(a,d) \in \overline{\Omega}_+$
we have the identity
\begin{equation}
\begin{array}{lcl}
\big(\overline{u}_{D}, \overline{v}_D \big) ( a, d)
 & = &
 \big( 1 - \overline{v}_D , 1 - \overline{u}_D \big)
     ( 1 - a , d).
\\[0.2cm]
\end{array}
\end{equation}
\end{cor}
\begin{proof}
The symmetry
$g(1-u, 1-a) = - g(u, a)$
implies that
\begin{equation}
G(1- u, 1 - v; 1-a,d)
= - G( u, v; a,d).
\end{equation}
In addition, we have $G(u,v;a , d) = 0$
if and only if $G(v, u ; a, d) = 0$.
The statements hence follow
from the ordering \sref{eq:bif:abc:ordering}.
\end{proof}

Our final result concerns the special case $a = \frac{1}{2}$,
in which case it is possible to be more explicit.
In particular, the  bichromatic roots
$(\overline{u}_B , \overline{v}_B)$
and $(\overline{u}_D , \overline{v}_D)$
lie on the line $u + v = 1$
and collide precisely when
$g'(u;\frac{1}{2}) = g'(v;\frac{1}{2}) = 0$.
\begin{cor}
For any $0 \le d \le \frac{1}{24}$ we have
\begin{equation}
  \overline{u}_B(\frac{1}{2}, d) = 1 - \overline{v}_B(\frac{1}{2}, d),
\end{equation}
while for any $\frac{1}{24} \le d \le \frac{1}{16} = d_+(1/2)$
we have
\begin{equation}
  \overline{u}_D(\frac{1}{2}, d) = 1 - \overline{v}_D(\frac{1}{2}, d).
\end{equation}
In addition,
we have the identities
\begin{equation}
  \label{eq:bif:ids:for:cusp}
  \begin{array}{lcl}
    \overline{u}_A( \frac{1}{2}, \frac{1}{24})
    = \overline{u}_B( \frac{1}{2}, \frac{1}{24})
    = \overline{u}_C( \frac{1}{2}, \frac{1}{24})
     = \overline{u}_D( \frac{1}{2}, \frac{1}{24})
     & = &
       \frac{1}{2} - \frac{1}{6} \sqrt{3} ,
    \\[0.2cm]
    \overline{v}_A( \frac{1}{2}, \frac{1}{24})
    = \overline{v}_B( \frac{1}{2}, \frac{1}{24})
     = \overline{v}_C( \frac{1}{2}, \frac{1}{24})
      = \overline{u}_D( \frac{1}{2}, \frac{1}{24})
      & = &
       \frac{1}{2} + \frac{1}{6} \sqrt{3} .
  \end{array}
\end{equation}
\end{cor}

\subsection{Saddle-nodes around $a = 0$ }

In this section we construct the branches
$(\overline{u}_B, \overline{v}_B)$
and $(\overline{u}_C, \overline{v}_C)$
of solutions to $G(u,v ;a, d) = 0$
in the regime where $(a,d) \approx (0, 0)$.
In particular, we define
\begin{equation}
H(u, v ; a, d) = G(u, 1 + v;  a,d)
\end{equation}
and determine the zeroes of $H$ for which $(u,v,a,d)$ are small.

\begin{prop}
\label{prp:bif:exp:a:zero}
There exist constants $\delta_a > 0$, $\delta_d > 0$ and $\epsilon > 0$
together with a function $d_c : (0, \delta_a) \to (0, \delta_d)$
and a constant $K \ge 1$
so that the following holds true.
\begin{itemize}
\item[(i)]{
  For every $0 < a < \delta_a$ and $0 < d < d_c(a)$
  the equation $H(u,v;a,d) = 0$ has precisely two solutions
  on the set $\{\abs{u} + \abs{v}  < \epsilon \}$.
}
\item[(ii)]{
  For every $0 < a < \delta_a$ and  $d_c(a) < d < \delta_d$
  the equation $H(u,v;a,d) = 0$ has no solutions
  on the set $\{\abs{u} + \abs{v}  < \epsilon \}$.
}
\item[(iii)]{
  For every $0 < a < \delta_a$ we have the estimate
  \begin{equation}
    \abs{ d_c(a) - \frac{1}{8} a^2 - \frac{1}{32} a^4 } \le K a^5 .
  \end{equation}
}
\item[(iv)]{
  For every $0 < a < \delta_a$ and
  the equation $H(u,v;d_c(a),a) = 0$ has precisely
  one solution $\big(u_c(a), v_c(a)\big)$ on the set
  $\{\abs{u} + \abs{v}  < \epsilon \}$.
  We have
  the estimates
  \begin{equation}
    \begin{array}{lcl}
    \abs{u_c(a) - \frac{1}{2} a}
       & \le & K a^4 ,
    \\[0.2cm]
    \abs{v_c(a) + \frac{1}{4} a^2 + \frac{1}{8} a^3}
      & \le & K a^4.
    \end{array}
  \end{equation}
}
\end{itemize}
\end{prop}

Writing $H(u, v; a, d) = \big(H_1(u,v;a,d), H_2(u,v;a,d)\big)^T$
we can compute
\begin{equation}
\begin{array}{lcl}
H_1(u,  v ; a, d)
& = &
u^2 - ua - u^3 + u^2 a + 2d + 2dv - 2d u ,
\\[0.2cm]
H_2(u, v; a, d)
& = &
-v - 2 v^2 + v a - v^3
+ v^2 a + 2d u - 2d - 2d v .
\end{array}
\end{equation}
Our strategy is to use the identity $H_2 = 0$
to eliminate $v$ and then recast $H_1 = 0$
into the normal form of a saddle-node bifurcation.

\begin{lem}
Pick $\delta > 0$
sufficiently small. Then
there exist constants $\epsilon > 0$
and $K \ge 1$
together with functions
\begin{equation}
\alpha_0: (-\delta, \delta)^2 \to \Real,
\qquad
\alpha_1: (-\delta, \delta)^2 \to \Real,
\qquad
R_{\alpha;2} : (-\delta, \delta)^3 \to \Real
\end{equation}
that satisfy the following properties.
\begin{itemize}
\item[(i)]{
  For every $(u, a, d) \in  (-\delta, \delta)^3 $
  the equation $H_2(u, v; a, d) = 0$
  has a unique solution $v = v_*$ in the set $\{\abs{v} < \epsilon\}$.
  This solution is given by
  \begin{equation}
    \label{eq:bif:qexp:ansatz:v}
    v_*(u ; a , d) = \alpha_0(a, d)
      + \alpha_1(a, d) u
      + u^2 R_{\alpha;2}(u;a, d) .
  \end{equation}
}
\item[(ii)]{
  Upon writing
  \begin{equation}
    \begin{array}{lcl}
     \alpha_0(a,d)
        & = & -2d - 2ad + S_{\alpha_0}(a, d),
      \\[0.2cm]
      \alpha_1(a,d)
        & = & 2d + S_{\alpha_1} (a, d),
        \\[0.2cm]
     \end{array}
  \end{equation}
  the bounds
  \begin{equation}
    \begin{array}{lcl}
      \abs{S_{\alpha_0}(a,d) }
        & \le & K ( d^2 + \abs{d } a^2 ) ,
      \\[0.2cm]
      \abs{S_{\alpha_1}(a,d) }
        & \le & K \abs{d} ( \abs{a} + \abs{d}),
        \\[0.2cm]
     \end{array}
  \end{equation}
  together with
  \begin{equation}
    \begin{array}{lcl}
      \abs{\partial_d S_{\alpha_0}(a, d)}
        & \le & K ( \abs{d} + a^2) ,
        \\[0.2cm]
      \abs{\partial_d S_{\alpha_1}(a, d)}
        & \le &  K ( \abs{a} + \abs{d})
     \end{array}
  \end{equation}
  hold for all $(a,d) \in (-\delta, \delta)^2$.
}
\item[(iii)]{
  For every $(u,a,d) \in (-\delta, \delta)^3$
  we have the bounds
  \begin{equation}
    \label{eq:bif:sn:bnds:r:alpha}
    \begin{array}{lcl}
      \abs{R_{\alpha;2}(u;a, d) }
      + \abs{R_{\alpha;2}'(u;a, d) }
      + \abs{R_{\alpha;2}''(u;a, d) }
      + \abs{R_{\alpha;2}'''(u;a, d) }
        & \le & K d^2 ,
       \\[0.2cm]
    \end{array}
  \end{equation}
  together with
  \begin{equation}
    \begin{array}{lcl}
      \abs{\partial_d R_{\alpha;2}(u;a, d) }
      + \abs{\partial_d R_{\alpha;2}'(u;a, d) }
      + \abs{\partial_d R_{\alpha;2}''(u;a, d) }
      + \abs{\partial_d R_{\alpha;2}'''(u;a, d) }
        & \le & K \abs{d} .
       \\[0.2cm]
    \end{array}
  \end{equation}
}
\end{itemize}
\end{lem}
\begin{proof}
Substituting the Ansatz
\sref{eq:bif:qexp:ansatz:v}
into $H_2$, we
obtain the fixed point problems
\begin{equation}
\begin{array}{lcl}
\alpha_0
  & = & -2d -2d\alpha_0  + \alpha_0 a+ \alpha_0^2 a-2\alpha_0^2-\alpha_0^3,
\\[0.2cm]
\alpha_1
 & = & 2d
  -4\alpha_0\alpha_1+2\alpha_0\alpha_1 a-2d\alpha_1+\alpha_1 a
    -3 \alpha_0^2 \alpha_1 ,
\end{array}
\end{equation}
together with
\begin{equation}
\begin{array}{lcl}
R_2 & = &
 2\alpha_0 R_2 a+R_2 a+\alpha_1^2 a-3 \alpha_0 \alpha_1^2-2 d R_2 -2 \alpha_1^2
  -3 \alpha_0^2 R_2-4 \alpha_0 R_2
\\[0.2cm]
& & \qquad
  +(2 \alpha_1 R_2 a  -6 \alpha_0 \alpha_1 R_2-4 \alpha_1 R_2-\alpha_1^3)u
\\[0.2cm]
& & \qquad
  +(R_2^2 a-3 \alpha_0 R_2^2-2 R_2^2-3 \alpha_1^2 R_2)u^2
\\[0.2cm]
& & \qquad
  -3 \alpha_1 R_2^2u^3   -R_2^3 u^4.
\\[0.2cm]
\end{array}
\end{equation}
These fixed point problems can be successively solved
for small $a$, $d$ and $u$, along with their differentiated counterparts.
The desired estimates
can subsequently be obtained in a standard fashion
by computing Taylor expansions.
\end{proof}

In order to eliminate $v$ from $H_1$, we
define the function
\begin{equation}
\label{eq:bif:sn:def:j}
\begin{array}{lcl}
\mathcal{J}(u;a,d) & = & H_1(u, v_*(u;a,d); a, d)
\\[0.2cm]
& = &\beta_0(a,d) + \beta_1(a,d) u
+ u^2\big( 1 + a - u + R_{\beta;2}(u; a, d)
 \big),
\end{array}
\end{equation}
which forces us to write
\begin{equation}
\begin{array}{lcl}
\beta_0(a,d) & = &
2d + 2d \alpha_0(a,d),
\\[0.2cm]
\beta_1(a,d) & = &
-a - 2d + 2d \alpha_1(a,d),
\\[0.2cm]
R_{\beta,2}(u; a, d) & = &
 2d R_{\alpha;2}(u; a, d) .
\end{array}
\end{equation}
By applying a shift to $u$ the linear term
in \sref{eq:bif:sn:def:j} can be removed,
transforming \sref{eq:bif:sn:def:j} into
a normal form for saddle node bifurcations.

\begin{lem}
Pick $\delta > 0$ sufficiently small.
Then there exist constants $\epsilon > 0$
and $K \ge 1$
together with functions
\begin{equation}
u_*: (-\delta, \delta)^2 \to \Real,
\qquad
\zeta_0: (-\delta, \delta)^2 \to \Real,
\qquad
R_{\zeta;2} : (-\delta, \delta)^3 \to \Real
\end{equation}
that satisfy the following properties.
\begin{itemize}
\item[(i)]{
  For every $(\tilde{u}, a, d) \in (-\delta, \delta)^3$,
  we have the identity
  \begin{equation}
     \mathcal{J}(u_*(a, d) + \tilde{u} ;a,d)
      =    \zeta_0(a , d)
      + \tilde{u}^2 \big[
        1 + R_{\zeta;2}(\tilde{u}; a , d)
        \big] .
  \end{equation}
}
\item[(ii)]{
  Upon writing
  \begin{equation}
    \begin{array}{lcl}
       u_*(a , d) & = &
         \frac{1}{2} a + d
           - \frac{1}{8} a^2
           + \frac{1}{2} ad
           - \frac{1}{16} a^3
         + S_{u_*}(a , d) ,
       \\[0.2cm]
       \zeta_0(a, d) & = &
          2 d - \frac{1}{4} a^2
          - da  - 5 d^2
          + \frac{1}{8} a^3
          + \frac{1}{4} d a^2
          - \frac{1}{64} a^4
          + S_{\zeta_0}(a, d),
    \end{array}
  \end{equation}
  the bounds
  \begin{equation}
    \begin{array}{lcl}
      \abs{S_{u_*}(a,d) }
        & \le & K \big[ d^2 + a^2 \abs{ d } + a^4 \big] ,
      \\[0.2cm]
      \abs{
        S_{\zeta_0}(a, d)
        } & \le & K ( a^5 + \abs{d}^3 + d^2 \abs{a} + \abs{a}^3 \abs{d} ) ,
      \\[0.2cm]
    \end{array}
  \end{equation}
  together with
  \begin{equation}
    \begin{array}{lcl}
      \abs{\partial_d S_{u_*}(a,d) }
        & \le & K \big[ \abs{d} + a^2   \big] ,
      \\[0.2cm]
      \abs{
        \partial_d S_{\zeta_0}(a, d)
        } & \le & K (  d^2 + \abs{d} \abs{a} + \abs{a}^3  )
      \\[0.2cm]
    \end{array}
  \end{equation}
  hold for all $(a,d) \in (-\delta, \delta)^2$.
}
\item[(iii)]{
  For every $(\tilde{u}, a ,d ) \in (-\delta, \delta)^3$ we have the bounds
  \begin{equation}
    \begin{array}{lcl}
      \abs{ R_{\zeta;2}(\tilde{u} ; a, d) }
        & \le & K (\abs{a}+ \abs{d}) ,
      \\[0.2cm]
      \abs{ R_{\zeta;2}'(\tilde{u} ; a, d) }
        & \le & K \abs{d}^3.
      \\[0.2cm]
    \end{array}
  \end{equation}
}
\end{itemize}
\end{lem}
\begin{proof}
We first introduce the notation
\begin{equation}
\mathcal{N}_{\beta;2;u_*}(\tilde{u} ; a,d)
= \tilde{u}^{-2}\Big[
  R_{\beta;2}(u_* + \tilde{u})
  - R_{\beta;2}(u_*; a, d)
  - R'_{\beta;2}(u_*;a, d) \tilde{u}
  \Big]
\end{equation}
for $\tilde{u} \neq 0$,
together with $\mathcal{N}_{\beta;2;u_*}(0;a,d) =
\frac{1}{2} R''_{\beta;2}(u_*; a, d)$.
This allows us to compute
\begin{equation}
\begin{array}{lcl}
\mathcal{J}( u_* + \tilde{u}, a, d)
& = &
\gamma_0( a, d , u_*)
+ \gamma_1( a ,d , u_*) \tilde{u}
+ \big(1 + a - 3 u_* + R_{\gamma;2}
   (\tilde{u}; a, d, u_*)  \big) \tilde{u}^2 ,
\end{array}
\end{equation}
in which
\begin{equation}
\begin{array}{lcl}
\gamma_0(a, d, u_*) & = &
  \beta_0 +
  \beta_1 u_*
  + u_*^2[ 1 + a - u_* +  R_{\beta;2}(u_*) ] ,
\\[0.2cm]
\gamma_1(a, d, u_*) & = &
  \beta_1 +
  u_*^2 [ -1 + R_{\beta;2}'(u_*; a, d) ]
  + 2 u_* [1 + a - u_* + R_{\beta;2}(u_*) ] ,
\\[0.2cm]
R_{\gamma;2}(\tilde{u}; a, d, u_*)
& = &
   (u_* + \tilde{u})^2 \mathcal{N}_{\beta;2;u_*}(\tilde{u}; a,d)
  + R_{\beta;2}'(u_*; a, d) (2 u_* \tilde{u} + \tilde{u}^2 )
  + R_{\beta;2}(u_*; a, d).
\end{array}
\end{equation}
On account of \sref{eq:bif:sn:bnds:r:alpha}
we have\footnote{All primed constants in this paper are strictly positive
and do not depend on the variables appearing on the left hand side
of the inequalities where they appear.}
\begin{equation}
\abs{\mathcal{N}_{\beta;2;u_*}(\tilde{u} ; a, d)}
+
\abs{\mathcal{N}'_{\beta;2;u_*}(\tilde{u} ; a, d)}
\le C_1'
\end{equation}
and hence also
\begin{equation}
\label{eq:bif:cusp:bnds:r:gamma:2}
\abs{\mathcal{R}_{\gamma;2}(\tilde{u} ; a , d,u_*) }
+
\abs{\mathcal{R}'_{\gamma;2}(\tilde{u} ; a , d,u_*) }
\le C_2' .
\end{equation}
Setting $\gamma_1 = 0$
leads to the fixed point problem
\begin{equation}
\label{eq:bif:sn:fxp:for:ustar}
u_* = -\frac{1}{2} \beta_1
 - a u_* + \frac{3}{2} u_*^2
 - \frac{1}{2} u_*^2  R'_{\beta;2}(u_*; a, d)
 - u_* R_{\beta;2}(u_*; a,d),
\end{equation}
which has a unique small solution that
we write for the moment as
\begin{equation}
u_*(a, d, \beta_1) = - \frac{1}{2} \beta_1
  + \frac{1}{2} a \beta_1
  + \frac{3}{8} \beta_1^2
  - \frac{1}{2} a^2 \beta_1
  - \frac{9}{8} a \beta_1^2
  - \frac{9}{16} \beta_1^3 + \tilde{S}_{u_*}(a, d, \beta_1).
\end{equation}
In a standard fashion one obtains the bound
\begin{equation}
\abs{\tilde{S}_{u_*}(a, d, \beta_1) }
\le C_3'
\big[
  \beta_1^4 +
  \abs{ a }  \abs{\beta_1}^3 +
  (a^2 + d^2) \beta_1^2
   + (\abs{a}^3 + \abs{d}^3) \abs{\beta_1}
\big] .
\end{equation}
Differentiating \sref{eq:bif:sn:fxp:for:ustar}
we obtain
\begin{equation}
\begin{array}{lcl}
D_d u_*
& = & -\frac{1}{2} \partial_d \beta_1
  - \frac{1}{2} u_*^2
       \partial_d R'_{\beta;2}(u_*;a,d)
  - u_* \partial_d R_{\beta;2}(u_*; a,d)
\\[0.2cm]
& & \qquad
 - a D_d u_*
 + 3 u_* D_d u_*
 - u_* R'_{\beta;2}(u_*; a, d) D_d u_*
 - R_{\beta;2}(u_*; a, d) D_d u_* ,
\end{array}
\end{equation}
which yields the estimate
\begin{equation}
\abs{D_d u_* - \big[
  - \frac{1}{2}
  + \frac{1}{2} a
  + \frac{3}{4} \beta_1
\big] \partial_d \beta_1
} \le C_4' \big[
    d^2 \abs{\beta_1}
    + a^2
    + \abs{a} \abs{ \beta_1}
    + \beta_1^2
\big] \abs{\partial_d \beta_1 }.
\end{equation}

Using the bounds
\begin{equation}
\begin{array}{lcl}
\abs{\beta_0(a,d) - [2d - 4d^2]}
& \le & C_5' d^2 (\abs{a} +  \abs{d}),
\\[0.2cm]
\abs{\beta_1(a,d) + (a + 2d) }
  &\le& C_5' d^2,
\\[0.2cm]
\end{array}
\end{equation}
together with
\begin{equation}
\begin{array}{lcl}
\abs{\partial_d \beta_0(a,d)
 - [2 - 8 d ] }
  & \le & C_5' \abs{d}( \abs{a} + \abs{d}),
\\[0.2cm]
\abs{\partial_d \beta_1(a,d) +2 }
 & \le & C_5' \abs{d}
\end{array}
\end{equation}
and finally
\begin{equation}
\abs{R_{\beta;2}(u_*; a, d )} + \abs{R'_{\beta;2} (u_* ; a, d) }
+ \abs{R''_{\beta;2}(u_* ; a, d) } + \abs{R'''_{\beta;2}(u_* ; a, d) } \le C_6' \abs{d}^3,
\end{equation}
the desired estimates follow
by writing
\begin{equation}
\begin{array}{lcl}
\zeta_0(a, d) & = & \gamma_0\big( a, d, u_*(a , d) \big) ,
\\[0.2cm]
R_{\zeta;2}(\tilde{u}; a, d)
 & = & a - 3 u_*(a,d)
 + R_{\gamma;2}\big(\tilde{u}, a , d,
     u_*(a, d) \big)
\end{array}
\end{equation}
and computing $\partial_d \zeta_0$
directly.
\end{proof}

\begin{proof}
[Proof of Proposition \ref{prp:bif:exp:a:zero}]
We note first that the map
\begin{equation}
\tilde{u} \mapsto \tilde{u}
 \sqrt{1 + R_{\zeta;2}(\tilde{u} ;a ,d ) }
\end{equation}
is invertible for
$(\tilde{u}, a, d) \in (-\delta, \delta)^3$.
In order to find $d_c$
it hence suffices to
solve $\zeta_0(a , d_c) = 0$,
which gives
the fixed-point problem
\begin{equation}
d_c
 = \frac{1}{8} a^2 - \frac{1}{16} a^3 + \frac{1}{128}a^4
  + \frac{1}{2} a d_c - \frac{1}{8} d_c a^2   + \frac{5}{2} d_c^2
  - \frac{1}{2} S_{\zeta_0}(a, d_c).
\end{equation}
Our estimate on $\partial_d S_{\zeta_0}$
guarantees the existence of a unique small solution
for small $a$. The estimate (iii)
now follows in a standard fashion.
Writing $u_c = u_*\big(a , d_c(a)\big)$
and $v_c = v_*\big(u_*(a , d_c(a));a , d_c(a)  \big)$
the expansions in (iv) follow directly
by substitution.
\end{proof}

\subsection{The cusp bifurcation around $a=1/2$}

Our goal here is to unfold the structure of the solution-set
to $G(u,v;a, d)$ near the critical point
$(a, d) = (1/2, 1/24)$
where the branches $(\overline{u}_{\#}, \overline{u}_{\#})$
with $\# \in \{A , B, C, D \}$ all collide in a cusp bifurcation;
see Figures \ref{f:regions} and
\ref{f:ABCD}.
In particular, we introduce the cusp location
\begin{equation}
\begin{array}{lcl}
(u_{\mathrm{cp}}, v_{\mathrm{cp}}, a_{\mathrm{cp}}, d_{\mathrm{cp}})
 & = &
   \big( u_{\mathrm{min}}(\frac{1}{2}), u_{\mathrm{max}}(\frac{1}{2}),
     \frac{1}{12}, \frac{1}{24} \big)
\\[0.2cm]
 & = & \big(\frac{1}{2}-\frac{1}{6}\sqrt{3},
  \frac{1}{2}+\frac{1}{6}\sqrt{3}, \frac{1}{2}, \frac{1}{24} \big)
\end{array}
\end{equation}
together with the function
\begin{equation}
H_{\mathrm{cp}}(u, v; a, d) = G \big(u_{\mathrm{cp}} + u, v_{\mathrm{cp}} + v ;
  a_{\mathrm{cp}} + a, d_{\mathrm{cp}}+ d
    \big)
\end{equation}
and set out to determine the zeroes of $H_{\mathrm{cp}}$
for which $(u,v,a,d)$ are small.

\begin{prop}
\label{prp:bif:cusp:exp}
There exist constants $\delta_a > 0$,
$\delta_d > 0$ and $\epsilon > 0$
together with a function $a_c: (-\delta_d, 0)
\to (0, \delta_a)$ and a constant $K \ge 1$
so that the following holds true.
\begin{itemize}
\item[(i)]{
  For every $0 \le d < \delta_d$
  and any $a \in (-\delta_a, \delta_a)$,
  the equation $H_{\mathrm{cp}}(u, v; a,d ) = 0$
  has precisely one solution
  on the set
  $\{\abs{u} + \abs{v} < \epsilon\}$.
}
\item[(ii)]{
  For every $-\delta_d < d < 0$ and
  $a \in (-\delta_a , \delta_a)$
  with $\abs{a} > a_c(d)$,
  the equation $H_{\mathrm{cp}}(u,v;a,d) = 0$ has precisely one solution
  on the set
  $\{\abs{u} + \abs{v} < \epsilon\}$.
}
\item[(iii)]{
  For every $-\delta_d < d < 0$ and
  $a \in \big(-a_c(d), a_c(d) \big)$,
  the equation $H_{\mathrm{cp}}(u,v;a,d) = 0$ has precisely
  three solutions
  on the set
  $\{\abs{u} + \abs{v} < \epsilon \}$.
}
\item[(iv)]{
  For any $-\delta_d < d < 0$,
  the equation $H_{\mathrm{cp}}(u, v; a_c(d) , d) = 0$
  has precisely two solutions
  on the set
  $\{\abs{u} + \abs{v} < \epsilon\}$.
}
\item[(v)]{
  For any $d \in (-\delta_d, 0)$ we have the estimate
  \begin{equation}
    \abs{a_c(d) - \sqrt{-1152 d^3} }
      \le K d^2 .
  \end{equation}
}
\end{itemize}
\end{prop}

In order to recast our equation into an efficient
form, we introduce two functions $(h_1, h_2)$ by writing
\begin{equation}
\left(
   \begin{array}{c}
      h_1(p,q; a, d) \\[0.2cm]
      h_2(p,q; a, d) \\[0.2cm]
        \end{array}
  \right)
= \left(
   \begin{array}{cc} 1 &1 \\[0.2cm] 1 & - 1  \\[0.2cm]
        \end{array}
  \right) H_{\mathrm{cp}} \big( p+ q , p - q ; a , d \big),
\end{equation}
which can be evaluated as
\begin{equation}
\begin{array}{lcl}
h_1(p,q;a,d) & = &
  -\frac{1}{3}a+2 a p^2-6pq^2+2 a q^2
    -\frac{2}{3}\sqrt{3}a q-2p^3+2\sqrt{3}pq ,
\\[0.2cm]
h_2(p, q; a, d) & = &
-\frac{2}{3}\sqrt{3}a p-\frac{1}{3}q+\frac{4}{3}\sqrt{3}d
  +\sqrt{3}p^2+\sqrt{3}q^2-6p^2q-8dq+4a pq-2q^3 .
\\[0.2cm]
\end{array}
\end{equation}
Since $h_2$ contains a term that is linear in $q$,
we set out to eliminate this variable
by demanding $h_2 = 0$.
\begin{lem}
Pick $\delta > 0$ sufficiently small.
Then there exist constants $\epsilon > 0$
and $K \ge 1$, together with functions
\begin{equation}
\alpha_0: (-\delta, \delta)^2 \to \Real,
\qquad
\alpha_1: (-\delta, \delta)^2 \to \Real,
\qquad
\alpha_2: (-\delta, \delta)^2 \to \Real,
\qquad
R_{\alpha;3} : (-\delta, \delta)^3 \to \Real
\end{equation}
that satisfy the following properties.
\begin{itemize}
\item[(i)]{
For every $(p, a, d) \in (-\delta, \delta)^3$
the
equation $h_2(p, q; a, d) = 0$
has a unique solution $q = q_*$ in the set
$\{\abs{q} < \epsilon\}$.
This solution is given by
\begin{equation}
\label{eq:bif:cp:ans:q:star}
q_*(p; a, d)
 =
 \alpha_0(a, d)
 + \alpha_1(a, d)p
 + \alpha_2(a, d) p^2
 + p^3 R_{\alpha;3}(p; a, d).
\end{equation}
}
\item[(ii)]{
  Upon writing
  \begin{equation}
    \begin{array}{lcl}
    \alpha_0(a,d) & = &
        4 \sqrt{3} d
        + S_{\alpha_0}(a, d) ,
    \\[0.2cm]
    \alpha_1(a,d) & = &
        - 2 \sqrt{3} a
         + S_{\alpha_1}(a,d) ,
    \\[0.2cm]
    \alpha_2(a,d) & = &
       3\sqrt{3} + S_{\alpha_2}(a,d) ,
    \\[0.2cm]
    \end{array}
  \end{equation}
  the bounds
  \begin{equation}
    \begin{array}{lclclcl}
      \abs{S_{\alpha_0}(a, d) }
        & \le &  K d^2 ,
        & &
          \abs{ D S_{\alpha_0}(a,d) }
           & \le & K \abs{d} ,
      \\[0.2cm]
      \abs{S_{\alpha_1}(a, d) }
         & \le & K \abs{a} \abs{d} ,
         & &
          \abs{ D S_{\alpha_1}(a,d) }
           & \le & K ( \abs{a} + \abs{d} ) ,
      \\[0.2cm]
      \abs{S_{\alpha_2}(a, d) }
        & \le &  K ( \abs{d} + a^2) ,
        & &
          \abs{ D S_{\alpha_2}(a,d) }
           & \le & K
      \\[0.2cm]
    \end{array}
  \end{equation}
  hold for all $(a,d) \in (-\delta, \delta)^2$.
}
\item[(iii)]{
  For every
  $(p, a, d) \in (-\delta, \delta)^3$
  we have the bounds
  \begin{equation}
      \abs{R_{\alpha;3}(p;a, d)} \le K ( \abs{a}
      + \abs{p} ),
      \qquad
      \qquad
      \abs{D_{a,d}R_{\alpha;3}(p;a, d)} \le
       K ,
  \end{equation}
  together with
  \begin{equation}
    \begin{array}{lcl}
    \abs{R^{(i)}_{\alpha;3}(p;a, d)}
    +
    \abs{D_{a,d}R^{(i)}_{\alpha;3}(p;a, d)}
    & \le & K
    \\[0.2cm]
    \end{array}
 \end{equation}
 for all $1 \le i \le 6$.
}
\end{itemize}
\end{lem}
\begin{proof}
Substituting the Ansatz \sref{eq:bif:cp:ans:q:star}
into $h_2$,
we obtain the fixed point problems
\begin{equation}
\begin{array}{lcl}
\frac{1}{3}\alpha_0 & = &
  \frac{4}{3} \sqrt{3} d
  -8 d \alpha_0 + \sqrt{3} \alpha_0^2
  -2 \alpha_0^3 ,
\\[0.2cm]
\frac{1}{3}\alpha_1 & = &
  - \frac{2}{3} \sqrt{3} a
  + 4 \alpha_0 a
  + 2 \sqrt{3} \alpha_0 \alpha_1
  -6 \alpha_0^2 \alpha_1
  -8 \alpha_1 d ,
\\[0.2cm]
\frac{1}{3}\alpha_2 & = &
  \sqrt{3} - 6 \alpha_0
  + \sqrt{3} \alpha_1^2
  + 4 a \alpha_1
  -6 \alpha_1^2 \alpha_0
  + 2 \sqrt{3} \alpha_0 \alpha_2
  -6 \alpha_0^2 \alpha_2
  -8 d \alpha_2 ,
\end{array}
\end{equation}
together with
\begin{equation}
\begin{array}{lcl}
\frac{1}{3} R_{\alpha;3} & = &
 -2 \alpha_1^3
 -6 \alpha_1
  + 2 \sqrt{3} \alpha_1 \alpha_2
   -12\alpha_2 \alpha_1\alpha_0
   +4a\alpha_2
    + 2 \sqrt{3} \alpha_0 R_{\alpha;3}
 -6 \alpha_0^2 R_{\alpha;3}
 -8d R_{\alpha;3}
\\[0.2cm]
&& \qquad
 + \Big( \sqrt{3}\alpha_2^2
  - 6 \alpha_2
  -6 \alpha_1^2 \alpha_2
 -6 \alpha_0 \alpha_2^2
 + 2 \sqrt{3} \alpha_1  R_{\alpha;3}
 -12 \alpha_0\alpha_1 R_{\alpha;3}
 +4  a R_{\alpha;3} \Big)p
\\[0.2cm]
& & \qquad
 + \Big(
  2 \sqrt{3} \alpha_2  R_{\alpha;3}
 -6 \alpha_1 \alpha_2^2
  -6 \alpha_1^2  R_{\alpha;3}
   -6  R_3
   -12 \alpha_0 \alpha_2 R_{\alpha;3}
  \Big) p^2
\\[0.2cm]
& & \qquad
 + \Big(
   \sqrt{3}  R_{\alpha;3}^2
 -2 \alpha_2^3
 -12 \alpha_2  \alpha_1 R_{\alpha;3}
 -6 \alpha_0 R_{\alpha;3}^2 \Big) p^3
\\[0.2cm]
& & \qquad
 + \Big(
 -6 \alpha_2^2  R_{\alpha;3}
 -6\alpha_1 R_{\alpha;3}^2 \Big) p^4
 -6 \alpha_2 R_{\alpha;3}^2 p^5
 - 2  R_{\alpha;3}^3 p^6 .
\\[0.2cm]
\end{array}
\end{equation}
These fixed point problems can be successively solved
for small $a$, $d$ and $p$,
which yields the desired
estimates.
\end{proof}

In order to eliminate $q$ from $h_1$,
we define the function
\begin{equation}
\mathcal{J}(p; a, d)
= \tilde{h}_1(p, q_*(p ; a, d), a ,d )
\end{equation}
and obtain the following
representation.
\begin{cor}
Pick $\delta > 0$ sufficiently small.
Then there exist a constant
$K \ge 1$, together with functions
\begin{equation}
\beta_0: (-\delta, \delta)^2 \to \Real,
\qquad
\beta_1: (-\delta, \delta)^2 \to \Real,
\qquad
\beta_2: (-\delta, \delta)^2 \to \Real,
\qquad
R_{\beta;3} : (-\delta, \delta)^3 \to \Real
\end{equation}
that satisfy the following properties.
\begin{itemize}
\item[(i)]{
For every $(p, a, d) \in (-\delta,\delta)^3$
we have
\begin{equation}
\label{eq:bif:cp:exp:j}
\mathcal{J}(p; a, d)
 =
 \beta_0(a, d)
 + \beta_1(a, d) p
 + \beta_2(a, d) p^2
 + p^3 [ 16 + R_{\beta;3}(p; a, d) ].
\end{equation}
}
\item[(ii)]{
  Upon writing
  \begin{equation}
    \begin{array}{lcl}
    \beta_0(a,d) & = &
        - \frac{1}{3} a
        + S_{\beta_0}(a, d),
    \\[0.2cm]
    \beta_1(a,d) & = &
        24d
         + S_{\beta_1}(a,d),
    \\[0.2cm]
    \beta_2(a,d) & = &
       -16a + S_{\beta_2}(a,d),
    \\[0.2cm]
    \end{array}
  \end{equation}
  the bounds
  \begin{equation}
    \begin{array}{lclclcl}
      \abs{S_{\beta_0}(a, d) }
        & \le &  K \abs{a} \abs{d}  ,
        & &
          \abs{ D S_{\alpha_0}(a,d) }
           & \le & K (\abs{a} + \abs{d}) ,
      \\[0.2cm]
      \abs{S_{\beta_1}(a, d) }
         & \le & K (a^2 + d^2) ,
         & &
          \abs{ D S_{\alpha_1}(a,d) }
           & \le & K (\abs{a} + \abs{d}) ,
      \\[0.2cm]
      \abs{S_{\beta_2}(a, d) }
        & \le &  K \abs{a}( \abs{d} + a^2) ,
        & &
          \abs{ D S_{\alpha_2}(a,d) }
           & \le & K ( \abs{a} + \abs{d})
      \\[0.2cm]
    \end{array}
  \end{equation}
  hold for all $(a, d) \in (-\delta, \delta)^2$.
}
\item[(iii)]{
  For every
  $(p, a, d) \in (-\delta, \delta)^3$
  we have the bounds
  \begin{equation}
    \begin{array}{lcl}
    \abs{R_{\beta;3}(p;a,d)} & \le &
      K ( \abs{d} + a^2 + \abs{a} \abs{p} + p^2 ) ,
    \\[0.2cm]
    \abs{R'_{\beta;3}(p;a,d)} & \le & K ( \abs{a}
        + \abs{p} ) ,
    \\[0.2cm]
    \abs{D_{a,d}R_{\beta;3}(p;a,d)}
    + \abs{D_{a,d} R'_{\beta;3} (p;a,d)} & \le &
      K,
    \end{array}
  \end{equation}
  together with
  \begin{equation}
    \label{eq:bif:cp:bnd:r:i:beta:3}
    \begin{array}{lcl}
      \abs{ R^{(i)}_{\beta;3}(p; a, d)}
       + \abs{D_{ a, d} R^{(i)}_{\beta;3}(p; a, d)}
        & \le & K
     \end{array}
  \end{equation}
  for all $2 \le i \le 6$.
  %
}
\end{itemize}
\end{cor}
\begin{proof}
Substitution yields
\begin{equation}
\begin{array}{lcl}
\beta_0 & = &
 - \frac{1}{3} a
 - \frac{2}{3} \sqrt{3} a \alpha_0
 + 2 a \alpha_0^2,
\\[0.2cm]
\beta_1 & = &
 4 a \alpha_0 \alpha_1
  - \frac{2}{3} \sqrt{3} a \alpha_1
  -6 \alpha_0^2 + 2 \sqrt{3} \alpha_0,
\\[0.2cm]
\beta_2 & = &
  4 a \alpha_0 \alpha_2
  - \frac{2}{3} \sqrt{3} a \alpha_2
  + 2 a - 12 \alpha_0 \alpha_1
  + 2 \sqrt{3} \alpha_1
  + 2 a \alpha_1^2,
\\[0.2cm]
\end{array}
\end{equation}
together with
\begin{equation}
\begin{array}{lcl}
R_{\beta;3} & = &
  4 a \alpha_1 \alpha_2
  -6 \alpha_1^2
  -12 \alpha_0 \alpha_2
  + 2 \sqrt{3} (\alpha_2 - 3 \sqrt{3} )
  - \frac{2}{3} \sqrt{3} a R_{\alpha;3}
  + 4 a \alpha_0 R_{\alpha;3}
\\[0.2cm]
& & \qquad
+\Big(  2 a \alpha_2^2
-12 \alpha_1 \alpha_2
+ 2 \sqrt{3}  R_{\alpha;3}
+ 4 a \alpha_1 R_{\alpha;3}
-12 \alpha_0 R_{\alpha;3} \Big) p
\\[0.2cm]
& & \qquad
+ \Big(
4 a \alpha_2 R_{\alpha;3}
-6 \alpha_2^2
-12 \alpha_1 R_{\alpha;3} \Big) p^2
\\[0.2cm]
& & \qquad
+ \Big( 2 a R_{\alpha;3}^2
- 12 \alpha_2 R_{\alpha;3} \Big)p^3
-6 R_{\alpha;3}^2 p^4 .
\end{array}
\end{equation}
The desired estimates can be determined
directly from these expressions.
\end{proof}

By applying a (small) shift to $p$
the undesired quadratic term
in \sref{eq:bif:cp:exp:j}
can be eliminated. The bifurcation curve
in $(a,d)$ space can subsequently be found
by determing when the remaining equation
has  roots of order two or higher.

\begin{lem}
Pick $\delta > 0$ sufficiently small.
Then there exists a constant $K \ge 1$
together with functions
\begin{equation}
\zeta_0: (-\delta, \delta)^2 \to \Real,
\qquad
\zeta_1: (-\delta, \delta)^2 \to \Real,
\qquad
\zeta_3: (-\delta, \delta)^2 \to \Real
\end{equation}
and
\begin{equation}
p_*: (-\delta, \delta)^2 \to \Real,
\qquad
R_{\zeta;4} : (-\delta, \delta)^3 \to \Real
\end{equation}
that satisfy the following properties.
\begin{itemize}
\item[(i)]{
  For every $(\tilde{p} , a, d) \in (-\delta,\delta)^3$,
  we have the identity
  \begin{equation}
     \mathcal{J}(p_*(a, d) + \tilde{p} ;a,d)
      =    \zeta_0(a , d)
        + \zeta_1(a,d) \tilde{p}
        + \zeta_3(a, d) \tilde{p}^3
      + \tilde{p}^4
           R_{\zeta;4}(\tilde{p}; a , d).
  \end{equation}
}
\item[(ii)]{
  Upon writing
  \begin{equation}
    \begin{array}{lcl}
       \zeta_0(a, d) & = &
          -\frac{1}{3} a
          + S_{\zeta_0}(a, d),
       \\[0.2cm]
       \zeta_1(a,d) & = &
       24 d + S_{\zeta_1}(a,d),
       \\[0.2cm]
       \zeta_3(a,d) & = &
         16 + S_{\zeta_3}(a,d),
    \end{array}
  \end{equation}
  the bounds
  \begin{equation}
    \begin{array}{lclclcl}
      \abs{
        S_{\zeta_0}(a, d)
        } & \le & K (\abs{a} \abs{d} + \abs{a}^3) ,
        & &
          \abs{D S_{\zeta_0}(a,d) }
           & \le & K ( \abs{a} + \abs{d} ) ,
      \\[0.2cm]
        \abs{
        S_{\zeta_1}(a, d)
        } & \le & K (a^2 + d^2 ) ,
        & &
          \abs{D S_{\zeta_1}(a,d) }
           & \le & K ( \abs{a} + \abs{d} ) ,
      \\[0.2cm]
      \abs{
        S_{\zeta_3}(a, d)
        } & \le & K ( \abs{a} + \abs{d} ) ,
        & &
          \abs{D S_{\zeta_3}(a,d) }
           & \le & K
    \end{array}
  \end{equation}
  hold for all $(a, d) \in (-\delta, \delta)^2$.
}
\item[(iii)]{
  For every $(\tilde{p}, a, d) \in (-\delta, \delta)^3$
  we have the bounds
  \begin{equation}
    \begin{array}{lcl}
      \abs{ R_{\zeta;4}(\tilde{p} ; a, d) }
      + \abs{ R_{\zeta;4}'(\tilde{p} ; a, d) }
      + \abs{ R_{\zeta;4}''(\tilde{p} ; a, d) }
        & \le & K ,
      \\[0.2cm]
      \abs{ D_{a,d}R_{\zeta;4}(\tilde{p} ; a, d) }
      + \abs{ D_{a,d}R_{\zeta;4}'(\tilde{p} ; a, d) }
        & \le & K .
      \\[0.2cm]
    \end{array}
  \end{equation}
}
\end{itemize}
\end{lem}
\begin{proof}
We first introduce the notation
\begin{equation}
\mathcal{N}_{\beta;3;p_*}(\tilde{p};a,d)
= \tilde{p}^{-3}\Big[
  R_{\beta;3}(p_* + \tilde{p};a,d)
  - R_{\beta;3}(p_*; a, d)
  - R'_{\beta;3}(p_*;a, d) \tilde{p}
  - \frac{1}{2} R''_{\beta;3}(p_*; a, d) \tilde{p}^2
  \Big]
\end{equation}
for $\tilde{p} \neq 0$,
with $\mathcal{N}_{\beta;3;p_*}(0;a,d) =
\frac{1}{6} R'''_{\beta;3}(p_*; a, d)$.
This allows us to compute
\begin{equation}
\begin{array}{lcl}
\mathcal{J}(p_* + \tilde{p} ;
a,d) & = & \gamma_0(a,d,p_*)
  + \gamma_1(a, d, p_*) \tilde{p}
  + \gamma_2(a, d, p_*) \tilde{p}^2
\\[0.2cm]
& & \qquad
  + \big[16 + R_{\gamma;3}(\tilde{p}; a, d, p_*) \big]
       \tilde{p}^3 ,
\end{array}
\end{equation}
in which
\begin{equation}
\begin{array}{lcl}
\gamma_0(a,d,p_*) & = &
 \beta_0(a,d)
 + \beta_1(a,d)p_*(a,d)
 + \beta_2(a,d) p_*(a,d)^2
\\[0.2cm]
& & \qquad
  +p_*^3 \big[ 16 + R_{\beta;3}(p_*; a, d) \big] ,
\\[0.2cm]
\gamma_1(a,d,p_*) & = &
 \beta_1(a,d)
 + 2\beta_2(a,d)p_*
\\[0.2cm]
& & \qquad
 + p_*^3 R_{\beta;3}'(p_*; a, d)
  + 3 p_*^2 \big[16 + R_{\beta;3}(p_*;a,d) \big] ,
\\[0.2cm]
\gamma_2(a, d,p_*) & = &
  \beta_2(a,d)
  +\frac{1}{2} p_*^3 R_{\beta;3}''(p_*;a,d)
  + 3 p_*^2 R_{\beta;3}'(p_*;a,d)
  + 3 p_* \big[ 16 + R_{\beta;3}(p_*;a,d) \big] ,
\\[0.2cm]
\end{array}
\end{equation}
together with
\begin{equation}
\begin{array}{lcl}
R_{\gamma;3}(\tilde{p}; a,d,p_*)
& = &
(p_* + \tilde{p} )^3 
  \mathcal{N}_{\beta;3;p_*}(\tilde{p} ; a, d)
+ \frac{1}{2} (3 p_*^2 + 3 p_* \tilde{p} + \tilde{p}^2)
       R_{\beta;3}''(p_* ; a, d)
\\[0.2cm]
& & \qquad
+ (3p_* + \tilde{p} ) R_{\beta;3}'(p_*;a,d)
+ R_{\beta;3}(p_*;a,d) .
\end{array}
\end{equation}
On account of \sref{eq:bif:cp:bnd:r:i:beta:3}
we have
\begin{equation}
\abs{\mathcal{N}_{\beta;3;p_*}(\tilde{p};a,d)}
+
\abs{\mathcal{N}'_{\beta;3;p_*}(\tilde{p};a,d)}
+\abs{\mathcal{N}''_{\beta;3;p_*}(\tilde{p};a,d)}
+ \abs{\mathcal{N}'''_{\beta;3;p_*}(\tilde{p};a,d)}
\le C_1'
\end{equation}
and hence also
\begin{equation}
\label{eq:bif:cusp:bnds:r:gamma:3}
\abs{\mathcal{R}_{\gamma;3}(\tilde{p} ; a , d,p_*) }
+
\abs{\mathcal{R}'_{\gamma;3}(\tilde{p} ; a , d,p_*) }
+ \abs{\mathcal{R}''_{\gamma;3}(\tilde{p} ; a , d,p_*) }
+ \abs{\mathcal{R}'''_{\gamma;3}(\tilde{p} ; a , d,p_*) }
\le C_2'
\end{equation}
for small $\abs{\tilde{p}}$.

Setting $\gamma_2(a,d,p_*) = 0$
leads to a fixed-point equation for $p_*$,
which has a unique small solution
$p_*(a,d)$ that admits the bounds
\begin{equation}
\abs{p_*(a,d) } \le C_3' \abs{\beta_2(a,d)} \le C_4' a ,
\qquad
\abs{ D_{a,d} p_*(a,d) } \le C_4'.
\end{equation}
The results hence follow by
writing
\begin{equation}
\begin{array}{lcl}
\zeta_0(a,d) & = &
 \gamma_0\big( a, d, p_*(a,d) \big) ,
\\[0.2cm]
\zeta_1(a, d) & = &
 \gamma_1\big( a, d, p_*(a, d) \big) ,
\\[0.2cm]
\zeta_3(a, d) & = &
  16 + R_{\gamma;3}\big(0 ;a,d,p_*(a,d) \big) ,
\\[0.2cm]
\end{array}
\end{equation}
together with
\begin{equation}
\begin{array}{lcl}
R_{\zeta;4}(\tilde{p};a,d)
 & = &  \tilde{p}^{-1}
             \big[ R_{\gamma;3}\big(\tilde{p} ; a, d, p_*(a, d)\big)
                   - R_{\gamma;3}\big(0 ; a, d, p_*(a, d)\big)
             \big]
%
\end{array}
\end{equation}
for $\tilde{p} \neq 0$ and
\begin{equation}
R_{\zeta;4}(0;a,d) = R'_{\gamma;3}\big(0 ; a ,d , p_*(a,d) \big) .
\end{equation}
Here we use \sref{eq:bif:cusp:bnds:r:gamma:3}
to get bounds on $R_{\zeta;4}$ and its derivatives.
%
\end{proof}

\begin{proof}[Proof of Proposition \ref{prp:bif:cusp:exp}]
In order to find roots of order two or higher
we need to solve the system
\begin{equation}
\label{eq:bif:cp:2nd:order:eqn:to:solve}
\begin{array}{lcl}
\zeta_0(a, d)
 + \zeta_1(a,d)\tilde{p}
 + \zeta_3(a,d)\tilde{p}^3
 + \tilde{p}^4 R_{\zeta;4}(\tilde{p}, a, d)
& = & 0 ,
\\[0.2cm]
\zeta_1(a, d)
+ 3 \zeta_3(a, d) \tilde{p}^2
+ 4 \tilde{p}^3 R_{\zeta;4}(\tilde{p}, a, d)
+ \tilde{p}^4 R'_{\zeta;4}(\tilde{p}, a, d)
& = & 0  .
\end{array}
\end{equation}
Solving the second equation we
find two branches
\begin{equation}
\tilde{p}_\pm(a,d)
 = \pm \sqrt{ -d /2 } + S_{\tilde{p}_{\pm}}(a, d) ,
\end{equation}
in which we have
\begin{equation}
\begin{array}{lcl}
S_{\tilde{p}_{\pm} }(a, d)
  & \le & C_1' (\abs{a} + \abs{d} ) ,
\\[0.2cm]
D_{a,d} S_{\tilde{p}_{\pm}}(a, d) & \le & C_1' .
\\[0.2cm]
\end{array}
\end{equation}
Plugging this into the first line of \sref{eq:bif:cp:2nd:order:eqn:to:solve} we obtain
\begin{equation}
\frac{1}{3} a
=
 \pm 16 d \sqrt{-d/2}
 + O( d^2 + \abs{ad} + \abs{a}^3 + a^2 \sqrt{\abs{d}} ),
\end{equation}
from which the desired expression for $a_c$ follows.
\end{proof}

\subsection{Geometry of the cubic}
\label{sec:bif:geom}

Our strategy to establish
Propositions
\ref{prp:st:dpm:ex}-\ref{prp:st:uv:ustb:ex}
hinges upon geometric
properties of the cubic $g(u;a)$.
As a preparation,
we introduce the notation
\begin{equation}
u_{\mathrm{infl}}(a) = \frac{1}{3}(a + 1 )
\end{equation}
together with
\begin{equation}
u_{\mathrm{min}}(a) =
u_{\mathrm{infl}}(a) - \frac{1}{3} \sqrt{1 - a(1 - a) } ,
\qquad
u_{\mathrm{max}}(a) = u_{\mathrm{infl}}(a) +  \frac{1}{3} \sqrt{1 - a(1 - a)}
\end{equation}
and note that
\begin{equation}
g'\big(u_{\mathrm{min}}(a);a\big) = 0,
\qquad
g''\big(u_{\mathrm{infl}}(a);a \big) = 0,
\qquad
g'\big(u_{\mathrm{max}}(a) ; a \big) = 0.
\end{equation}
In addition, we note that $g''(u;a) > 0$
for $u < u_{\mathrm{infl}}(a)$ and $g''(u;a) < 0$ for $u > u_{\mathrm{infl}}(a)$.
When $0 < a < \frac{1}{2}$ we have the ordering
\begin{equation}
0 < u_{\mathrm{min}}(a) < a < u_{\mathrm{infl}}(a) < u_{\mathrm{max}}(a) < 1.
\end{equation}

\begin{figure}
\begin{center}
\begin{minipage}{0.45\textwidth}
\includegraphics[width=\textwidth]{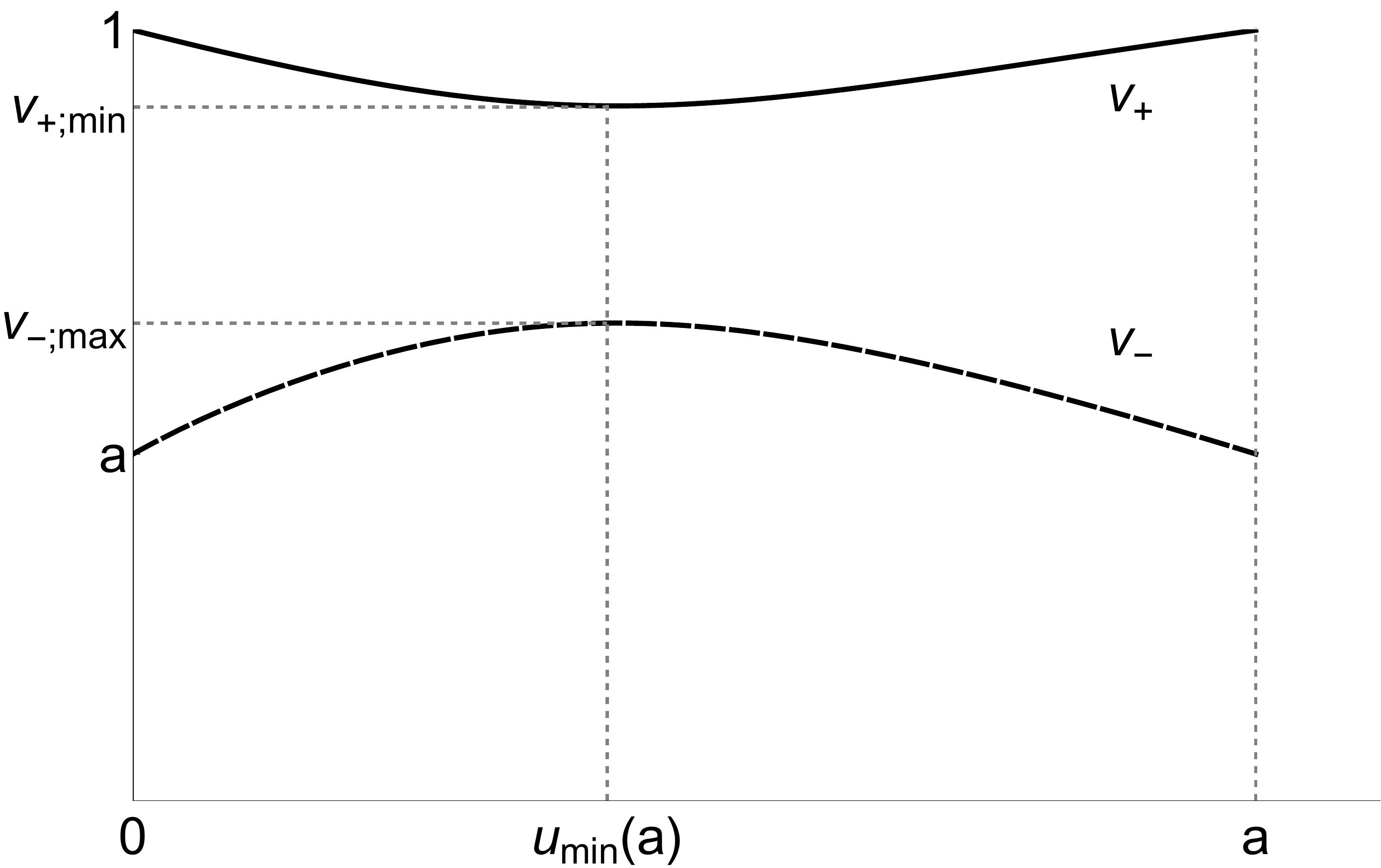}
\end{minipage}\quad
\begin{minipage}{0.45\textwidth}
\includegraphics[width=\textwidth]{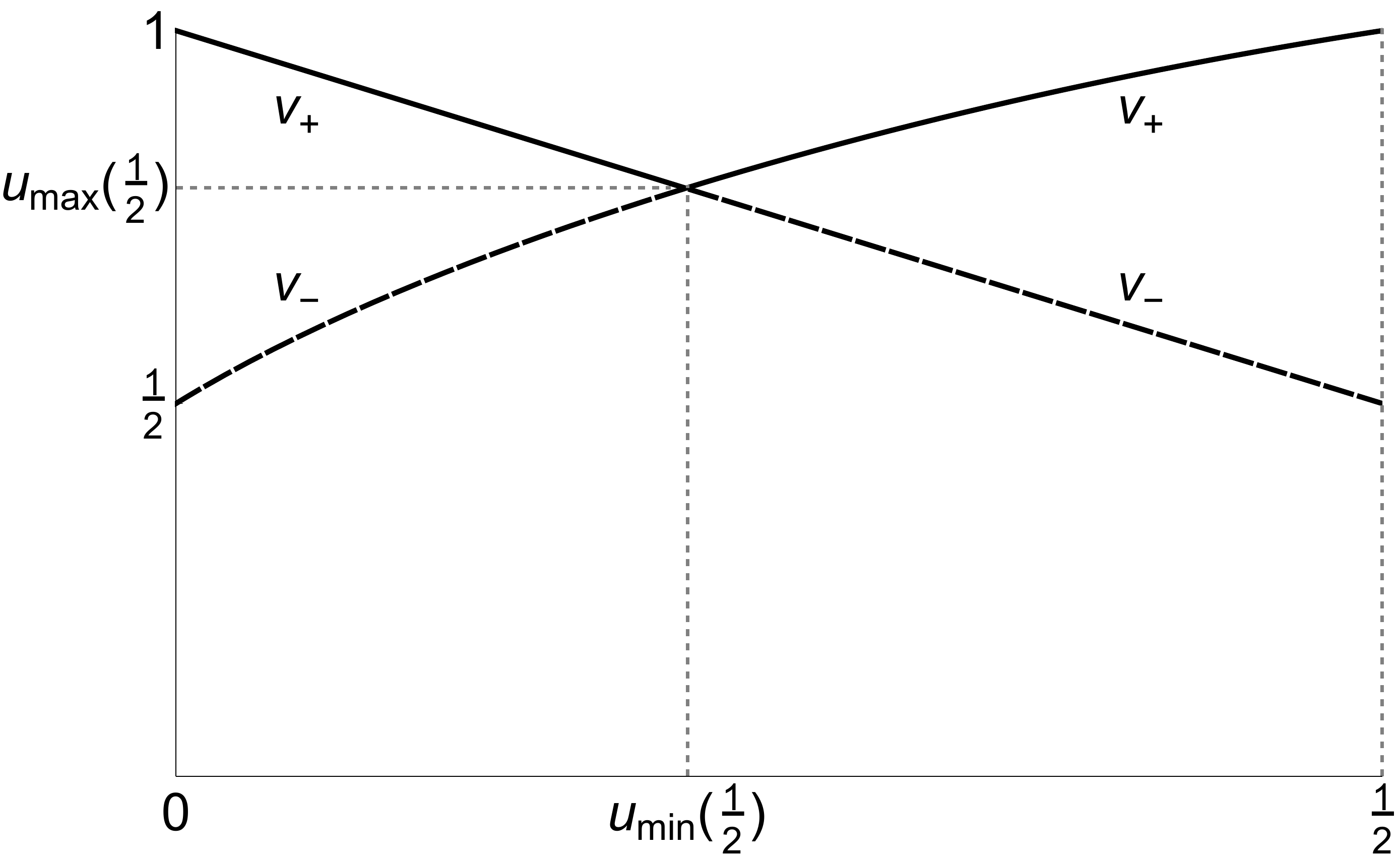}
\end{minipage}
\caption{The functions $v_\pm$ defined in Lemmata
\ref{lem:bif:setup:bs:props:vpm}-\ref{lem:bif:setup:bs:props:vpm:a:half}
for $a = 0.45$ (left) and $a = 0.5$ (right). In the latter case the derivatives
have a discontinuity at $u_{\mathrm{min}}(\frac{1}{2})$.
}\label{f:vpvm}
\end{center}
\end{figure}

Recalling the definition
\sref{eq:mr:def:G:G12},
one readily sees that
\begin{equation}
G_1(u, v; a, d) + G_2(u, v ; a, d)
 = g(u;a) + g(v;a).
\end{equation}
In order to exploit the fact that this
identity does not
depend on $d$,
we first use basic properties of the cubic
to parametrize solutions to
$g(u;a) + g(v ;a) = 0$.
Restricting ourselves to $a \in [0, \frac{1}{2}]$,
it is possible to construct
two solution curves $v_\pm(u)$
that are defined for
$u \in [0, a]$; see Figure \ref{f:vpvm}.

\begin{lem}
\label{lem:bif:setup:bs:props:vpm}
Fix $0 < a < \frac{1}{2}$. Then there are two constants
\begin{equation}
a < v_{-;\mathrm{max}} < v_{+;\mathrm{min}} < 1
\end{equation}
together with two
$C^\infty$-smooth
functions
\begin{equation}
v_-: [0, a] \to [a ,v_{-;\mathrm{max}}],
\qquad
v_+ : [0, a] \to [v_{+;\mathrm{min}},1]
\end{equation}
that satisfy the following properties.
\begin{itemize}
\item[(i)]{
  We have $g(u;a) = - g(v_-(u);a) = - g(v_+(u);a)$ for all $0 \le u \le a$.
}
\item[(ii)]{
  If $g(u;a) = - g(v;a)$ for some pair $u \in (0 ,a)$ and $v \in [0,1]$,
  then $v = v_-(u)$ or $v = v_+(u)$.
}
\item[(iii)]{
  We have $v_-(0) = v_-(a) = a$ and $v_+(0) = v_+(a) = 1$.
}
\item[(iv)]{
  We have the identities
  \begin{equation}
    \label{eq:bif:prlm:deriv:vpm}
    v_\pm'(u) = -[ g'( v_\pm(u);a )]^{-1} g'(u ;a)
  \end{equation}
  for all $0 \le u \le a$.
}
\item[(v)]{
  We have $v_-'(u) > -1 $ for all $0 \le u < a$,
  together with $v_-'(a) = -1$.
}
\end{itemize}
\end{lem}
\begin{proof}
Since $0 < a < \frac{1}{2}$ we
have $g(u_{\mathrm{max}}(a);a) > - g(u_{\mathrm{min}}(a);a)$,
which implies that $v_\pm(u) \neq u_{\mathrm{max}}(a)$.
Properties (i)-(iv) hence follow immediately
from the implicit function theorem.

To obtain (v), we take $u_{\mathrm{min}}(a) <u < a$
and recall $a < v_-(u) < u_{\mathrm{max}}$.
If $v_-(u) \le u_{\mathrm{infl}}(a)$ then clearly
$g'(v_-(u);a) > g'(u;a) > 0$, as desired.
In order to hande the remaining case $v_-(u) > u_{\mathrm{infl}}(a)$,
we introduce the reflection
$u_{\mathrm{refl}} = 2 u_{\mathrm{infl}}(a) - u$.
Exploiting the point symmetry of the graph of $g(\cdot ;a)$
around its inflection
point $\big(u_{\mathrm{infl}}(a), g(u_{\mathrm{infl}}(a) ;a) \big)$,
the inequality $g(u_{\mathrm{infl}}(a) ;a) > 0$ implies that
\begin{equation}
u_{\mathrm{infl}}(a) < v_-(u) < u_{\mathrm{refl}}.
\end{equation}
Since $g''(\tilde{u};a) < 0$ for $\tilde{u} > u_{\mathrm{infl}}(a)$,
we obtain
\begin{equation}
g'\big(v_-(u) ;a \big) > g'( u_{\mathrm{refl}} ;a ) = g'( u;a) > 0,
\end{equation}
which implies $v_-'(u) > -1$.
\end{proof}

\begin{lem}
\label{lem:bif:setup:bs:props:vpm:a:half}
Fix $a =  \frac{1}{2}$. Then there are two
functions
\begin{equation}
v_-: [0, a] \to [a ,u_{\mathrm{max}}(a)],
\qquad
v_+ : [0, a] \to [u_{\mathrm{max}}(a),1]
\end{equation}
that satisfy items (i) - (iii)
from Lemma \ref{lem:bif:setup:bs:props:vpm}
together with the
following additional properties.
\begin{itemize}
\item[(i)]{
  We have
  \begin{equation}
    v_-\big(u_{\mathrm{min}}(a) \big) = v_+\big(u_{\mathrm{min}}(a)\big)
     = u_{\mathrm{max}}(a).
  \end{equation}
}
\item[(ii)]{
  For any $u \in [0, a] \setminus \{ u_{\mathrm{min}}(a) \}$
  we have the identities
  \begin{equation}
    \label{eq:bif:prlm:deriv:vpm:a:half}
    v_\pm'(u) = -[ g'( v_\pm(u) ;a )]^{-1} g'(u ;a).
  \end{equation}
}
\item[(iii)]{
  We have $v_-(u) = 1 - u$ for all $ u_{\mathrm{min}}(a) \le u \le 1$,
  while $v_+(u) = 1 - u$ for all $0 \le u \le u_{\mathrm{min}}(a)$.
}
\item[(iv)]{
  We have the limits
  \begin{equation}
    \lim_{u \uparrow u_{\mathrm{min}}(a)} v_-'(u) =
     \lim_{u \downarrow u_{\mathrm{min}}(a) } v_+'(u) =  + 1,
  \end{equation}
  together with
  \begin{equation}
    \lim_{u \uparrow u_{\mathrm{min}}(a)} v_-''(u) =
      \lim_{u \downarrow u_{\mathrm{min}}(a)} v_+''(u)
      = - \frac{4}{3} \sqrt{3}.
  \end{equation}
}
\end{itemize}
\end{lem}
\begin{proof}
Items (i) and (iii) follow
directly from the symmetry of $g(\cdot; a)$,
while (ii) follows from the implicit function
theorem. To obtain (iv),
we first compute
\begin{equation}
g(u_{\mathrm{min}} + u ;a)
= g(u_{\mathrm{min}};a) + \frac{1}{2} \sqrt{3} u^2
 - u^3,
\qquad
g(u_{\mathrm{max}} +v ;a)
= - g(u_{\mathrm{min}} ;a) - \frac{1}{2} \sqrt{3} v^2
- v^3.
\end{equation}
In particular, the identity
\begin{equation}
g(u_{\mathrm{min}}+ u; a) = - g(u_{\mathrm{max}} + v ; a)
\end{equation}
can be rewritten as
\begin{equation}
v^2
 =
u^2
- \frac{2}{3} \sqrt{3} \big[
    u^3 + v^3
  \big] ,
\end{equation}
which can be interpreted
as a fixed point problem for $v^2$
upon assuming that $v$ and $u$
have the same sign.
For small $\abs{u}$ this problem
has a solution
that can be expanded as
\begin{equation}
v^2 = u^2 - \frac{4}{3} \sqrt{3} u^3
  + O (u^4) ,
\end{equation}
which yields
\begin{equation}
v = u - \frac{2}{3} \sqrt{3} u^2
 + O(u^3).
\end{equation}
The desired limits follow directly
from this expansion.
\end{proof}

\subsection{Tangencies}
\label{sec:bif:tgn}

Let us now fix $a \in [0, \frac{1}{2}]$.
In order to find solutions
to $G(u,v ;a , d)= 0$ with $d > 0$,
we introduce the function
\begin{equation}
v_d(u) = u - \frac{g(u ;a)}{2d}
\end{equation}
and note that the results above
show that it
suffices to find $u \in [0,a]$
for which one of the equations
\begin{equation}
\label{eq:bif:id:v:to:solve}
v_{\pm}(u) = v_d(u)
\end{equation}
holds.

Our goal here is to show that
non-transverse intersections
of this type can only
occur at local minima
of $v_\pm - v_d$.
Together with the strict monotonicity
\begin{equation}
\partial_d v_d(u) < 0
\end{equation}
that holds for $u \in (0, a)$,
this will allow us to obtain global results
in \S\ref{sec:bif:glb}.

For the moment, we note that
\begin{equation}
\label{eq:bif:tng:a}
v_-(a) = v_d(a) = a
\end{equation}
for every $d> 0$.
In addition, we may compute
\begin{equation}
v_d'(a) = 1 - \frac{g'(a;a)}{2d} ,
\end{equation}
together with
\begin{equation}
v_-'(a) = - [g'(a;a)]^{-1} g'(a;a) = -1.
\end{equation}
In particular, when $g'(a) = 4d$ the intersection
\sref{eq:bif:tng:a} is tangential. In the sequel we show that
in fact $v_- > v_d$ on $[0, a)$ for this critical value of $d$.

For intersections with $u \in (0, a)$ such
explicit computations are significantly harder to carry out,
which is why we pursue a more indirect approach here.
As a preparation, we compute
\begin{equation}
\label{eq:preps:sec:deriv:v:pm}
v''_\pm(u) = - \frac{g''(u)}{g'\big(v_\pm(u)\big)} - \frac{g'(u)^2}{g'\big(v_\pm(u)\big)^3} g''\big(v_\pm(u)\big)
\end{equation}
together with
\begin{equation}
\label{eq:preps:sc:deriv:v:star}
v_d''(u) =- \frac{g''(u)}{2d}.
\end{equation}
In addition,
for any $\kappa \le g'(u_{\mathrm{infl}}(a);a) =  \frac{1}{3}(a^2 - a + 1)$,
we introduce the expressions
\begin{equation}
\begin{array}{lcl}
u_{l}(\kappa) & = &  u_{\mathrm{infl}}(a) - \sqrt{\frac{1}{3}(g'(u_{\mathrm{infl}}(a);a) - \kappa)} ,
\\[0.2cm]
u_r(\kappa) & = & u_{\mathrm{infl}}(a) + \sqrt{\frac{1}{3}(g'(u_{\mathrm{infl}}(a);a) - \kappa)} .
\end{array}
\end{equation}
It is easy to verify that $u_l(\kappa)$ and $u_r(\kappa)$
are the two solutions to the
quadratic equation $g'(u;a) = \kappa$.

\begin{lem}
\label{lem:bif:trnvs:basic:props:ints}
Fix $0 < a \le \frac{1}{2}$ and $d > 0$
and suppose that
\begin{equation}
v_{\#}'(u) = v_d'(u) = \beta
\end{equation}
for some $\# \in \{-, +\}$ and $0 \le u \le a$,
with $u \neq u_{\mathrm{min}}(a)$ in case $a = \frac{1}{2}$.
Then the following statements hold.
\begin{itemize}
\item[(i)]{
We have
\begin{equation}
\label{eq:bif:tgn:id:for:delta:sec}
2d[v_{\#}''(u) - v_d''(u) ]
= \frac{1}{1 - \beta} g''(u;a)
  + \frac{\beta^3}{1 - \beta} g''\big(v_{\#}(u) ;a\big).
\end{equation}
}
\item[(ii)]{
If $\# = -$ then we have $ \beta \in [-1, 0) \cup (1, \infty)$.
On the other hand, the inclusion $\beta \in (0, 1)$ holds  if $\# = +$.
}
\item[(iii)]{
  We have $u = u_l\big( 2d ( 1 - \beta) \big)$.
}
\item[(iv)]{
  Suppose that $\# = -$. Then the identity
  \begin{equation}
    \label{eq:bif:tgn:id:for:v:l}
    v_{\#}(u) = u_l\big( 2d (1 - \beta^{-1}) \big)
  \end{equation}
  holds if $v_-(u) \le u_{\mathrm{infl}}(a)$. On the other hand,
  we have
  \begin{equation}
    \label{eq:bif:tgn:id:for:v:r}
    v_{\#}(u) = u_r\big(2d(1 - \beta^{-1}) \big)
  \end{equation}
  if $v_-(u) > u_{\mathrm{infl}}(a)$.
}
\item[(v)]{
  The identity
  \begin{equation}
    v_{\#}(u) = u_r\big( 2d (1 - \beta^{-1}) \big)
  \end{equation}
  holds if $\# =  +$.
}
\end{itemize}
\end{lem}
\begin{proof}
We first consider the case $0 < a< \frac{1}{2}$.
For any $0 \le \tilde{u} \le a$,
one sees that $v_d'(\tilde{u}) = 1$
holds if and only if $g'(\tilde{u};a) = v_\pm'(\tilde{u}) = 0$,
which shows that $\beta \notin \{0 ,1\}$.

For any $0 \le \tilde{u} \le a$
we have $g'(v_+(\tilde{u});a) < 0$ and hence
\begin{equation}
\mathrm{sign} \big[ v'_+(\tilde{u}) \big] = \mathrm{sign} \big[ g'(\tilde{u} ;a) \big].
\end{equation}
If $\# = +$, this shows that $g'(u;a) > 0$ and $v'_d(u) > 0$,
hence $\beta \in (0, 1)$.

On the other hand, we have $g'(v_-(\tilde{u}) ; a) > 0$ for all
$0 \le \tilde{u} \le a$. If $g'(u;a) < 0$ then we have $v_d'(u) = \beta > 1$,
while if $g'(u;a) > 0$ and $\# = -$ we may use item (v)
from Lemma \ref{lem:bif:setup:bs:props:vpm} to conclude
\begin{equation}
  -1 \le v_-'(u) = \beta < 0,
\end{equation}
which establishes (ii).

In order to obtain (iii),
it suffices to recall the bound
$u \le a \le u_{\mathrm{infl}}(a)$
and note that the identity $v_*'(u) = \beta$ implies that
\begin{equation}
g'(u) = 2d ( 1 - \beta).
\end{equation}
On the other hand, the identity $v_\pm'(u) = \beta$ implies
\begin{equation}
g'\big(v_\pm(u) \big) = -2d \beta^{-1} ( 1 - \beta).
\end{equation}
Items (iv) and (v) now follow directly,
remembering that $v_+(u) \ge u_{\mathrm{max}}(a) \ge u_{\mathrm{infl}}(a)$.

Exploiting \sref{eq:preps:sec:deriv:v:pm} we may now compute
\begin{equation}
\begin{array}{lcl}
v''_\pm(u) & = & \frac{\beta}{2d(1 - \beta)} g''(u;a)
+ \frac{(2d)^2 (1 - \beta)^2 \beta^3}{(2d)^3 (1 - \beta)^3}
   g''(v_\pm(u);a )
\\[0.2cm]
& = &
\frac{\beta}{2d(1 - \beta)} g''(u;a)
+ \frac{ \beta^3}{ 2d (1 - \beta)} g''(v_\pm(u);a).
\end{array}
\end{equation}
The desired identity in (i) hence follows
directly from \sref{eq:preps:sc:deriv:v:star}.
In order to conclude the proof,
it suffices to note that
the arguments above remain valid
when $a = \frac{1}{2}$.
Indeed,  the critical cases $g'(u;a) = 0$
and $g'(v_\pm;a) = 0$ are excluded
by the requirement that $u \neq u_{\mathrm{min}}(a)$.

\end{proof}

\begin{lem}
\label{lem:bif:trnvs:vminus:large:u}
Fix $0 < a \le \frac{1}{2}$
and suppose that
\begin{equation}
v'_-(u) = v_d'(u)
\end{equation}
for some $u_{\mathrm{min}}(a) \le u < a$ and $d > 0$,
with $u \neq u_{\mathrm{min}}(a)$ if $a = \frac{1}{2}$.
Then we have
\begin{equation}
v_-''(u) > v_d''(u).
\end{equation}
\end{lem}
\begin{proof}
Exploiting (v) of Lemma
\ref{lem:bif:setup:bs:props:vpm}
and (ii) of Lemma \ref{lem:bif:trnvs:basic:props:ints}
we have $ v_-'(u) \in (-1, 0)$.
In addition, the inequalities $v_-(u) > u$ and $g''' < 0$
imply that
$g''(v_-(u);a) \le g''(u;a)$.
Writing $\beta = v_-'(u)$
we may hence estimate
\begin{equation}
\begin{array}{lcl}
2d[v_{-}''(u) - v_d''(u) ]
& \ge & \frac{1}{1 - \beta} g''(u;a)
+ \frac{\beta^3}{1 - \beta} g''\big(u  ;a\big)
\\[0.2cm]
& = &
\frac{\beta^3 + 1}{1 - \beta} g''(u ;a)
\\[0.2cm]
& > & 0,
\end{array}
\end{equation}
in which we used $g''(u ;a) > 0$.
\end{proof}

Intersections with $v_\pm'(u) > 0$ are more delicate to analyze.
Items (i), (iii) and (iv) of Lemma \ref{lem:bif:trnvs:basic:props:ints}
suggest that it is worthwhile to consider the
two functions
\begin{equation}
\begin{array}{lcl}
h_l(\beta) & = &
\frac{1}{2d(1 - \beta)}
\Big[
   g''\Big(u_l\big( 2d(1 - \beta) \big) ;a \Big)
   +
   \beta^3
   g''\Big(u_l\big(  2d (1 - \beta^{-1}) \big) ; a \Big)
\Big]  ,
\\[0.2cm]
h_r(\beta)
&  = &
\frac{1}{2d(1 - \beta)}
\Big[
     g''\Big(u_l\big( 2d(1 - \beta) \big) ;a \Big)
   +
   \beta^3
   g''\Big(u_r\big(  2d (1 - \beta^{-1}) \big) ;a \Big)
\Big] .
\end{array}
\end{equation}

\begin{lem}
Pick $0 < a \le \frac{1}{2}$
and  $0 < d \le \frac{g'(a;a)}{4}$.
Then for any $\beta > 1$ the inequality
\begin{equation}
h_l(\beta) < 0
\end{equation}
holds, while for any $\beta \in (0, \infty) \setminus \{1\}$
we have
\begin{equation}
\label{eq:trnvs:hr:pos}
h_r(\beta) > 0 .
\end{equation}
\end{lem}
\begin{proof}
Observe first that for $ \beta > 0$ we have
\begin{equation}
\max\{ 2d ( 1 - \beta) ,
 2d (1 - \beta^{-1} ) \}
 \le 2d \le \frac{g'(a;a)}{2} \le \frac{g'(u_{\mathrm{infl}}(a); a)}{2},
\end{equation}
which implies that $h_l(\beta)$
and $h_r(\beta)$ are well-defined.

A little algebra yields
\begin{equation}
\begin{array}{lcl}
h_l(\beta) & = &
 \frac{\sqrt{3}}{d(1 - \beta)}
 \Big[
   \sqrt{g'(u_{\mathrm{infl}}(a); a) + 2d(\beta - 1)}
   + \beta^3 \sqrt{ g'(u_{\mathrm{infl}}(a);a)
   + 2d \beta^{-1}(1 - \beta)  }
 \Big] ,
\\[0.2cm]
h_r(\beta) & = &
 \frac{\sqrt{3}}{d(1 - \beta)}
 \Big[
   \sqrt{g'(u_{\mathrm{infl}}(a);a) + 2d(\beta - 1)}
   - \beta^3 \sqrt{ g'(u_{\mathrm{infl}}(a);a)
   + 2d \beta^{-1}(1 - \beta)  }
 \Big] .
\end{array}
\end{equation}
It is clear that $h_l(\beta) < 0$ for $\beta > 1$.
Upon writing
\begin{equation}
\begin{array}{lcl}
\Delta(\beta)
& = &
 g'(u_{\mathrm{infl}}(a);a) + 2d(\beta - 1)
- \big(
     \beta^6 g'(u_{\mathrm{infl}}(a);a) + 2d \beta^5(1 - \beta)
  \big)
\\[0.2cm]
& = &
  (1 - \beta^6) \big( g'(u_{\mathrm{infl}}(a);a) - 2d \big)
  + 2d \beta (1 - \beta^4) ,
\end{array}
\end{equation}
it is easy to verify  that $\Delta(\beta) < 0$
for $\beta > 1$
and $\Delta(\beta) > 0$ for $0 < \beta < 1$.
This yields the final inequality
\sref{eq:trnvs:hr:pos}.
\end{proof}

\begin{lem}
\label{lem:bif:tgn:option:a:b}
Pick $0 < a \le \frac{1}{2}$
and $0 < d \le \frac{g'(a;a)}{4}$.
Then we have
\begin{equation}
\label{eq:bif:tgn:deriv:at:zero}
v_-'(0) < v_d'(0).
\end{equation}
In addition, suppose that
\begin{equation}
\label{lem:bif:tgn:derivs:eq}
v_-'(u) = v_*'(u)
\end{equation}
for some $0  \le u \le a$,
with $u \neq u_{\mathrm{min}}(a)$ if $a = \frac{1}{2}$.
Then one of the following two statements must hold.
\begin{itemize}
\item[(a)]{
  We have the inequality
  \begin{equation}
    v_-''(u) > v_d''(u).
  \end{equation}
}
\item[(b)]{
  We have the identities
  \begin{equation}
    u = a, \qquad d = \frac{g'(a)}{4},
    \qquad v_-''(u) =  v_d''(u).
  \end{equation}
}
\end{itemize}
\end{lem}
\begin{proof}
An easy computation yields
\begin{equation}
v_d'(0) = 1 - \frac{g'(0 ; a)}{2d}
\ge 1 - \frac{2g'(0 ;a)}{g'(a ; a)}
> - \frac{g'(0 ;a)}{g'(a ; a)} = v_-'(0).
\end{equation}
We introduce the critical value
\begin{equation}
u_c = \mathrm{sup} \{ 0 \le u \le a : v_-(\tilde{u}) \le u_{\mathrm{infl}}(a)
\hbox{ for all } 0 \le \tilde{u} \le u \}
\end{equation}
and remark that $u_c = 0$ when $a = \frac{1}{2}$.
This allows us to define the value
\begin{equation}
u_I = \mathrm{sup}\big\{0 \le u \le \min\{ u_c , u_{\mathrm{min}}(a) \}:
  v_d'(u) > v_-'(u)
\big\},
\end{equation}
which again satisfies $u_I = 0$ when $a = \frac{1}{2}$.

We claim that also $v_d'(u_I)  > v_-'(u_I)$.
Indeed, assuming this is false we can
define $\beta = v_d'(u_I) = v_-'(u_I) \ge 0$.
Item (ii) of Lemma \ref{lem:bif:trnvs:basic:props:ints}
the implies $\beta >1$. Since $h_l(\beta) < 0$
we must have
\begin{equation}
v_-''(u_I) < v_d''(u_I) ,
\end{equation}
which yields a contradiction.

In particular, if \sref{lem:bif:tgn:derivs:eq} holds
then we must have $u \ge \min\{ u_{\mathrm{min}}(a), u_c \}$.
If $u_{\mathrm{min}}(a) \le u < a$, Lemma
\ref{lem:bif:trnvs:vminus:large:u}
shows that (a) must hold.
On the other hand,
if $u_c \le u < u_{\mathrm{min}}(a)$,
then we can define $\beta = v_d'(u) = v_-'(u) \ge 0$
and conclude as above that $\beta > 1$. In addition,
we have $v_-(u) \ge u_{\mathrm{infl}}(a)$,
which allows us to use
$h_r(\beta)  > 0$ and item (iv) of
Lemma \ref{lem:bif:trnvs:basic:props:ints}
to show that (a) must hold.
In the final case $u = a$, the remarks
at the start of this section together with a direct
computation of $v_d''(a)$ and $v_-''(a)$ imply the identities in (b).
\end{proof}

In the remainder of this section we
collect several consequences of these computations.
In each case, we either rule out
non-transverse intersections of $v_\pm$ with $v_d$
or show that they must occur at local minima
of $v_\pm - v_d$.

\begin{cor}
\label{cor:bif:trnsv:v:plus:tangency}
Fix $0 < a \le \frac{1}{2}$
together with $0 < d \le \frac{g'(a;a)}{4}$
and suppose that
\begin{equation}
v_+'(u) = v_d'(u)
\end{equation}
for some $0  \le u \le a$,
with $u \neq u_{\mathrm{min}}(a)$ if $a = \frac{1}{2}$.
Then we have
\begin{equation}
v_+''(u) > v_d''(u)  .
\end{equation}
\end{cor}
\begin{proof}
Using the fact that
$h_r(\beta ) > 0$
for $\beta \in (0, 1)$,
this follows directly
from items (i), (ii) and (v) of Lemma \ref{lem:bif:trnvs:basic:props:ints}.
\end{proof}

\begin{cor}
\label{cor:bif:donward:tangency}
Fix $0 < a \le \frac{1}{2}$
together with $0 < d < \frac{g'(a;a)}{4}$
and suppose that
\begin{equation}
v_-'(u) = v_*'(u)
\end{equation}
for some $0  \le u \le a$,
with $u\neq u_{\mathrm{min}}(a) $ if $a = \frac{1}{2}$.
Then we have
\begin{equation}
v_-''(u) > v_d''(u).
\end{equation}
In addition, we have
\begin{equation}
\label{eq:cor:vminus:subcrit:ineq:v:star:minus}
v_-'(a) > v_d'(a).
\end{equation}
\end{cor}
\begin{proof}
The first inequality follows directly from
the fact that option (b) in Lemma \ref{lem:bif:tgn:option:a:b}
cannot hold because of the restriction on $d$.
The final inequality can be verified directly
by noting that
\begin{equation}
v_d'(a) =  1 - \frac{g'(a;a)}{2d}
< 1 - \frac{2 g'(a;a)}{g'(a;a)}
= -1 = v_-'(a).
\end{equation}
\end{proof}

\begin{cor}
\label{cor:bif:tgn:crit:case:deriv:minus:smaller}
Fix $0 < a \le \frac{1}{2}$
together with $d = \frac{g'(a;a)}{4}$.
Then we have
\begin{equation}
\label{eq:cor:critical:ineq:for:v:minus:star:i}
v_-'(u) < v_d'(u)
\end{equation}
for all $0 \le u < a$,
with the exception of $u = u_{\mathrm{min}}(a)$
in the special case $a = \frac{1}{2}$.
\end{cor}
\begin{proof}
It is easy to verify that $v_-'(a) = v_d'(a)$
and $v_-''(a) = v_d''(a)$.
We also compute
\begin{equation}
v_d'''(a) = -2 \frac{g'''(a;a)}{g'(a;a)}
  = \frac{12}{a ( 1-a)} > 0
\end{equation}
together with
\begin{equation}
\begin{array}{lcl}
v_-'''(u)
& = & - \frac{g'''(u;a)}{g'(v_-(u);a)}
 - 3\frac{g''(u;a)g'(u;a)}{g'(v_-(u);a)^3}g''(v_-(u);a)
\\[0.2cm]
& & \qquad
 - 3 \frac{g'(u;a)^3}{g'(v_-(u);a)^5}
   g''(v_-(u);a)^2
 + \frac{g'(u;a)^3}{g'(v_-(u);a)^4}
    g'''(v_-(u);a),
\end{array}
\end{equation}
which gives
\begin{equation}
v'''_-(a)
= -6 \frac{g''(a;a)^2}{g'(a;a)^2 } \le 0.
\end{equation}
In particular, we see that
\begin{equation}
v_-'(a - \epsilon) < v_d'(a - \epsilon)
\end{equation}
for all sufficiently small $\epsilon > 0$.
If $a \neq \frac{1}{2}$,
the conclusion now follows
from \sref{eq:bif:tgn:deriv:at:zero}
together with (a)
from Lemma \ref{lem:bif:tgn:option:a:b}.

For $a = \frac{1}{2}$,
one also needs to use the identities
\begin{equation}
v_d'\big(u_{\mathrm{min}}(a)\big) = 1,
\qquad
v_d''\big(u_{\mathrm{min}}(a) \big) = - 12 \sqrt{3}
\end{equation}
together with the limits
in item (iv) of Lemma \ref{lem:bif:setup:bs:props:vpm:a:half}
to conclude that
\begin{equation}
v_-'(u_{\mathrm{min}}(a) - \epsilon) < v_d'( u_{\mathrm{min}}(a) - \epsilon)
\end{equation}
for all sufficiently small $\epsilon > 0$. The arguments
above allow us to extend this to $\epsilon \in (0, u_{\mathrm{min}}(a)]$.
In addition, we have
$v_-'(\tilde{u}) = -1 < v_d'(\tilde{u})$
for $u_{\mathrm{min}}(a) < \tilde{u} < a$.
\end{proof}

\begin{cor}
\label{cor:bif:trv:inc:bef:u:min}
Pick $0 < a \le \frac{1}{2}$
and $0 < d \le \frac{g'(a;a)}{4}$.
Then we have
\begin{equation}
\label{eq:cor:critical:ineq:for:v:minus:star:ii}
v_-'(u) < v_d'(u)
\end{equation}
for all $0 \le u < u_{\mathrm{min}}(a)$.
\end{cor}
\begin{proof}
Writing $d_c = \frac{g'(a;a)}{4}$,
we may use
Corollary \ref{cor:bif:tgn:crit:case:deriv:minus:smaller}
to compute
\begin{equation}
v_d'(u) = 1 - \frac{1}{2d} g'(u;a) \ge 1 - \frac{2}{g'(a;a)} g'(u;a)
 =  v_{d_c}'(u) > v_-'(u)
\end{equation}
for $u \in [0, u_{\mathrm{min}}(a))$.
\end{proof}

\begin{cor}
\label{cor:bif:trv:inc:aft:u:min}
Fix $a = \frac{1}{2}$
and $0 < d \le \frac{g'(a;a)}{4}$.
Then we have
\begin{equation}
\label{eq:cor:critical:ineq:for:v:plus:star}
v_+'(u) > v_d'(u)
\end{equation}
for all $u_{\mathrm{min}}(a) < u \le a$.
\end{cor}
\begin{proof}
Item (iv) of Lemma \ref{lem:bif:setup:bs:props:vpm:a:half}
allow us to compute
\begin{equation}
v_d'\big(u_{\mathrm{min}}(a) \big) = 1 = \lim_{u \downarrow u_{\mathrm{min}}(a) } v_+'(u)
\end{equation}
together with
\begin{equation}
v_d''\big(u_{\mathrm{min}}(a) \big) < \lim_{u \downarrow u_{\mathrm{min}}(a) } v''_+(u),
\end{equation}
which allows us to conclude that
\begin{equation}
v_+'(u_{\mathrm{min}}(a)+\epsilon) > v'_d( u_{\mathrm{min}}(a)+ \epsilon)
\end{equation}
for all sufficiently small $\epsilon > 0$.
Corollary \ref{cor:bif:trnsv:v:plus:tangency}
allows us to extend this conclusion
to the desired interval $\epsilon \in \big(0, a - u_{\mathrm{min}}(a) \big)$.
\end{proof}

\subsection{Structure}
\label{sec:bif:glb}

We are now ready to analyze the global structure
of the solution set to $G(u, v; a, d) = 0$.
Our first two results fix $a \in (0, \frac{1}{2}]$
and track the intersections of the curves
$v_\pm$ that were introduced in \S\ref{sec:bif:geom}
with the curve $v_d$ introduced in \S\ref{sec:bif:tgn}.
These intersections disappear as the parameter $d$ is increased;
see Figure \ref{f:bif:ints:vdpm}.

\begin{figure}
\begin{center}
\begin{minipage}{0.45\textwidth}
\includegraphics[width=\textwidth]{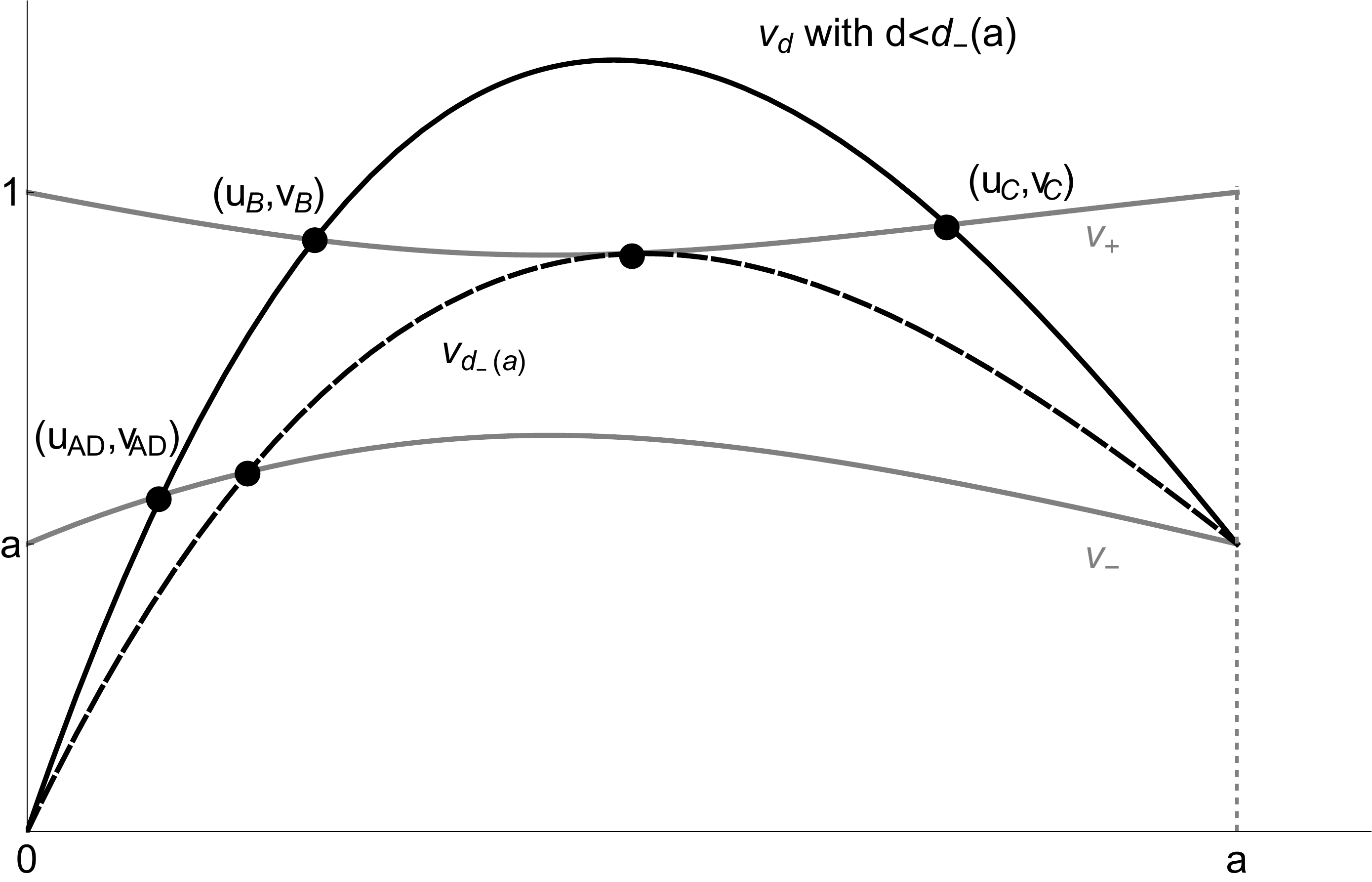}
\end{minipage}\quad
\begin{minipage}{0.45\textwidth}
\includegraphics[width=\textwidth]{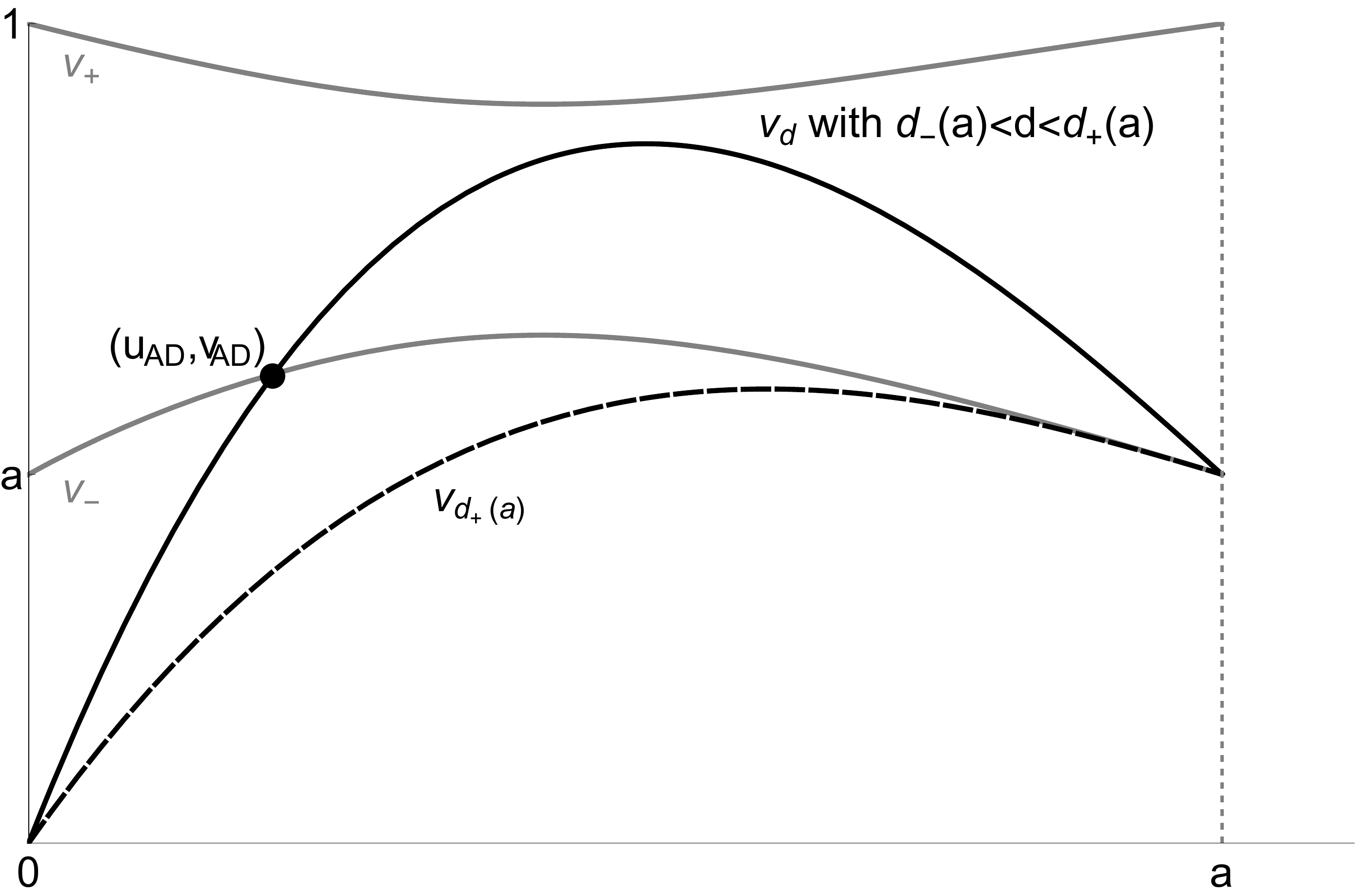}
\end{minipage}
\caption{Fix $0 < a < \frac{1}{2}$.
The branches $u_B$ and $u_C$ described in Lemma \ref{lem:bif:str:roots:BC}
arise as the two intersections of $v_+$ and $v_d$ on $[0, a]$, which collide as $d \uparrow d_-(a)$ (left).
On the other hand, the branch $u_{AD}$ described in Lemma \ref{lem:bif:str:roots:A} arises as the
unique intersection of the curves $v_-$ and $v_d$ on $[0, a)$, which converges to $a$ as $d \uparrow d_+(a)$
(right).
}\label{f:bif:ints:vdpm}
\end{center}
\end{figure}

\begin{lem}
\label{lem:bif:str:roots:A}
Fix $0 < a \le \frac{1}{2}$.
Then there exists a continuous strictly increasing function
\begin{equation}
u_{AD}: [0,\frac{g'(a;a)}{4}] \to [0,a]
\end{equation}
that satisfies the following properties.
\begin{itemize}
\item[(i)]{
  We have $u_{AD}(0) = 0$
  and $u_{AD}(\frac{g'(a;a)}{4}) = a$.
}
\item[(ii)]{
  The identity $v_-\big(u_{AD}(d) \big) = v_d\big( u_{AD}(d) \big) $
  holds for any $0 < d \le \frac{g'(a;a)}{4}$.
}
\item[(iii)]{
  Suppose that $v_-(u) = v_d(u) $
   for some $0 < d \le \frac{g'(a;a)}{4}$
  and $0 \le u \le a$.
  Then in fact $u \in \{ u_{AD}(d), a \}$.
}
\item[(iv)]{
  Consider any $0 < d < \frac{g'(a;a)}{4}$
  for which $(a, d) \neq (\frac{1}{2} , \frac{1}{24})$.
  Then we have the inequality
  \begin{equation}
    \label{eq:lem:bif:str:roots:a:transverse}
    v_-'\big(u_{AD}(d) \big) < v_d'\big( u_{AD}(d) \big).
  \end{equation}
}
\item[(v)]{
  For any $d > \frac{g'(a;a)}{4}$
  and $0 \le u < a$ we have
  $v_-(u) > v_d(u)$.
}
\end{itemize}
\end{lem}
\begin{proof}
For convenience, we introduce the function
$h_d(u) = v_-(u) - v_d(u) $ and set out to count the zeroes
of $h_d$ on the interval $[0, a]$.
We first note that $h_d(a) = 0$ for all $d > 0$.
When $d = \frac{g'(a;a)}{4}$ this is in fact the only zero,
which can be seen by using \sref{eq:cor:critical:ineq:for:v:minus:star:i}
and explicitly verifying
the inequality
\begin{equation}
v_-\big(u_{\mathrm{min}}(a) \big)
 > v_d\big( u_{\mathrm{min}}(a) \big)
\end{equation}
for $(a,d) = (\frac{1}{2}, \frac{g'(\frac{1}{2}; \frac{1}{2})}{4} )$.

On the other hand,
\sref{eq:cor:vminus:subcrit:ineq:v:star:minus}
implies that $h_d$ has at least two zeroes for $0 < d < \frac{g'(a;a)}{4}$.
Furthermore, we claim that
$h_d$ has precisely two zeroes for $0 < d \le d_1$
upon choosing $d_1 > 0$ to be sufficiently small.
Indeed, for any $u \in (0, a)$ we have $v_d''(u) < 0$
and we can enforce
\begin{equation}
\abs{v_d'(u)} + v_d(u) \ge 1 + \max_{0 \le \tilde{u} \le a }\{ \abs{v_-'(\tilde{u})} \}
\end{equation}
by restricting the size of $d > 0$.

We now pick $\epsilon > 0$
in such a way that
the function $h_d$
is strictly decreasing on $[0, u_{\min}(a)+ \epsilon]$
for all $d_1 \le d \le \frac{g'(a;a)}{4}$.
This is possible because Corollary \ref{cor:bif:trv:inc:bef:u:min}
allows us to enforce $h_d' < 0$ on this interval,
with the exception of the single point  $u = u_{\mathrm{min}}(a)$ when $a = \frac{1}{2}$.

Let us now define the critical value
\begin{equation}
d_c = \mathrm{sup} \{ d_1 \le d \le \frac{g'(a;a)}{4}:
  h_d =0 \hbox{ has two distinct solutions on } [0, a] \}
\end{equation}
and assume for the moment that $d_c < g'(a;a)/4$. The preparations above
show that there exists $ u_{\mathrm{min}}(a) < u < a$
with $h_{d_c}(u) = 0$ and $h_{d_c}'(u) = 0$.
Corollary \ref{cor:bif:donward:tangency} now implies
$h''_{d_c}(u) > 0$.
As a consequence of the monotonicity $\partial_d v_d(u) < 0$,
this means that for all sufficiently small $\delta > 0$,
the function $h_d$ with $d = d_c - \delta$
must have at least three zeroes.
This yields a contradiction, which implies that $d_c = g'(a;a)/4$.

We may hence define $u_{AD}(d) \in [0, a]$ to be the left-most root of
$h_d(u) = 0$ for $0 < d \le d_c$. The statements (i)-(v) follow readily
from the observations above together with the monotonicity
$\partial_d v_d(u) < 0$ that holds for $u \in (0, a)$.
\end{proof}

\begin{lem}
\label{lem:bif:str:roots:BC}
Fix $0 < a \le \frac{1}{2}$.
Then there exists a constant $0 < d_- < \frac{g'(a;a)}{4}$
together with two continuous functions
\begin{equation}
(u_B, u_C): [0,d_-] \to [0,a] \times [0,a]
\end{equation}
that satisfy the following properties.
\begin{itemize}
\item[(i)]{
  We have $(u_B, u_C)(0) = (0, a)$
  and $u_B(d_-) = u_C(d_-)$.
}
\item[(ii)]{
  The function $u_B$ is strictly increasing,
  while the function $u_C$ is strictly decreasing.
}
\item[(iii)]{
  For any $0 < d \le d_-$ the identity
  $v_+(u_{\#}) = v_d(u_{\#} ) $
  holds for $\# \in \{B, C\}$.
}
\item[(iv)]{
  If $v_+(u) = v_d(u)$ for some $0 < d \le d_-$
  and $0 \le u \le a$
  then in fact $u \in \{ u_B(d), u_C(d) \}$.
}
\item[(v)]{
  For any $0 < d < d_-$ we have
  \begin{equation}
    \label{eq:bif:str:ineqs:star:vs:plus}
    v_+'\big(u_B(d)\big) < v_d'\big(u_B(d)\big) ,
    \qquad
    v_+'\big(u_C(d)\big) > v_d'\big(u_C(d)\big).
  \end{equation}
  If $a \neq \frac{1}{2}$, then we also have
  \begin{equation}
    \label{eq:lem:bif:str:roots:bc:tangency}
     v_+'\big(u_C(d_-)\big) = v_+'\big(u_B(d_-)\big)
    = v_d'\big(u_B(d_-) \big) = v_d'\big(u_C(d_-)\big)
     .
  \end{equation}
}
\item[(vi)]{
  For any $d > d_-$ the inequality
  $v_+(u) > v_d(u)  $ holds for all $0 \le u \le a$.
}
\item[(vii)]{
  If $a = \frac{1}{2}$ then we have
  $d_-(\frac{1}{2}) = \frac{1}{24}$
  together with $u_B(\frac{1}{24}) = u_C(\frac{1}{24}) = u_{\mathrm{min}}(a)$.
}
\end{itemize}
\end{lem}
\begin{proof}
Writing $h_d(u) = v_+(u) - v_d(u)$,
we observe first that $h_d$ is strictly decreasing
on $[0, u_{\mathrm{min}}(a)]$ because $v_d' > 0$ and $v_+' < 0$
on the interior of this interval.

In addition, in the special case $a =\frac{1}{2}$
we can use Corollary \ref{cor:bif:trv:inc:aft:u:min}
to conclude that $h_d$ is strictly increasing
on $[u_{\mathrm{min}}(a), a]$.
Since $h_d\big(u_{\mathrm{min}}(a)\big) = 0$ occurs precisely when
$d = \frac{1}{24}$, all the desired statements
can be easily verified.

Throughout the remainder of this proof we therefore
assume that $0 < a < \frac{1}{2}$.
Arguing as in the proof of Lemma
\ref{lem:bif:str:roots:A},
we may pick $d_1 > 0$ in such a way that
$h_d$ has precisely two zeroes on $[0, a]$
for every $0 < d \le d_1$.
This allows us to define the critical value
\begin{equation}
d_- = \mathrm{sup} \{ d_1 \le d :
  h_d = 0 \hbox{ has two distinct solutions on } [0, a] \}.
\end{equation}
Since $v_- < v_+$ it is clear that $d_- < \frac{g'(a;a)}{4}$.

Let us assume for the moment that $h_{d_-}$ has two or more
zeroes on $[0, a]$. This implies that there exists
at least one $u \in (0, a)$ for which $h_{d_-}(u) = 0$
and $h_{d_-}'(u) = 0$. Using Corollary \ref{cor:bif:donward:tangency}
it follows
that $h''_{d_-}(u) > 0$ must hold  for all such zeroes.
As a consequence of the monotonicity $\partial_d v_d(u) < 0$,
this means that for all sufficiently small $\delta > 0$,
the function $h_d$ with $d = d_- - \delta$
must have at least three zeroes.
This yields a contradiction, which by continuity shows
that $h_{d_-}(u) = 0$ has precisely one root on $[0, a]$.
Upon defining $u_B(d)$ and $u_C(d)$ to be the left-most
respectively right-most root of $h_d(u) = 0$,
the desired properties (i) - (vii) can be easily verified.
\end{proof}

Recalling the function $G$ introduced
in \sref{eq:mr:def:G:G12},
we see that
\begin{equation}
D_{1,2}G(u, v; a, d) = \left(
  \begin{array}{cc}
    g'(u) - 2d & 2d
     \\[0.2cm]
     2d & g'(v) - 2d
     \\[0.2cm]
  \end{array}
\right) .
\end{equation}
In order to study the stability
and parameter-dependence of the roots
constructed in Lemmata
\ref{lem:bif:str:roots:A}-\ref{lem:bif:str:roots:BC},
it is crucial
to understand when the determinant of $D_{1,2}G$ vanishes.
The result below states that this happens
at tangential intersections of $v_\pm$ and $v_d$.

\begin{lem}
\label{lem:bif:str:roots:determinant}
Fix $0 < a < \frac{1}{2}$
together with $d > 0$
and suppose that
\begin{equation}
v_{\#}(u) = v_d(u)
\end{equation}
for some $\# \in \{-, +\}$ and $u \in [0, a]$.
Then we have $\mathrm{det} \, D_{1,2}G(u, v_d(u) ; a, d ) = 0$
if and only if $v_{\#}'(u) = v_d'(u)$.
\end{lem}
\begin{proof}
We first consider the case $v_{\#}'(u) = 0$,
which occurs when $g'(u ;a ) = 0$
and hence $u = u_{\mathrm{min}}(a)$.
Using
$0 < a < \frac{1}{2}$
we see that
\begin{equation}
v_{-}(u) < u_{\mathrm{max}}(a) < v_+(u).
\end{equation}
Writing $v = v_d(u) = v_{\pm}(u)$,
we hence obtain
$g'(v ; a) \neq 0$
and hence
\begin{equation}
  \mathrm{det} \, D_{1,2}G(u, v ;a , d)
     = -2d g'(v; a) \neq 0.
\end{equation}
Since $v_d'(u) = 1$, the desired equivalence indeed holds
for this case.

Assuming now that $v_{\#}'(u) \neq 0$
and hence $g'(u ; a) \neq 0$,
we again write $v = v_d(u) = v_{\#}(u)$
and use \sref{eq:bif:prlm:deriv:vpm}
to compute
\begin{equation}
g'(v ;a) = - [v'_{\#}(u) ]^{-1} g'(u; a) .
\end{equation}
In particular, we find
\begin{equation}
\begin{array}{lcl}
\mathrm{det} \, D_{1,2}G(u, v ; a, d)  & = &
(g'(u;a) - 2d) \big(-  [v'_{\#}(u) ]^{-1} g'(u;a)   - 2d\big) - 4 d^2
\\[0.2cm]
& = &
- [v'_{\#}(u)]^{-1} g'(u;a)^2
- 2d g'(u;a) + 2d [v'_{\#}(u)]^{-1} g'(u;a)
\\[0.2cm]
& = &
2d g'(u;a) [v'_{\#}(u)]^{-1}
\Big[
  1 - \frac{g'(u;a)}{2d} -  v'_{\#}(u)
\Big]
\\[0.2cm]
& = &
2d g'(u;a) [v'_{\#}(u)]^{-1}
\Big[
  v_d'(u) - v'_{\#}(u)
\Big],
\end{array}
\end{equation}
from which the statement follows.
\end{proof}

In order to characterize the dependence
of $d_-(a)$ on $a$,
we introduce the function
\begin{equation}
G_{\mathrm{sn}}(u, v, d ; a)
 = \big(G_1(u,v;a,d) , G_2(u,v;a,d), \det D_{1,2}G(u, v;a,d) \big)^T
\end{equation}
and symbolically write
\begin{equation}
[D_{1,2,3}G_{\mathrm{sn}}(u,v,d;a)]^{-1}
= \big[ \det D_{1,2,3} G_{\mathrm{sn}}(u, v, d;a)\big]^{-1}
\left(
  \begin{array}{ccc}
     * & * & *
     \\[0.2cm]
     * &* & *
     \\[0.2cm]
     \gamma_1(u,v,d;a)
     & \gamma_2(u,v,d;a)
     & \gamma_3(u,v,d;a)
  \end{array}
  \right) .
\end{equation}
For any $0 < a \le \frac{1}{2}$,
we use the functions defined in Lemma
\ref{lem:bif:str:roots:BC} to
introduce the notation
\begin{equation}
\omega(a) =
\Big( u_{B}\big(d_-(a)\big), v_+\big(u_{B}(d_-(a))\big), d_-(a) \Big)
=
\Big( u_{C}\big(d_-(a)\big), v_+\big(u_{C}(d_-(a))\big), d_-(a) \Big),
\end{equation}
which corresponds to the critical point
where the branches $u_B$ and $u_C$ collide.

\begin{cor}
\label{cor:bif:char:roots:ovl:G}
Upon fixing $0 < a < \frac{1}{2}$,
the following two statements are equivalent.
\begin{itemize}
\item[(a)]{
  The identity $G_{\mathrm{sn}}(u, v, d; a) = 0$
  holds for some $d > 0$ and some pair $(u,v) \in [0,1]^2$
  that has $v \ge u$.
}
\item[(b)]{
  We have $(u, v, d) = \omega(a)$
  or $(u, v, d) =
  (a , a, \frac{g'(a;a)}{4} )$ .
}
\end{itemize}
\end{cor}
\begin{proof}
As a preparation, we note that $\det D_{1,2}G(a,a;a,d) =0$
if and only if $d = \frac{g'(a;a)}{4}$.
In addition, it is easy to check that
$\det D_{1,2}G(0, 0;a,d) > 0$
and $\det D_{1,2}G(1, 1;a,d) > 0$
for all $d \ge 0$.

The implication $(b)\rightarrow(a)$
can now be verified directly
using Lemma \ref{lem:bif:str:roots:determinant}
and \sref{eq:lem:bif:str:roots:bc:tangency}.
In addition, we only need to establish
the reverse implication under the additional assumption
that $v > u$.

Let us therefore assume that $(a)$ holds
with $g(u) = -g(v) \neq 0$, which allows us to write
\begin{equation}
0 < u < a < v < 1.
\end{equation}
In particular, Lemma \ref{lem:bif:setup:bs:props:vpm}
implies that $v = v_-(u)$
or $v = v_+(u)$.
Lemma \ref{lem:bif:str:roots:determinant}
together with \sref{eq:bif:str:ineqs:star:vs:plus}
and \sref{eq:lem:bif:str:roots:a:transverse}
now imply that (b) must hold.
\end{proof}

\begin{lem}
We have the identities
\begin{equation}
\label{eq:bif:ids:for:gamma:1:2}
\begin{array}{lcl}
\gamma_1( 0, 1 , 0 ; 0) & = & -2,
\\[0.2cm]
\gamma_2( 0, 1 , 0 ; 0) & = & 0.
\end{array}
\end{equation}
In addition, the identity
\begin{equation}
\begin{array}{lcl}
\gamma_3(
 \omega(a) ; a) & = & 0
\end{array}
\end{equation}
holds for any $0 < a < \frac{1}{2}$,
together with the inequalities
\begin{equation}
\label{eq:bif:ineq:for:det:d:ovl:G}
\begin{array}{lcl}
\det D_{1,2,3} G_{\mathrm{sn}}
( \omega(a) ; a)
 & < & 0 ,
\\[0.2cm]
\gamma_1
( \omega(a) ; a)
- \gamma_2(\omega(a);a)
 & < &
 0 ,
\\[0.2cm]
\gamma_1
(\omega(a);a)
\gamma_2
( \omega(a);a)
& \ge & 0 .
\end{array}
\end{equation}
\end{lem}
\begin{proof}
Let us assume for the moment that
$G_{\mathrm{sn}}(u, v, d; a) = 0$,
which directly implies
\begin{equation}
\label{eq:bif:d:123:g3:is:zero:consq}
(g'(u;a) - 2d ) g'(v;a)  = 2d g'(u;a).
\end{equation}
In view of the identity
\begin{equation}
\label{eq:bif:d:ovl:g}
D_{1,2,3}G_{\mathrm{sn}}(u,v,d;a)
= \left(
  \begin{array}{ccc}
     g'(u;a) - 2d & 2d & 2 (v - u)
     \\[0.2cm]
     2d & g'(v;a) - 2d & 2 (u - v)
     \\[0.2cm]
     (g'(v;a) - 2d) g''(u;a)
     & (g'(u;a) - 2d) g''(v;a)
     & -2 (g'(u;a) + g'(v;a) )
  \end{array}
  \right),
\end{equation}
we can use \sref{eq:bif:d:123:g3:is:zero:consq}
to compute
\begin{equation}
\label{eq:bif:d:123:ovl:g:dir:expr}
\begin{array}{lcl}
\mathrm{det} \, D_{1,2,3} G_{\mathrm{sn}}(u,v,d;a)
& = & 2(v-u)
\Big[
 g'(u;a) (g'(u;a) - 2d) g''(v;a)
 - g'(v;a) (g'(v;a) - 2d) g''(u;a)
\Big]
\\[0.2cm]
& =&
  2 (v - u)
  \Big[
  2d g'(u;a)^2 \frac{1}{g'(v;a)} g''(v;a)
 - g'(v;a) (g'(v;a) - 2d) g''(u;a)
\Big].
\\[0.2cm]
\end{array}
\end{equation}
Expanding the subdeterminants
of \sref{eq:bif:d:ovl:g}
and reusing \sref{eq:bif:d:123:g3:is:zero:consq},
we also find
\begin{equation}
\begin{array}{lcl}
\gamma_1(u,v,d ;a) & = & 2d (g'(u;a) - 2d) g''(v;a)
 - (g'(v;a) - 2d)^2 g''(u;a) ,
\\[0.2cm]
\gamma_2(u,v,d;a) & = &
2d (g'(v;a) -2 d) g''(u;a)
  - (g'(u;a) - 2d)^2 g''(v;a) ,
\\[0.2cm]
\gamma_3(u,v,d;a) & = & 0,
\end{array}
\end{equation}
which allows us to explicitly
verify \sref{eq:bif:ids:for:gamma:1:2}.

Let us now fix $0 < a< \frac{1}{2}$.
For any $(u, v) \in [0,1]^2$ and $d > 0$ we introduce the function
\begin{equation}
\begin{array}{lcl}
h\big(u,v , d \big)
 & = &
 \frac{1}{2d} g''(u; a)
- \frac{g''(u ; a)}{g'(v; a)}
- \frac{g'(u ; a)^2}{g'(v ; a)^3} g''(v ; a)
\\[0.2cm]
& = &\frac{1}{2d g'(v ; a)^2}
\Big[
  g'(v ; a)(g'(v ; a) - 2d) g''(u ; a)
  - 2d \frac{g'(u ; a)^2}{g'(v; a)} g''(v ; a)
\Big] ,
\end{array}
\end{equation}
which allows us to rewrite
\sref{eq:bif:d:123:ovl:g:dir:expr}
in the form
\begin{equation}
\begin{array}{lcl}
\mathrm{det} \, D_{1,2,3} G_{\mathrm{sn}}(u,v,d;a)
& = &
-4d (v -u) g'(v;a)^2 h(u, v, d).
\end{array}
\end{equation}
Using \sref{eq:preps:sec:deriv:v:pm} and \sref{eq:preps:sc:deriv:v:star}
we observe that
\begin{equation}
v_+''(u) - v_d''(u)
= h\big(u, v_+(u) ,d \big).
\end{equation}
In particular, Corollary \ref{cor:bif:trnsv:v:plus:tangency}
and \sref{eq:lem:bif:str:roots:bc:tangency}
imply that
\begin{equation}
h\big( \omega(a)  \big) > 0,
\end{equation}
which shows that
\begin{equation}
\mathrm{det} \, D_{1,2,3} G_{\mathrm{sn}}(\omega(a);a)
< 0,
\end{equation}
as desired.

Let us now write $(u, v, d) = \omega(a)$
and compute
\begin{equation}
\begin{array}{lcl}
\gamma_1\big( \omega(a) ; a \big)
 - \gamma_2\big(\omega(a)  ; a\big)
& = & g'(u) \big(g'(u) - 2d) g''(v)
  - g'(v) \big(g'(v) - 2d) g''(u)
\\[0.2cm]
& = &
 \frac{1}{2} (v - u)^{-1}
   \mathrm{det} \, D G_{\mathrm{sn}}(u, v, d;a)
\\[0.2cm]
& \le & 0.
\end{array}
\end{equation}
In addition, since
$u > u_{\mathrm{min}}(a)$
and hence $v_d'(u)
 = v_+'(u) > 0$,
we have $g'(u) < 2d$.
This allows us to define
\begin{equation}
\begin{array}{lcl}
\alpha & = &
  \sqrt{2d}  \big(2d - g'(u;a)\big)^{3/2} g''(v;a) ,
\\[0.2cm]
\beta & = &
   \sqrt{2d}\big(2d - g'(v;a ) \big)^{3/2} g''(u ; a)
\end{array}
\end{equation}
and compute
\begin{equation}
\begin{array}{lcl}
\gamma_1(\omega(a);a) \gamma_2(\omega(a);a)
& = &
  (4d^2) (4 d^2) g''(u;a) g''(v;a)
  + (4d^2)^2 g''(u;a) g''(v;a)
\\[0.2cm]
& & \qquad
  - 2d (g'(u;a) - 2d)^3 g''(v;a)^2
  - 2d (g'(v;a) - 2d)^3 g''(u;a)^2
\\[0.2cm]
& = &
  2 \alpha \beta  + \alpha^2 + \beta^2
\\[0.2cm]
& \ge & 0,
\end{array}
\end{equation}
as desired.
\end{proof}

\begin{cor}
\label{cor:bif:d:min:str:increasing}
The map $a \mapsto \omega(a)$ is $C^\infty$-smooth
on $(0, \frac{1}{2})$ and we have
$ d_-'(a) > 0$.
\end{cor}
\begin{proof}
Since $d_- < \frac{g'(a;a)}{4}$,
we may use
Corollary \ref{cor:bif:char:roots:ovl:G},
together with the first inequality in
\sref{eq:bif:ineq:for:det:d:ovl:G}
to apply the implicit function
and establish the $C^\infty$-smoothness of $\omega$.

Since $g'(a;a)/4 \downarrow 0$ as $a \downarrow 0$,
a squeezing argument shows that
$\lim_{a \downarrow 0} \omega(a) = (0 , 1 , 0)$.
Exploiting \sref{eq:bif:ids:for:gamma:1:2},
and \sref{eq:bif:ineq:for:det:d:ovl:G},
the continuity of  $\gamma_{1}$
and $\gamma_2$
implies that
\begin{equation}
\gamma_1(\omega(a);a) < \gamma_2(\omega(a);a) \le 0
\end{equation}
for all $0 < a < \frac{1}{2}$.
In particlar,
writing $(u,v,d) = \omega(a)$,
we may use the inequalities
\begin{equation}
D_2 g (u ;a ) < 0,
\qquad
D_2  g(v ;a ) < 0
\end{equation}
to compute
\begin{equation}
d_-'(a) = -
 [\mathrm{det} \, D_{1,2,3} G_{\mathrm{sn}}(\omega(a);a)]^{-1}
\big[\gamma_1(\omega(a);a) D_2 g(u ; a)
   + \gamma_2(\omega(a);a) D_2 g( v; a) \big]
 > 0 .
\end{equation}
\end{proof}

\begin{proof}[Proof of Proposition \ref{prp:st:dpm:ex}]
Items (i), (iii), (iv)
follow from Lemmata
\ref{lem:bif:str:roots:A}-\ref{lem:bif:str:roots:BC} and
the observations
in the proof of
Corollary \ref{cor:bif:char:roots:ovl:G}.
Item (ii) follows from
Corollary \ref{cor:bif:d:min:str:increasing},
while the expansions in (v) follow from Propositions \ref{prp:bif:exp:a:zero}
and \ref{prp:bif:cusp:exp}.
\end{proof}

\begin{proof}[Proof of Proposition \ref{prp:st:uv:stb:ex}]
The identity $g(1-u, 1-a) = - g(u, a)$
implies that
\begin{equation}
G(1- u, 1 - v; 1-a , d)
= - G( u, v; a, d).
\end{equation}
This allows us to extend the solutions
constructed in Lemmata
\ref{lem:bif:str:roots:A}-\ref{lem:bif:str:roots:BC}
to the entire interval $0 \le a \le 1$
in the fashion outlined in Corollary \ref{cor:bif:1:min:a},
which yields (ii) and (iv).

To establish (i), we note that
for each $\# \in \{A , B, C \}$ the sign of
$\det D_{1,2} G(\overline{u}_{\#} , \overline{v}_{\#} ; a, d)$
is constant on $\Omega_-$
on account of Corollary \ref{cor:bif:char:roots:ovl:G}.
Using (ii) the eigenvalues of these matrices
can be explicitly computed at $d = 0$, which yields the desired
stability properties.

To obtain the strict ordering in (iii),
we fix $0 < a \le \frac{1}{2}$
and $0 < d < d_-(a)$.
We observe  that $v_+(u) \ge v_-(u)$
for $u \in [0,a]$, with
equality only at $ u = u_{\mathrm{min}}(a)$
for $a = \frac{1}{2}$. Since $v_d(0) = 0$
this implies that
$\overline{u}_A(a,d) < \overline{u}_B(a,d)$
and $\overline{v}_A(a,d) < \overline{v}_B(a,d)$.
By construction, we also have
$\overline{u}_B(a,d) < \overline{u}_C(a,d)$.
In addition, it cannot be the case
that $\overline{v}_B(a,\tilde{d}) = \overline{v}_{C}(a,\tilde{d})$
for any $0 < \tilde{d} < d_-(a)$
since then $G_2 = 0$ implies that also
$\overline{u}_B(a,\tilde{d}) = \overline{u}_C(a,\tilde{d})$.
Using the general identity
\begin{equation}
D_{a,d} (\overline{u}_{\#} , \overline{v}_{\#} )
 = -[D_{1,2}G(\overline{u}_{\#} , \overline{v}_{\#} ; a, d)]^{-1}
  D_{3,4} G(\overline{u}_{\#} , \overline{v}_{\#} ; a, d)
\end{equation}
we see that
\begin{equation}
\partial_d \overline{v}_{\#} (a, 0)
 =  - \frac{1}{g'(\overline{v}_{\#}(a, 0) ;a) }
\big(\overline{u}_{\#}(a, 0) - \overline{v}_{\#}(a, 0) \big),
\end{equation}
which yields
\begin{equation}
\partial_d \overline{v}_{B} (a, 0)
 =  \frac{1}{g'(1 ; a)}
<
   \frac{1 - a}{g'(1;a)}
 = \partial_d \overline{v}_{C} (a, 0).
\end{equation}
These observations allow us to conclude
$\overline{v}_B(a,d) < \overline{v}_{C}(a,d)$,
as desired.
\end{proof}

\begin{proof}[Proof of Proposition \ref{prp:st:uv:ustb:ex}]
Arguing in a similar fashion as the proof of
Proposition \ref{prp:st:uv:stb:ex},
the statements follow from
Lemma \ref{lem:bif:str:roots:A}
and Corollary \ref{cor:bif:char:roots:ovl:G}.
\end{proof}



\section{Travelling waves}
\label{sec:twv}

Our goal here is to establish the existence
of bichromatic wave solutions
to the Nagumo LDE \sref{eq:mr:lde:main}
and to obtain detailed results concerning their speed.
In particular, we establish
Theorems \ref{thm:mr:twv:ex} and \ref{thm:mr:twv:sets:t},
which are the main results of this paper.

Upon introducing
the standard discrete Laplacian
\begin{equation}
[\Delta^+ u](\xi) = u(\xi + 1) + u(\xi - 1) - 2 u(\xi)
\end{equation}
together with the
off-diagonal matrix
\begin{equation}
\mathcal{J} =
\left( \begin{array}{cc} 0 & 1 \\ 1 & 0 \end{array} \right),
\end{equation}
the travelling wave system
\sref{eq:mr:wave:mfde}
can be rewritten as
\begin{equation}
\label{eq:twv:wave:mfde:comp}
-c \Phi' = d \mathcal{J} \Delta^+ \Phi + G(\Phi ; a, d).
\end{equation}
For any $(a, d) \in \Omega_-$, we set out to
seek solutions to \sref{eq:twv:wave:mfde:comp}
that satisfy the boundary conditions
\begin{equation}
\label{eq:twv:wave:bnd:conds}
\lim_{\xi \to -\infty}
  \Phi(\xi) = (0, 0),
\qquad
\lim_{\xi \to + \infty}
  \Phi(\xi) = \big(\overline{u}_B(a,d) , \overline{v}_B(a,d) \big) .
\end{equation}
The existence of such solutions
will be established in \S\ref{sec:twv:ex},
where we also show $c \ge 0$ and establish Theorem \ref{thm:mr:twv:ex}.
In \S\ref{sec:twv:char} we subsequently set out
to derive criteria that distinguish between the two
cases $c = 0$ and $c > 0$. We verify these
criteria in \S\ref{sec:twv:expl} for parameters $(a,d) \in \Omega_-$
that are close to the cusp $(\frac{1}{2}, \frac{1}{24})$
and the corner $(1, 0)$. This allows us to establish
our main results contained in Theorem \ref{thm:mr:twv:sets:t}.

\subsection{Existence of waves}
\label{sec:twv:ex}

Our preparatory work in \S\ref{sec:eqlb} allows us to invoke
the theory developed in \cite{CHENGUOWU2008} to establish a general existence
result for bichromatic waves. Indeed, the equilibria $(0, 0)$ and $(\overline{u}_B, \overline{v}_B)$
are both stable under the dynamics of $(\dot{u}, \dot{v}) = G(u, v; a, d)$
and all intermediate equilibria are unstable.
Using a straightforward estimate
based on the ordering \sref{eq:bif:abc:ordering},
the wavespeeds can be shown to be non-negative.

\begin{lem}
\label{lem:tvw:main:ex}
Pick $(a,d) \in \Omega_-$.
Then there exists a constant $c \in \Real$
and a non-decreasing function $\Phi: \Real \to \Real^2$
that satisfy \sref{eq:twv:wave:mfde:comp}-\sref{eq:twv:wave:bnd:conds}.
This constant $c$ is unique,
while  the function $\Phi$ is unique up to translation if $c \neq 0$.
In the latter case we also have $\Phi'(\xi) > (0,0)$ for all $\xi \in \Real$.
\end{lem}
\begin{proof}
These statements follow directly from the main results
in \cite{CHENGUOWU2008}.
\end{proof}

\begin{lem}
\label{lem:twv:speed:non:neg}
Pick $(a,d) \in \Omega_-$.
Then the constant $c$ defined in Lemma
\ref{lem:tvw:main:ex}
satisfies $c \ge 0$.
\end{lem}
\begin{proof}
Suppose that $c \neq 0$. We estimate
\begin{equation}
\begin{array}{lcl}
-c \Phi'_v(\xi)
& = &
d [ \Phi_u(\xi + 1) + \Phi_u(\xi - 1) - 2 \Phi_v(\xi) ]
+ g (\Phi_v(\xi) ; a)
\\[0.2cm]
& \le &
  2d [\overline{u} - v(\xi)]
   + g(\Phi_v(\xi); a) .
\\[0.2cm]
\end{array}
\end{equation}
Since $0 < \overline{u}_B(a,d) < a < \overline{v}_B(a,d)$, there exists
$\xi_* \in \Real$ for which
$\Phi_v(\xi_*) = \overline{u}_B(a,d)$.
This yields
\begin{equation}
-c \Phi'_v(\xi_*)
 \le g( \overline{u} ; a ) < 0,
\end{equation}
from which we conclude $c > 0$.
\end{proof}

Writing $c(a,d)$ and $\Phi(a,d)$ for the wavespeed and waveprofile
defined in
Lemma \ref{lem:tvw:main:ex},
we introduce the set
\begin{equation}
\mathcal{T} = \{(a , d) \in \Omega_-: c(a,d) > 0 \},
\end{equation}
which corresponds to the set $\mathcal{T}_{\mathrm{low}}$ used in
\S\ref{sec:mr}.
%
In addition, for any $(a,d) \in \mathcal{T}$
we introduce the linear operators
\begin{equation}
\mathcal{L}_{a,d} : W^{1,\infty}(\Real; \Real^2) \to L^\infty(\Real ; \Real^2),
\qquad
\mathcal{L}^{\mathrm{adj}}_{a,d} : W^{1,\infty}(\Real; \Real^2) \to L^\infty(\Real ; \Real^2),
\end{equation}
that act as
\begin{equation}
\begin{array}{lcl}
\mathcal{L}_{a,d} \phi
& = & - c(a,d) \phi' - d \mathcal{J} \Delta^+ \phi - DG(\Phi(a,d) ; a, d) \phi,
\\[0.2cm]
\mathcal{L}^{\mathrm{adj}}_{a,d} \psi
& = &  c(a,d) \psi' - d \mathcal{J} \Delta^+ \psi - DG(\Phi(a,d) ; a, d) \psi.
\end{array}
\end{equation}
The results in \cite[{\S}8]{HJHNEGDIF} imply that
there exists $\Psi = \Psi(a,d) \in W^{1,\infty}(\Real; \Real^2)$ with $\Psi > (0, 0)$
for which we have the identities
\begin{equation}
\mathrm{Ker} \, \mathcal{L}_{a,d} = \mathrm{span} \{ \Phi'(a,d) \},
\qquad
\mathrm{Ker} \, \mathcal{L}^{\mathrm{adj}}_{a,d} = \mathrm{span} \{ \Psi(a,d) \}
\end{equation}
together with the normalization
\begin{equation}
\int_{-\infty}^\infty \langle \Psi(\xi), \Phi'(\xi) \rangle \,d  \xi = 1.
\end{equation}
In particular, \cite[Thm. A]{MPA} implies that
\begin{equation}
\label{eq:twv:char:range}
\mathrm{Range} \, \mathcal{L}_{a,d} = \{ f \in L^\infty(\Real;\Real^2):
  \int_{-\infty}^\infty \langle \Psi(\xi) , f(\xi) \rangle \, d \xi = 0 \}.
\end{equation}
These ingredients allow us to use the implicit function
theorem to show that the pair $(c,\Phi)$ depends smoothly
on the parameters $(a,d) \in \mathcal{T}$. In addition,
we obtain a sign on $\partial_a c$.

\begin{lem}
\label{lem:twv:smoothness}
The maps
\begin{equation}
\begin{array}{lclcl}
\mathcal{T} \ni (a, d) & \mapsto & c(a, d) & \in & (0, \infty) ,
\\[0.2cm]
\mathcal{T} \ni (a, d) & \mapsto & \Phi(a, d) & \in & BC^1(\Real ;\Real^2)
\end{array}
\end{equation}
are $C^\infty$-smooth. In addition, we have
\begin{equation}
\partial_a c(a,d) > 0
\end{equation}
for all $(a,d) \in \mathcal{T}$.
\end{lem}
\begin{proof}
The $C^1$-smoothness of the maps $(c, \Phi)$ follows from
\cite[Thm. 2.3]{HJHNEGDIF}. On account of the smoothness of the
function $g$,
this can readily be extended to the desired $C^\infty$-smoothness
by using the ideas in \cite[{\S}8]{HJHNEGDIF}
to set up an implicit function argument along the lines of
\cite[Prop. 6.5]{MPB}.

Differentiating \sref{eq:twv:wave:mfde:comp}
with respect to $a$, we compute
\begin{equation}
- [\partial_a c ] \Phi' - c [\partial_a \Phi' ]
= d \mathcal{J} \Delta^+ [\partial_a \Phi ] + DG(\Phi; a, d) \partial_a \Phi  +
\partial_a G(\Phi;a,d),
\end{equation}
which gives
\begin{equation}
- [\partial_a c ] \Phi' + \mathcal{L}_{a,d} \partial_a \Phi = \partial_a G(\Phi ;a , d).
\end{equation}
Applying \sref{eq:twv:char:range}
and noting that $\partial_a g(u;a) = - u (1 - u) <0$
for all $u \in (0, 1)$,
we obtain
\begin{equation}
-\partial_a c = \int_{-\infty}^\infty \langle \partial_a G(\Phi(\xi);a,d) , \Psi(\xi) \rangle \, d \xi < 0,
\end{equation}
as desired.
\end{proof}

\begin{cor}
\label{cor:twv:pos:zone}
If $(a, d) \in \mathcal{T}$,
then also
$(a', d) \in \mathcal{T}$
for all $(a', d) \in \Omega_-$
with $a' \ge a$.
\end{cor}

\begin{proof}[Proof of Theorem \ref{thm:mr:twv:ex}]
The statements follow directly from
Lemmata \ref{lem:tvw:main:ex}-\ref{lem:twv:smoothness}.
\end{proof}

\subsection{Characterization of waves}
\label{sec:twv:char}

We now set out to derive conditions that guarantee either
$c(a,d)=0$ or $c(a,d) > 0$. To this end, we introduce the functions
\begin{equation}
u_{a,d}(v) = v - \frac{g(v; a)}{2d},
\qquad
v_{a,d}(u) = u - \frac{g(u;a)}{2d}
\end{equation}
and note that
\begin{equation}
\overline{u}_{\#}(a,d)
 = u_{a,d}\big( \overline{v}_{\#}(a, d) \big),
\qquad
\overline{v}_{\#}(a,d) =
 v_{a,d}\big( \overline{u}_{\#}(a,d) \big)
\end{equation}
for each $\# \in \{A , B , C \}$
and $(a,d) \in \Omega_-$.
The local extrema of these functions
are located at the critical points
\begin{equation}
\begin{array}{lcl}
\gamma_{c;\pm}(a,d)
  & = &
  \frac{1}{3}(a + 1)
\pm \frac{1}{3} \sqrt{a^2 - a + 1 - 6 d}
\\[0.2cm]
& = &
  u_{\mathrm{infl}}(a) \pm \sqrt{\frac{1}{3}(g'(u_{\mathrm{infl}}(a);a) - 2d )} .
\end{array}
\end{equation}
One can verify that the functions
$u_{a,d}$ and $v_{a,d}$ are strictly decreasing
on $[\gamma_{c;-}(a,d), \gamma_{c;+}(a,d)]$
and strictly increasing outside this interval.
The following ordering result exploits this
characterization and will allow us
to establish $c(a,d) = 0$ for a significant portion
of the parameter set $\Omega_-$.

\begin{lem}
\label{lem:twv:vcrit:ordering}
For any $(a,d) \in \overline{\Omega}_-$
we have the ordering
\begin{equation}
\label{eq:twv:ineqs:crit:v:b}
0 \le \gamma_{c;-}(a,d) \le a \le
   \gamma_{c;+}(a,d) \le \overline{v}_B(a,d).
\end{equation}
If $(a,d) \in \Omega_-$ then the inequalities
in \sref{eq:twv:ineqs:crit:v:b} are all strict.
\end{lem}
\begin{proof}
Let us first fix $(a,d) \in \Omega_-$.
We note that
$2d < 2 d_+(a) =  \frac{g'(a;a)}{2} \le \frac{g'(u_{\mathrm{infl}}(a);a)}{2}$,
which implies that $\gamma_{c;\pm}(a,d)$ are well-defined.
In addition, this allows us to compute
\begin{equation}
u_{a,d}'(a) = 1 - \frac{g'(a;a)}{2d} \le -1,
\end{equation}
which means that $a \in \big(\gamma_{c;-}(a,d) , \gamma_{c;+}(a,d) \big)$.

In particular,
the function
$v \mapsto u_{a,d}(v)$ is strictly decreasing
on $[a, \gamma_{c;+}(a,d)]$.
On account of the orderings
\begin{equation}
u_{a,d}\big(\overline{v}_A(a,d) \big) =  \overline{u}_A(a,d)
< \overline{u}_B(a,d) = u_{a,d}\big(\overline{v}_B(a,d) \big) ,
\qquad
\qquad
  a
< \overline{v}_A(a,d) < \overline{v}_B(a,d)
\end{equation}
we hence cannot have
$\overline{v}_B(a,d) \le \gamma_{c;+}(a)$.
The results for the general case
$(a,d) \in \overline{\Omega}_-$ now follow
by continuity.
\end{proof}

\begin{lem}
\label{lem:twv:d:star}
Consider any $(a, d) \in \Omega_-$
for which $d \le \frac{1}{8} (1-a)^2$.
Then we have $c(a,d) = 0$.
\end{lem}
\begin{proof}
Fix any $0 < a < 1$, write
$d_* = \frac{1}{8} (1-a)^2$
and suppose that $(a, d_*) \in \Omega_-$.
On account of Corollary \ref{cor:twv:pos:zone}
it suffices to show that $c(a, d_*) = 0$.
Assuming to the contrary that $c = c(a, d_*) > 0$,
we compute
\begin{equation}
\label{eq:twv:d:spd:zero:curve:ineq:phi:prime:v}
\begin{array}{lcl}
0 & > &
-c \Phi'_v(\xi)
\\[0.2cm]
& = &
d_* [ \Phi_u(\xi + 1) + \Phi_u(\xi - 1) - 2 \Phi_v(\xi) ]
+ g (\Phi_v(\xi) ; a)
\\[0.2cm]
& > &
- 2 d_* \Phi_v(\xi) + g(\Phi_v(\xi) ; a)
\\[0.2cm]
& = &
 - 2 d_* u_{a,d_*}\big( \Phi_v(\xi) \big) .
\end{array}
\end{equation}
Since $0 < \gamma_{c;+}(a , d_*) < \overline{v}_B(a , d_*)$,
there exists $\xi_* \in \Real$
for which $\Phi_v(\xi_*) = \gamma_{c;+}(a, d_*)$.
The key point is that
\begin{equation}
u_{a, d_*}\big( \gamma_{c;+}(a, d_*) \big) = 0,
\end{equation}
which allows us to obtain a contradiction
by picking $\xi = \xi_*$
in \sref{eq:twv:d:spd:zero:curve:ineq:phi:prime:v}.
\end{proof}


We now set out to derive a (non-sharp)
condition that guarantees $c(a,d) > 0$.
The strategy will be to rule out the existence
of standing bichromatic waves. Let us therefore
pick any $(a,d) \in \Omega_- \setminus \mathcal{T}$, which implies $c(a,d) = 0$.
Writing $ (\Phi_u, \Phi_v) = \Phi(a,d)$ for the corresponding profile, we introduce
the sequence
\begin{equation}
\label{eq:twv:def:standing:seq}
(u_i, v_i) = \big( \Phi_u( 2 i ), \Phi_v(2i + 1) \big),
\end{equation}
which satisfies the limits
\begin{equation}
\label{eq:twv:limits:stnd:wave}
\lim_{i \to - \infty} (u_i, v_i)
  = (0 , 0)
\qquad
\lim_{i \to + \infty} (u_i, v_i)
  = (\overline{u}_B(a, d) , \overline{v}_B(a,d) \big),
\end{equation}
together with the difference equation
\begin{equation}
\begin{array}{lcl}
0 & =& d \big[ v_i + v_{i-1} - 2 u_i \big] + g(u_i ; a),
\\[0.2cm]
0 & = & d \big[ u_{i+1} + u_i - 2 v_i \big] +  g(v_i; a).
\end{array}
\end{equation}
Applying a shift to the first equation, we obtain the implicit system
\begin{equation}
\begin{array}{lcl}
v_{i+1} & = &   2 \big[ u_{i+1} - \frac{g(u_{i+1}; a)}{2d} \big] -  v_i ,
\\[0.2cm]
u_{i+1} & = & 2 \big[ v_i - \frac{g(v_i ; a)}{2d} \big] - u_i ,
\\[0.2cm]
\end{array}
\end{equation}
which can be written as
\begin{equation}\label{e:two:reflection:system}
\begin{array}{lcl}
v_{i+1} & = & 2 v_{a,d}( u_{i+1} ) - v_i ,
\\[0.2cm]
u_{i+1} & = & 2 u_{a,d}( v_i ) - u_i .
\\[0.2cm]
\end{array}
\end{equation}

\begin{figure}
\centering
\begin{minipage}{0.6\textwidth}
\centering
\includegraphics[width=\textwidth]{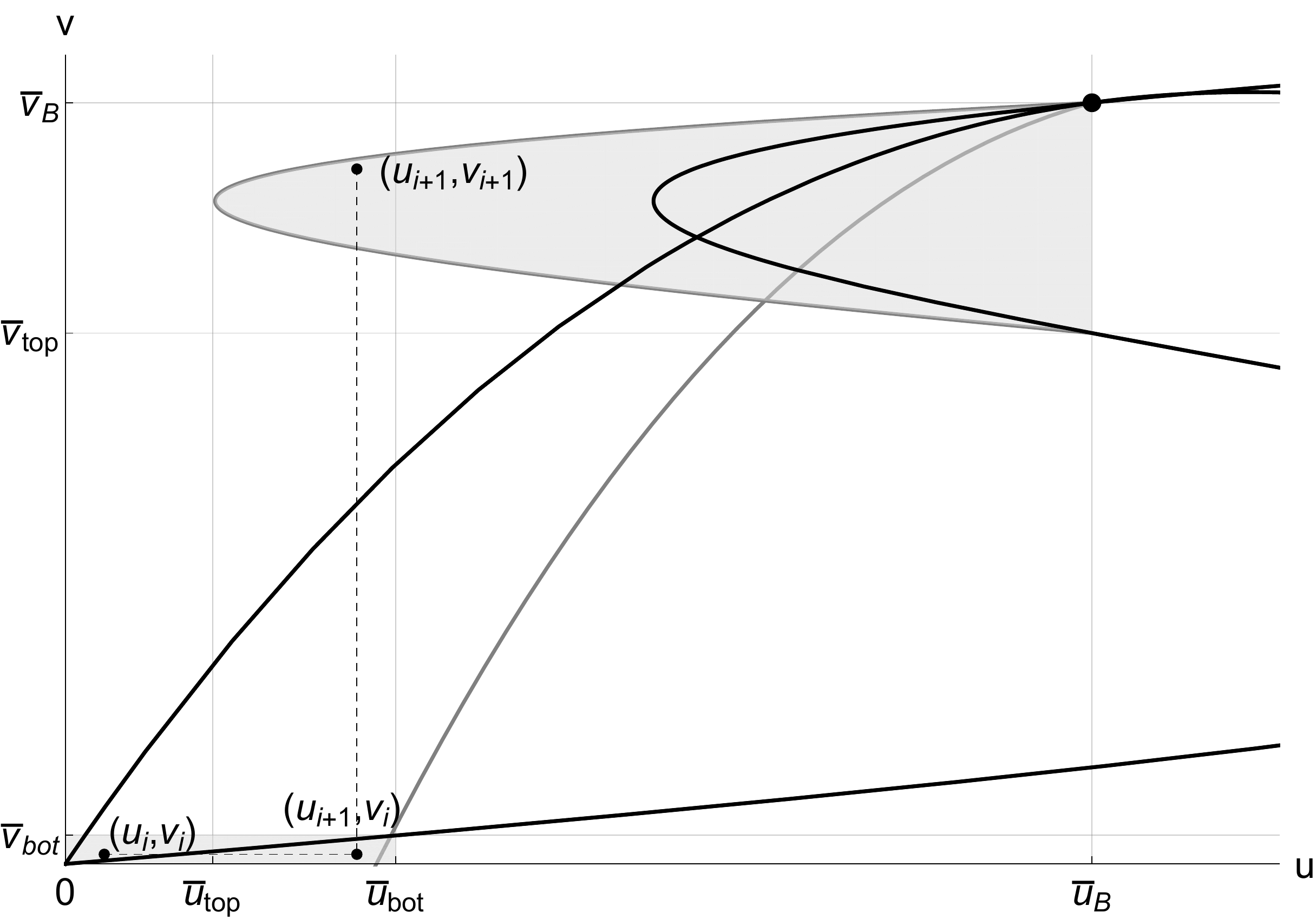}
\end{minipage}
\caption{Illustration of the two reflections described by the system
\eqref{e:two:reflection:system}. The values have been modified for illustrative purposes, since the real regions are minuscule.
}\label{f:two:reflection:illustration}
\end{figure}

In particular,
we can obtain $(u_{i+1} , v_{i+1})$
by first reflecting
$(u_{i}, v_i)$ horizontally
through the curve $u = u_{a,d}(v)$
and then vertically through the curve $v= v_{a,d}(u)$.
Based on this geometric intuition,
we set out to construct a rectangle
\begin{equation}
[\overline{u}_{\mathrm{top}}(a,d) , \overline{u}_B(a,d) ] \times [\overline{v}_{\mathrm{top}}(a,d), \overline{v}_B(a,d)]
\end{equation}
that must contain $(u_{i_0 + 1} , v_{i_0 + 1})$ for some critical $i_0$,
together with a rectangle
\begin{equation}
[0, \overline{u}_{\mathrm{bot}}(a,d) ] \times [0, \overline{v}_{\mathrm{bot}}(a,d)]
\end{equation}
that must contain the intermediate point $(u_{i_0 + 1} , v_{i_0} )$; see Figure \ref{f:two:reflection:illustration}.

\begin{lem}
There exist continuous functions
\begin{equation}
(\overline{v}_{\mathrm{bot}} , \overline{v}_{\mathrm{top}}):
 \Omega_{-} \to [0,1]^2
\end{equation}
that satisfy the inequalities
\begin{equation}
\label{eq:twv:def:ineqs:v:bot:top}
0 < \overline{v}_{\mathrm{bot}}(a,d) <
  \gamma_{c;-}(a,d)
  < \overline{v}_{\mathrm{top}}(a,d)
  <    \gamma_{c;+}(a,d)
  < \overline{v}_B(a,d)
\end{equation}
together with the identities
\begin{equation}
\label{eq:twv:def:v:bot:top:ids:for:ud}
u_{a,d} \big(\overline{v}_{\mathrm{bot}}(a,d) \big)
= u_{a,d}\big( \overline{v}_{\mathrm{top}}(a,d) \big)
  = \overline{u}_B(a,d)
\end{equation}
for each $(a,d) \in \Omega_-$.
Furthermore, these functions can be continuously extended
to $\overline{\Omega}_{-}$ in such a way
that \sref{eq:twv:def:v:bot:top:ids:for:ud} holds
whenever $d > 0$.
\end{lem}
\begin{proof}
Pick any $(a,d) \in \Omega_{-}$.
On account of the identity
$u_{a,d}(0) = 0$ and the inequalities
\begin{equation}
u_{a,d}\big(\gamma_{c;+}(a,d) \big) < u_{a,d}\big( \overline{v}_B(a,d) \big)
=  \overline{u}_{B}(a,d)  < a = u_{a,d}( a )   ,
\end{equation}
the equation $u_{a,d}(v) = \overline{u}_B(a,d)$ has three
distinct solutions
on $(0, \overline{v}_B(a,d)]$.
Using the fact that $\overline{u}_B(a , 0) = 0$,
these solutions can be extended continuously to $d = 0$
by writing $\overline{v}_{\mathrm{bot}}(a, 0) = 0$
and
$\overline{v}_{\mathrm{top}}(a,0) = a$.
The extension to $d = d_-(a) > 0$ can be achieved
by standard continuity arguments.
\end{proof}

For any $(a,d) \in \overline{\Omega}_-$
with $d > 0$
we now define the constant
\begin{equation}
\label{eq:twv:def:u:top}
\overline{u}_{\mathrm{top}}(a, d) = 2 u_{a,d}\big( \gamma_{c;+}( a, d) \big) - \overline{u}_B(a,d) .
\end{equation}
We note that $\big(\overline{u}_{\mathrm{top}}(a,d) , \gamma_{c;+}(a,d) \big)$ can be seen as the
horizontal reflection of $\big( \overline{u}_B( a, d) , \gamma_{c;+}(a,d) \big)$
through the curve $u = u_{a,d}(v)$.

\begin{lem}
\label{lem:twv:stn:wave:crit:point}
Pick any $(a,d) \in \Omega_-$ for which
$c(a,d) = 0$ and consider the sequence $\{ (u_i, v_i) \}$
defined in \sref{eq:twv:def:standing:seq}.
Then there exists $i_0 \in \mathbb{Z}$ for which
\begin{equation}
\label{eq:lem:crit:point:top}
\big(\overline{u}_{\mathrm{top}}(a,d) , \overline{v}_{\mathrm{top}}(a,d) \big) \le
\big (u_{i_0 + 1} , v_{i_0 + 1} \big) \le
  \big(\overline{u}_B(a,d) , \overline{v}_B(a,d) \big),
\end{equation}
while
\begin{equation}
(0, 0) \le \big(u_{i_0} , v_{i_0} \big) \le \big(\overline{u}_B(a,d) , \overline{v}_{\mathrm{bot}}(a,d) \big).
\end{equation}
\end{lem}
\begin{proof}
%
For any $i \in \mathbb{Z}$ we have the inequalities
\begin{equation}
u_i \le u_{i+1} \le \overline{u}_B(a,d),
\end{equation}
which implies that we must have
$u_i \le u_{a,d}(v_i) \le u_{i+1}$.
In particular, we see that
\begin{equation}
 u_{a,d}(v_i) \le \overline{u}_B(a,d),
\end{equation}
which implies that
\begin{equation}
v_i \notin \Big( \overline{v}_{\mathrm{bot}}(a,d), \overline{v}_{\mathrm{top}}(a,d) \Big).
\end{equation}
In addition, we have
\begin{equation}
u_i = 2 u_{a,d}(v_i) - u_{i+1} \ge 2 u_{a,d}(v_i) - \overline{u}_B(a,d).
\end{equation}
In particular, for every $i$
we either have
\begin{equation}
\label{eq:twv:conn:option:a}
(0,0) \le (u_i, v_i) \le \big(\overline{u}_B(a,d), \overline{v}_{\mathrm{bot}}(a,d) \big)
\end{equation}
or
\begin{equation}
\label{eq:twv:conn:option:b}
\big(\overline{u}_{\mathrm{top}}(a,d), \overline{v}_{\mathrm{top}}(a,d) \big)
\le (u_i, v_i) \le
\big(\overline{u}_B(a,d), \overline{v}_{B}(a,d) \big) .
\end{equation}
The limits \sref{eq:twv:limits:stnd:wave} imply
that there exists $M \gg 1$ so that
\sref{eq:twv:conn:option:a}
holds for all $i \le - M$
and \sref{eq:twv:conn:option:b} holds
for all $i \ge M$. In particular, there
must be jump between these two sets.
\end{proof}

For any $(a,d)\in \overline{\Omega}_-$
with $d > 0$, the inequalities
\sref{eq:twv:def:ineqs:v:bot:top}
imply that the function
$u_{a,d}$ is strictly increasing
on $[0, \overline{v}_{\mathrm{bot}}(a,d)]$.
This allows us to define an inverse
\begin{equation}
u_{a,d}^{\mathrm{inv}}: [0, \overline{u}_B(a,d)]
 \to [0, \overline{v}_{\mathrm{bot}}(a,d)].
\end{equation}
In addition, we introduce the
function
\begin{equation}
\label{eq:twv:def:v:d:refl}
v_{a,d;r}(u) = 2 v_{a,d}(u) - \overline{v}_B(a,d),
\end{equation}
which can be interpreted
as the vertical reflection of the line $v = \overline{v}_B(a,d)$ through the curve
$v = v_{a,d}(u)$.

For any $(a,d)\in \overline{\Omega}_-$
with $d > 0$, these definitions
allow us to introduce the notation
\begin{equation}
\overline{u}_{\mathrm{bot}}(a,d) = \max\{
 u \in [0, \overline{u}_B(a,d) ] :
   v_{a,d;r}(u) = u_{a,d}^{\mathrm{inv}}(u)
\} .
\end{equation}
On account of the inequalities
\begin{equation}
v_{a,d;r}(0) < 0 = u_{a,d}^{\mathrm{inv}}(0)
  < u_{a,d}^{\mathrm{inv}}\big(\overline{u}_B(a,d)\big) = \overline{v}_{\mathrm{bot}}(a,d)
   < \overline{v}_B(a,d) = v_{a,d;r}\big( \overline{u}_B(a,d) \big),
\end{equation}
one can verify that $\overline{u}_{\mathrm{bot}}$ is well-defined and continous in $(a,d)$.

\begin{lem}
\label{lem:twv:bot:lt:top:trv}
Pick $(a,d) \in \Omega_-$
and assume that
\begin{equation}
\label{eq:twv:bot:sm:top}
\overline{u}_{\mathrm{bot}}(a,d) < \overline{u}_{\mathrm{top}}(a,d)
\end{equation}
holds. Then we have $c(a,d) > 0$.
\end{lem}
\begin{proof}
Suppose the contrary that $c(a,d) =0$
and consider the sequence $\{(u_i, v_i)\}$ defined
by \sref{eq:twv:def:standing:seq}, together with the
critical value $i_0 \in \mathbb{Z}$
that appears in Lemma \ref{lem:twv:stn:wave:crit:point}.
Since the sequence $\{u_i\}$ is non-decreasing,
we have $u_{i_0+1} \ge u_{i_0}$
which implies that
\begin{equation}
u_{i_0 +1} \ge u_{a,d}(v_{i_0}).
\end{equation}
Exploiting $0 \le v_{i_0} \le \overline{v}_{\mathrm{bot}}(a,d)$ this gives
\begin{equation}
\label{eq:lem:stn:non:exis:up:bnd:v}
0 \le v_{i_0} \le u_{a,d}^{\mathrm{inv}}\big( u_{i_0 + 1} \big).
\end{equation}
On the other hand, the inequality
$v_{i_0 + 1} \le \overline{v}_B(a,d)$
yields
\begin{equation}
v_{i_0} \ge v_{a,d;r}(u_{i_0 + 1} ).
\end{equation}
Combining this with \sref{eq:lem:stn:non:exis:up:bnd:v},
we see that
$u_{i_0 + 1} \le \overline{u}_{\mathrm{bot}}(a,d)$. 
However,
\sref{eq:lem:crit:point:top} implies that $u_{i_0 + 1} \ge \overline{u}_{\mathrm{top}}(a,d)$,
which yields the desired contradiction.
\end{proof}

\begin{figure}
\centering
\begin{minipage}{0.6\textwidth}
\centering
\includegraphics[width=\textwidth]{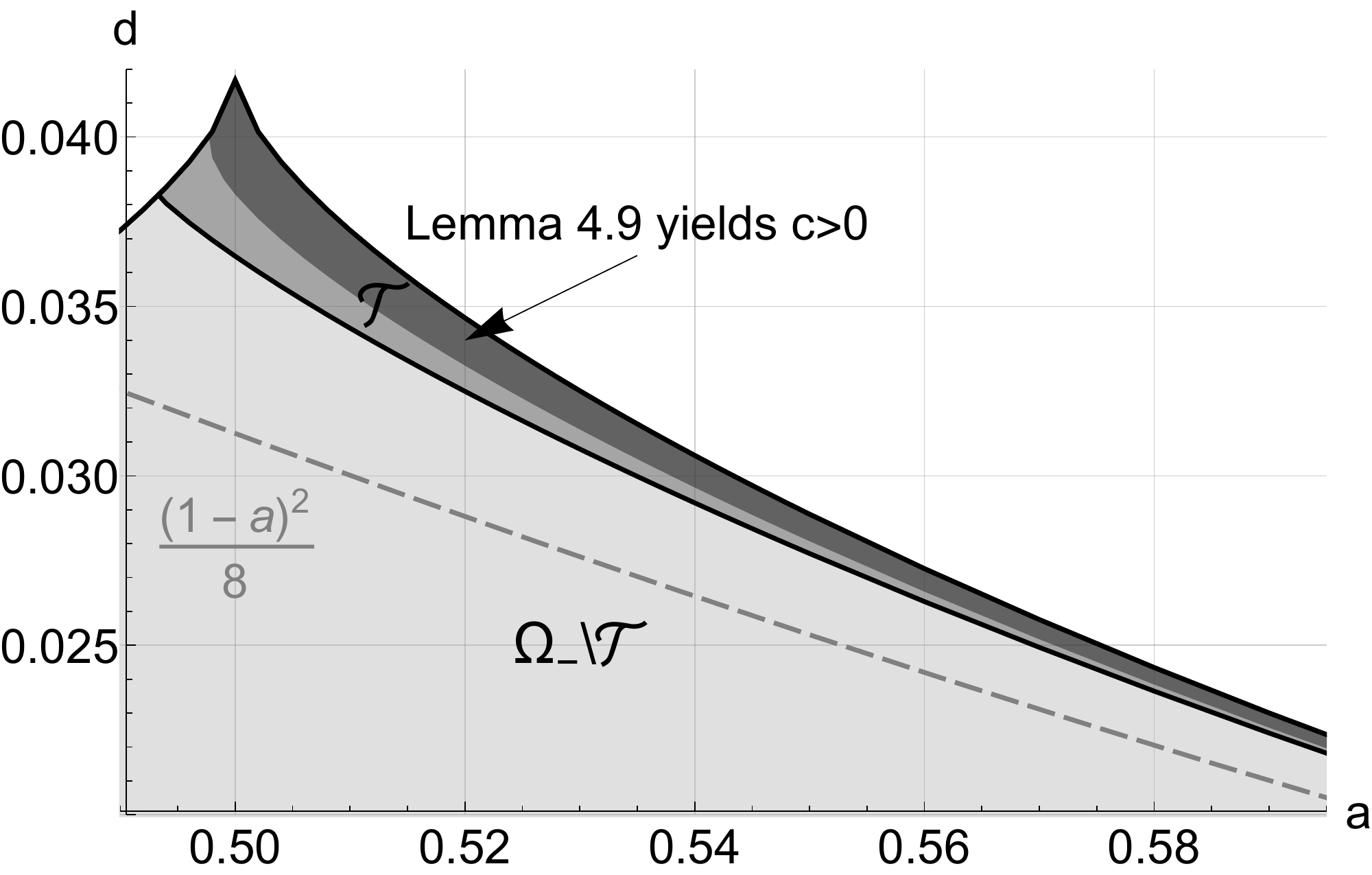}
\end{minipage}
\caption{
The darkest region contains all pairs $(a,d)$ where
the assumption $\overline{u}_{\mathrm{bot}}(a,d) < \overline{u}_{\mathrm{top}}(a,d)$
was verified numerically. We also include
the boundaries of the sets $\mathcal{T}$ and $\Omega_-$ as computed
numerically by the procedure described in \S\ref{sec:mr}.
}\label{f:two:reflection:works}
\end{figure}

\subsection{Verification of $\overline{u}_{\mathrm{top}} > \overline{u}_{\mathrm{bot}}$}
\label{sec:twv:expl}

In Figure \ref{f:two:reflection:works} we show where one may
numerically verify  that the scalar inequality \sref{eq:twv:bot:sm:top} holds,
which ensures that $c(a,d) > 0$. We also plot the curve $d = \frac{(1 - a)^2}{8}$,
below which we have established that $c(a,d) = 0$. Taken together, we feel that
these results cover a reasonable portion of the parameter space $\Omega_{-}$.

Our final task is to analytically verify (without resorting to any numerics) that $c(a,d) > 0$
near the cusp $(a,d) = \big(\frac{1}{2} , \frac{1}{24} \big)$
and the corner $(a,d) = (1, 0)$. As a preparation,
we construct a simplified but weaker version of the condition
$\overline{u}_{\mathrm{top}} > \overline{u}_{\mathrm{bot}}$ by exploiting
the monotonicity of $v_{a,d;r}$.

\begin{lem}
Pick $(a,d) \in \overline{\Omega}_-$ with $d > 0$.
Then the function $v_{a,d;r}$ is strictly increasing
on $[0, \overline{u}_B(a,d)]$.
\end{lem}
\begin{proof}
We note first that
$v_{a,d}'(u) = 1 - \frac{g'(u ;a)}{2d} > 0$
for $u \in [0, u_{\mathrm{min}}(a)]$.
In particular, we only have to consider the case
$\overline{u}_B(a,d) > u_{\mathrm{min}}(a)$,
which cannot occur for $a = \frac{1}{2}$.

Let us first assume that $0 < a < \frac{1}{2}$.
Using \sref{eq:bif:str:ineqs:star:vs:plus}
or \sref{eq:lem:bif:str:roots:bc:tangency}
we may conclude that
\begin{equation}
v_{a,d}'\big( \overline{u}_B(a,d) \big) \ge
  v_+'\big(\overline{u}_B(a,d) \big) > 0 .
\end{equation}
Since $\overline{u}_B(a,d ) \le a \le \gamma_{c;+}(a,d)$
this implies that
$\overline{u}_B(a,d) \le \gamma_{c;-}(a,d)$.
In particular,
$v_{a,d}$ and hence $v_{a,d;r}$
are strictly increasing
on $[0, \overline{u}_B(a,d)]$.

It remains to consider $a \in (\frac{1}{2} , 1)$.
Since $\overline{v}_B(1 - a, d) > u_{\mathrm{max}}(1 - a)$,
we may use Corollary \ref{cor:bif:1:min:a}
to conclude
\begin{equation}
\overline{u}_B( a, d)  = 1 - \overline{v}_B(1 -a,d)
 < 1 - u_{\mathrm{max}}(1 - a)
= u_{\mathrm{min}}(a).
\end{equation}
\end{proof}

\begin{cor}
\label{cor:twv:crit:for:no:trv:wave}
Consider any $(a,d) \in \overline{\Omega}_-$ with $d > 0$
and suppose that
\begin{equation}
\label{eq:twv:simpl:crit}
v_{a,d;r}\big( \overline{u}_{\mathrm{top}}(a,d) \big) > \overline{v}_{\mathrm{bot}}(a,d).
\end{equation}
Then we have $\overline{u}_{\mathrm{top}}(a,d) > \overline{u}_{\mathrm{bot}}(a,d)$.
\end{cor}
\begin{proof}
This follows from the
uniform bound $u_{a,d}^{\mathrm{inv}} \le \overline{v}_{\mathrm{bot}}(a, d)$
and the fact
that $v_{a,d;r}$ is strictly increasing.
\end{proof}

We now set out to verify the explicit condition
\sref{eq:twv:simpl:crit} for the boundary points $\big(a, d_-(a)\big)$
with $a \sim 1$ and $a = \frac{1}{2}$. Using the continuity of
$\overline{u}_{\mathrm{top}}$ and $\overline{u}_{\mathrm{bot}}$, this means that
$c(a,d) > 0$ for
all $(a,d) \in \Omega_-$ that are sufficiently close to these critical boundary points.

\begin{lem}
We have the expansions
\begin{equation}
\label{eq:twv:ex:uv:b}
\begin{array}{lcl}
\overline{u}_B\big(a , d_-(a) \big)
 & = & \frac{1}{4} ( 1 - a)^2 + \frac{1}{8} (1 - a)^3
  + O \big( ( 1 - a)^4 \big) ,
\\[0.2cm]
\overline{v}_B\big(a , d_-(a) \big)
 & = & 1 -\frac{1}{2} ( 1 - a)
   + O \big( ( 1 - a)^4 \big)
\\[0.2cm]
\end{array}
\end{equation}
as $a \uparrow 1$.
\end{lem}
\begin{proof}
Exploiting
Corollary \ref{cor:bif:1:min:a}
together with the
symmetry $d(a) = d(1 - a)$,
we obtain
\begin{equation}
\overline{u}_{B}\big(a, d_-(a)\big)
= 1- \overline{v}_B\big(1 - a, d_-(1-a) \big),
\qquad
\overline{v}_{B}(a , d_-(a))
 = 1 - \overline{u}_B\big(1 - a , d_-(1-a) \big).
\end{equation}
The desired expansions
hence follow from Proposition
\ref{prp:bif:exp:a:zero}.
\end{proof}

\begin{lem}
We have the expansions
\begin{equation}
\begin{array}{lcl}
\overline{u}_{\mathrm{top}} \big( a, d_-(a) \big)
 & = & \frac{1}{4} ( a - 1)^2 + O \big( (1 - a)^3 \big) ,
\\[0.2cm]
v_{a,d_-(a); r}\Big( \overline{u}_{\mathrm{top}}\big(a, d_-(a) \big)  \Big)
 & = & 1 + O \big( 1 - a \big)
\end{array}
\end{equation}
as $ a \uparrow 1$.
\end{lem}
\begin{proof}
These expansions can be found by
substition of \sref{eq:twv:ex:uv:b}
into the definitions \sref{eq:twv:def:u:top}
and \sref{eq:twv:def:v:d:refl}.
\end{proof}

On account of
the identity
$\gamma_{c;-}(1, 0) = \frac{1}{3}$
and the inequality
$\overline{v}_{\mathrm{bot}}(a,d) \le \gamma_{c;-}(a,d)$,
we see that Corollary \ref{cor:twv:crit:for:no:trv:wave}
implies that
\begin{equation}
\label{eq:twv:conclu:top:bot:a:one}
\overline{u}_{\mathrm{top}}\big(a,d_-(a) \big) > \overline{u}_{\mathrm{bot}}\big(a,d_-(a) \big)
\end{equation}
whenever $1 - a > 0$ is sufficiently small.

\begin{lem}
\label{lem:twv:cusp:ineq}
The inequality \sref{eq:twv:simpl:crit}
holds for $(a, d) = (\frac{1}{2} , \frac{1}{24})$.
\end{lem}
\begin{proof}
Writing $(a_{\mathrm{cp}}, d_{\mathrm{cp}}) = (\frac{1}{2} , \frac{1}{24})$,
we can explicitly compute
\begin{equation}
\gamma_{c;+}(a_{\mathrm{cp}},d_{\mathrm{cp}}) = \frac{1}{2} + \frac{1}{6} \sqrt{2} ,
\end{equation}
which together with
the expressions
\sref{eq:bif:ids:for:cusp} and \sref{eq:twv:def:u:top}
yields
\begin{equation}
\overline{u}_{\mathrm{top}}(a_{\mathrm{cp}},d_{\mathrm{cp}}) = \frac{1}{2}-\frac{4}{9}\sqrt{2}
  +\frac{1}{6}\sqrt{3} .
\end{equation}
Using \sref{eq:twv:def:v:d:refl} we obtain
\begin{equation}
v_{a_{\mathrm{cp}},d_{\mathrm{cp}};r}\big(\overline{u}_{\mathrm{top}}(a_{\mathrm{cp}},d_{\mathrm{cp}}) \big)
= \frac{1}{2}
  -\frac{1240}{243}\sqrt{2}
   +\frac{229}{54}\sqrt{3}
\sim 0.6286.
\end{equation}
In particular, we have
\begin{equation}
v_{a_{\mathrm{cp}},d_{\mathrm{cp}};r}
  \big(\overline{u}_{\mathrm{top}}(a_{\mathrm{cp}},d_{\mathrm{cp}}) \big)  > a_{\mathrm{cp}} \ge
 \gamma_{c;-}(a_{\mathrm{cp}},d_{\mathrm{cp}}) \ge
 \overline{v}_{\mathrm{bot}}(a_{\mathrm{cp}},d_{\mathrm{cp}}) ,
\end{equation}
as desired.
\end{proof}

\begin{proof}[Proof of Theorem \ref{thm:mr:twv:sets:t}]
Item (i), (ii) and (iii)
follow from Lemma \ref{lem:twv:d:star},
 Corollary \ref{cor:twv:pos:zone}
and Lemma \ref{lem:twv:cusp:ineq} respectively.
Item (iv) follows from Lemma \ref{lem:twv:vcrit:ordering},
together with Lemma \ref{lem:twv:bot:lt:top:trv}
and the continuity of the functions $\overline{u}_{\mathrm{bot}}$
and $\overline{u}_{\mathrm{top}}$. Indeed, the expression
\sref{eq:mr:expr:for:Gamma} can be rewritten as
\begin{equation}
\Gamma(a) =
  \overline{u}_{\mathrm{top}}\big(a, d_-(a) \big)
  - \overline{u}_{\mathrm{bot}}\big(a , d_-(a) \big).
\end{equation}
Finally, item (v) follows directly from \sref{eq:twv:conclu:top:bot:a:one}.
\end{proof}

\bibliographystyle{klunumHJ}
\bibliography{ref}

\end{document}